\documentclass[letterpaper]{amsart}

\usepackage[letterpaper]{geometry}
\usepackage{amsmath, amsthm, amsfonts, amsbsy, thmtools, amssymb}

\usepackage{thm-restate}
\usepackage{hyperref} 
\usepackage[]{ytableau}
\usepackage[mathscr]{euscript}
\usepackage{stmaryrd}
\usepackage[all, arc, curve, frame]{xy} 
\usepackage[T1]{fontenc}	
\usepackage{tikz}
\usetikzlibrary{arrows}

\usepackage{hyperref}
\usepackage{cleveref}

\makeatletter
\def\namedlabel#1#2{\begingroup
	#2%
	\def\@currentlabel{#2}%
	\phantomsection\label{#1}\endgroup
}
\makeatother

\numberwithin{equation}{section}

\SetSymbolFont{stmry}{bold}{U}{stmry}{m}{n}

\DeclareMathOperator{\Adm}{Adm}

\DeclareMathOperator{\b|}{\boldsymbol{|}}

\newtheorem{thm}{Theorem}[section]
\newtheorem{prop}[thm]{Proposition}
\newtheorem{lem}[thm]{Lemma}
\newtheorem{cor}[thm]{Corollary}

\theoremstyle{remark}
\newtheorem{rem}[thm]{Remark}
\theoremstyle{definition}
\newtheorem{definition}[thm]{Definition}

\title{Stochasticization of Solutions to the Yang-Baxter Equation}

\author{Amol Aggarwal, Alexei Borodin, and Alexey Bufetov}

\begin{document}

\begin{abstract}
	
	In this paper we introduce a procedure that, given a solution to the Yang-Baxter equation as input, produces a stochastic (or Markovian) solution to (a possibly dynamical version of) the Yang-Baxter equation. We then apply this ``stochasticization procedure'' to obtain three new, stochastic solutions to several different forms of the Yang-Baxter equation. The first is a stochastic, elliptic solution to the dynamical Yang-Baxter equation; the second is a stochastic, higher rank solution to the dynamical Yang-Baxter equation; and the third is a stochastic solution to a dynamical variant of the tetrahedron equation. 
\end{abstract}

\maketitle

\tableofcontents

\section{Introduction}

\label{Introduction}

The search for exactly solvable (or \emph{integrable}) systems has long played a prominent role in mathematics and physics. In recent years, a significant amount of effort has been devoted towards the more refined search for models that are not only integrable but also Markovian (or \emph{stochastic}). These include certain classes of interacting particles systems and directed polymers in random media, in which the asymmetric simple exclusion process (ASEP) and Kardar--Parisi--Zhang (KPZ) stochastic partial differential equation serve as representative examples, respectively.  

The first wave of integrable Markovian models started in the late 1990s with the papers of Johansson \cite{Johansson} and Baik--Deift--Johansson \cite{BaikDJ}, and the key to their integrability was in reductions to determinantal and Pfaffian random point processes. The second wave of integrable Markovian models started the in late 2000s, and their integrability mechanisms heavily relied on those that had been previously developed for non-Markovian integrable models of statistical and quantum mechanics. For example, looking at the earlier papers of the second wave we see that: (a) The pioneering work of Tracy--Widom \cite{TracyW1,TracyW2,TracyW3} on the ASEP was based on the famous idea of Bethe \cite{Bethe} of looking for eigenfunctions of a quantum many-body system in the form of superpositions of those for noninteracting bodies (coordinate Bethe ansatz); (b) The work of O'Connell \cite{O'Connell} and Borodin--Corwin \cite{MP} on semi-discrete Brownian polymers utilized properties of eigenfunctions of the Macdonald--Ruijsenaars quantum integrable system -- the celebrated Macdonald polynomials and their degenerations; (c) The physics papers of Dotsenko \cite{Dotsenko} and Calabrese--Le Doussal--Rosso \cite{CalabreseDR}, and a later work of Borodin--Corwin--Sasamoto \cite{BorodinCS}, used a duality trick to show that certain observables of infinite-dimensional models solve finite-dimensional quantum many-body systems that are in turn solvable by the coordinate Bethe ansatz.

It was later realized that all of the above examples, as well as many others, can be naturally united under a single umbrella -- integrable stochastic vertex models. 
The first such unification was suggested by Corwin--Petrov \cite{SHSVML} on the basis of Borodin \cite{FSRF} under the name of the \emph{stochastic higher spin six vertex model}; see Borodin--Petrov \cite{BorodinP2} for a lecture style exposition. Its existence was due to the fact that all of these models were governed by the same algebraic structure -- the quantum affine group $U_q(\widehat{\mathfrak{sl}_2})$. This was later extended to the level of the elliptic quantum group $E_{\tau,\eta}(\mathfrak{sl}_2)$ in Borodin \cite{ERSF} and Aggarwal \cite{DSHSVM}, which produced \emph{dynamical} stochastic vertex models. 
Concerning quantum groups of higher rank, stochastic vertex models corresponding to those of type A have been introduced by Kuniba--Mangazeev--Maruyama--Okado in \cite{SM}. In a certain degeneration, these models reproduce \emph{multi-species} exclusion processes that have been around since at least the 1990s. A dynamic extension was given by Kuniba in \cite{FW}.  


However, a question that largely remained without answer was whether one could provide a systematic way to search for and produce new examples of stochastic integrable systems. 

In the non-stochastic setting, exactly solvable systems are known to arise from solutions to the \emph{Yang-Baxter equation} (see the book by Baxter \cite{ESMSM} and the collection of foundational papers \cite{EIS}), which is a relation of the form 
\begin{flalign}
\label{rrr}
\displaystyle\sum_{i', j', k'} R_{ij}^{i' j'} R_{kj'}^{k'j''} R_{k'i'}^{k''i''}  = \displaystyle\sum_{i', j', k'} R_{ki}^{k'i'} R_{k'j}^{k''j'} R_{i'j'}^{i''j''},
\end{flalign}

\noindent where $\textbf{R} = \big[ R_{ij}^{i' j'} \big]$ is an \emph{$R$-matrix} (which might depend on additional parameters) and the sum is over all triples $(i', j', k')$, with $(i, j, k)$ and $(i'', j'', k'')$ fixed. 

Solutions to the Yang-Baxter equation \eqref{rrr} are in turn known (see, e.g., Jimbo-Miwa \cite{AASLM}) to come from the representation theory of (affine or elliptic) quantum groups; the latter topic is intensely studied with many examples, thereby giving rise to many examples of integrable systems. However, stochasticity of a system arising in this way is equivalent to the condition that $\sum_{i', j'} R_{ij}^{i' j'} = 1$ for each fixed $(i, j)$, and this relation is typically not satisfied.  

Our goal in this paper is to in a sense rectify this issue by explaining how a general solution $\textbf{R} = \big[ R_{ij}^{i' j'} \big]$ of the Yang-Baxter equation \eqref{rrr} can be ``stochasticized'' to form a matrix $\textbf{S} = \big[ S_{ij}^{i' j'} \big]$ that is stochastic and still satisfies a (possibly dynamical variant) of the Yang-Baxter equation \eqref{rrr}.

The above-mentioned works \cite{SHSVML, FSRF, ERSF, DSHSVM, SM, FW} (see also Kuan \cite{CDFSVMD}) did use certain conjugations to produce stochastic integrable models out of non-stochastic ones. These conjugations can be viewed as special cases of the general stochasticization procedure that we describe in this work. In fact, we will see that our procedure is much more universal and applies in a substantially broader context; in particular, it also allows one to see stochastic systems in situations where previous attempts failed to (for example, in the elliptic setting). 

It is natural to view our stochasticization procedure as a particular case of the broader concept of \emph{bijectivisation} of the Yang-Baxter equation introduced by Bufetov--Petrov \cite{FSSF}; a basic case ($\mathfrak{sl}_2$, spin $\frac 12$) of one of our examples actually appeared in Section 7 of that work and was our starting point. 

Curiously, the graphical representation of our stochasticization procedure is somewhat similar to the ``shadows world'' of Kirillov-Reshetikhin in \cite{KR}. Another similarity is that ``the shadows world'' of \cite{KR} is also closely related to solutions to the dynamical Yang-Baxter equation and solid-on-solid (SOS) models. However, this is where similarities end --- our definitions are different, as are the resulting solutions.   

In \Cref{EllipticEquation}, \Cref{HigherRankEquation}, and \Cref{TetrahedronStochastic} below we describe three new examples of stochastic solutions to (variants of) the Yang-Baxter equation obtained by applying our stochasticization procedure in different settings. Specifically, in \Cref{EllipticEquation}, we state an elliptic, stochastic solution to the dynamical Yang-Baxter equation; in \Cref{HigherRankEquation}, we state a dynamical variant of a stochastic higher rank vertex model; and, in \Cref{TetrahedronStochastic}, we state a stochastic solution to a dynamical version of the tetrahedron equation. We then provide an outline for the remainder of the paper in \Cref{Outline}.

\subsection{A Stochastic Elliptic Solution}

\label{EllipticEquation}

Elliptic solutions to the dynamical Yang-Baxter equation were introduced by Baxter 45 years ago through his eight-vertex SOS model as a means to solve the eight-vertex model \cite{EVMLSF,EVMLSE,EVMLSETM}. Since then, elliptic integrable systems have been studied extensively in the statistical physics literature; see, for instance, \cite{EVMGI,ESMSM,ESSOSMLHP,ESSOSM,FEVSOSM,EQG,AEQG,REQG,ESEMHF,EIS,TFM,AASLM} and references therein. However, none of these elliptic integrable systems happened to be stochastic. 

Based on the framework of elliptic quantum groups \cite{EQG,AEQG,REQG}, a family of stochastic dynamical vertex, or interaction round-a-face (IRF), models were proposed in \cite{ERSF}, which were later generalized in \cite{DSHSVM,FW} and analyzed from a probabilistic perspective in \cite{DCP}. However, in order to ensure stochasticity of these models, the works \cite{DSHSVM,ERSF,FW} took trigonometric degenerations of the originally elliptic solutions to the dynamical Yang-Baxter equation. Thus, the ellipticity of the stochastic vertex weights was lost. 

Here, we introduce a stochastic integrable system whose weights are elliptic. We begin with the following definition, which provides the stochastic weights.

\begin{definition}
	
	\label{1jsdefinition}
	
	Let $\tau, \eta, \lambda, x, y, v \in \mathbb{C}$ with $\Im \tau > 0$. Letting $f(z)$ denote the first Jacobi theta function (dependent on $\tau$) as in \eqref{theta1}, define
	\begin{flalign*}
	& S_{1; 1}^{\text{ell}} \big(1, 0; 1, 0 \b| \lambda, v; x, y \big)  = \displaystyle\frac{f(y - x) f(\lambda - 2 \eta)}{f(y - x + 2 \eta) f(\lambda)} \displaystyle\frac{f \big( \lambda + y - v + 2 \eta \big) f \big( x - v \big)}{  f \big( x - v + 2 \eta \big) f \big( \lambda + y - v \big)}; \\
	& S_{1; 1}^{\text{ell}} \big(0, 1; 1, 0 \b| \lambda, v; x, y \big)  = \displaystyle\frac{f(\lambda - y + x) f(2 \eta)}{f(y - x + 2 \eta) f(\lambda)} \displaystyle\frac{f \big( \lambda + y - v + 2 \eta \big) f \big( x - v \big)}{  f \big(\lambda + x - v + 2 \eta \big) f \big(y - v \big)}; \\
	& S_{1; 1}^{\text{ell}} \big(1, 0; 0, 1 \b| \lambda, v; x, y \big)  =  \displaystyle\frac{f(\lambda + y - x) f(2 \eta)}{f(y - x + 2 \eta) f(\lambda)} \displaystyle\frac{f \big( y -v + 2 \eta \big) f \big( \lambda + x - v \big)}{  f \big(x - v + 2 \eta \big) f \big( \lambda + y - v \big)}; \\
	& S_{1; 1}^{\text{ell}} \big(0, 1; 0, 1 \b| \lambda, v; x, y \big)  = \displaystyle\frac{f(y - x) f(\lambda + 2 \eta)}{f(y - x + 2 \eta) f(\lambda)} \displaystyle\frac{f \big( \lambda + x - v \big) f \big( y - v + 2 \eta \big)}{  f \big(\lambda + x - v + 2 \eta \big) f \big(y - v \big)}; \\
	& \qquad \qquad S_{1; 1}^{\text{ell}} \big(0, 0; 0, 0 \b| \lambda; x, y; v \big) = 1 = S_{1; 1}^{\text{ell}} \big( 1, 1; 1, 1 \b| \lambda; x, y; v \big), 
	\end{flalign*}
	
	\noindent and $S_{1; 1}^{\text{ell}} \big( i_1, j_1; i_2, j_2 \b| \lambda, v; x, y \big) = 0$ for all $(i_1, j_1; i_2, j_2)$ not of the above form. 
	
\end{definition}

To provide a diagrammatic interpretation of the quantities $S_{1; 1}^{\text{ell}} (i_1, j_1; i_2, j_2 \b| \lambda, v; x, y)$, let $\mathcal{D}$ be a finite subset of the graph $\mathbb{Z}_{\ge 0}^2$. A \emph{directed path ensemble} on $\mathcal{D}$ is a family of paths connecting adjacent vertices of $\mathcal{D}$, each edge of which is either directed up or to the right; see the left side of Figure \ref{vertexedgespaths} for an example.

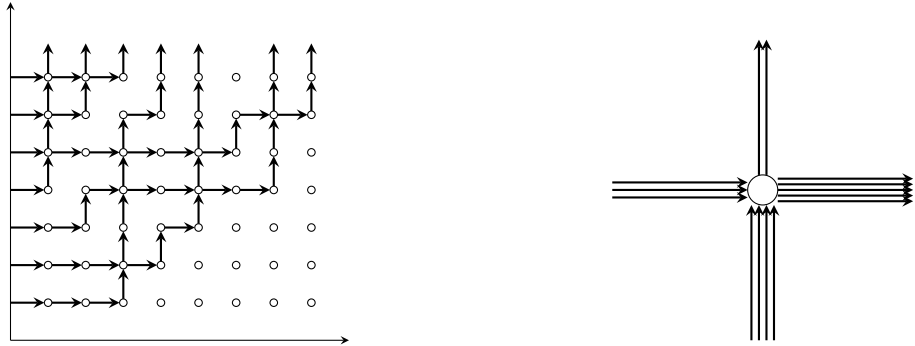
\begin{figure}[t]
	
	\begin{center}
		
		\begin{tikzpicture}[
		>=stealth,
		]

		\draw[->, black	] (0, 0) -- (0, 4.5);
		\draw[->, black] (0, 0) -- (4.5, 0);
		\draw[->,black, thick] (0, .5) -- (.45, .5);
		\draw[->,black, thick] (0, 1) -- (.45, 1);
		\draw[->,black, thick] (0, 1.5) -- (.45, 1.5);
		\draw[->,black, thick] (0, 2) -- (.45, 2);
		\draw[->,black, thick] (0, 2.5) -- (.45, 2.5);
		\draw[->,black, thick] (0, 3) -- (.45, 3);
		\draw[->,black, thick] (0, 3.5) -- (.45, 3.5);

		\draw[->,black, thick] (.55, .5) -- (.95, .5);
		\draw[->,black, thick] (.55, 1) -- (.95, 1);
		\draw[->,black, thick] (.55, 1.5) -- (.95, 1.5);
		\draw[->,black, thick] (.55, 2.5) -- (.95, 2.5);
		\draw[->,black, thick] (.55, 3) -- (.95, 3);
		\draw[->,black, thick] (.55, 3.5) -- (.95, 3.5);
		\draw[->,black, thick] (.5, 2.05) -- (.5, 2.45);
		\draw[->,black, thick] (.5, 2.55) -- (.5, 2.95);
		\draw[->,black, thick] (.5, 3.05) -- (.5, 3.45);
		\draw[->,black, thick] (.5, 3.55) -- (.5, 3.95);
		
		\draw[->,black, thick] (1.05, .5) -- (1.45, .5);
		\draw[->,black, thick] (1.05, 1) -- (1.45, 1);
		\draw[->,black, thick] (1.05, 2) -- (1.45, 2);
		\draw[->,black, thick] (1.05, 2.5) -- (1.45, 2.5);
		\draw[->,black, thick] (1.05, 3.5) -- (1.45, 3.5);
		\draw[->,black, thick] (1, 1.55) -- (1, 1.95);
		\draw[->,black, thick] (1, 3.05) -- (1, 3.45);
		\draw[->,black, thick] (1, 3.55) -- (1, 3.95);
		
		\draw[->,black, thick] (1.5, .55) -- (1.5, .95);
		\draw[->,black, thick] (1.5, 1.05) -- (1.5, 1.45);
		\draw[->,black, thick] (1.5, 1.55) -- (1.5, 1.95);
		\draw[->,black, thick] (1.5, 2.05) -- (1.5, 2.45);
		\draw[->,black, thick] (1.5, 2.55) -- (1.5, 2.95);
		\draw[->,black, thick] (1.5, 3.55) -- (1.5, 3.95);
		\draw[->,black, thick] (1.55, 1) -- (1.95, 1);
		\draw[->,black, thick] (1.55, 2) -- (1.95, 2);
		\draw[->,black, thick] (1.55, 2.5) -- (1.95, 2.5);
		\draw[->,black, thick] (1.55, 3) -- (1.95, 3);

		\draw[->,black, thick] (2, 1.05) -- (2, 1.45);
		\draw[->,black, thick] (2, 3.05) -- (2, 3.45);
		\draw[->,black, thick] (2, 3.55) -- (2, 3.95);
		\draw[->,black, thick] (2.05, 1.5) -- (2.45, 1.5);
		\draw[->,black, thick] (2.05, 2) -- (2.45, 2);
		\draw[->,black, thick] (2.05, 2.5) -- (2.45, 2.5);

		\draw[->,black, thick] (2.5, 1.55) -- (2.5, 1.95);
		\draw[->,black, thick] (2.5, 2.05) -- (2.5, 2.45);
		\draw[->,black, thick] (2.5, 2.55) -- (2.5, 2.95);
		\draw[->,black, thick] (2.5, 3.05) -- (2.5, 3.45);
		\draw[->,black, thick] (2.5, 3.55) -- (2.5, 3.95);
		\draw[->,black, thick] (2.55, 2) -- (2.95, 2);
		\draw[->,black, thick] (2.55, 2.5) -- (2.95, 2.5);

		\draw[->,black, thick] (3, 2.55) -- (3, 2.95);
		\draw[->,black, thick] (3.05, 2) -- (3.45, 2);
		\draw[->,black, thick] (3.05, 3) -- (3.45, 3);

		\draw[->,black, thick] (3.5, 2.05) -- (3.5, 2.45);
		\draw[->,black, thick] (3.5, 2.55) -- (3.5, 2.95);
		\draw[->,black, thick] (3.5, 3.05) -- (3.5, 3.45);
		\draw[->,black, thick] (3.5, 3.55) -- (3.5, 3.95);
		\draw[->,black, thick] (3.55, 3) -- (3.95, 3);
		
		\draw[->,black, thick] (4, 3.05) -- (4, 3.45);
		\draw[->,black, thick] (4, 3.55) -- (4, 3.95);

		\filldraw[fill=white, draw=black] (.5, .5) circle [radius=.05];
		\filldraw[fill=white, draw=black] (.5, 1) circle [radius=.05];
		\filldraw[fill=white, draw=black] (.5, 1.5) circle [radius=.05];
		\filldraw[fill=white, draw=black] (.5, 2) circle [radius=.05];
		\filldraw[fill=white, draw=black] (.5, 2.5) circle [radius=.05];
		\filldraw[fill=white, draw=black] (.5, 3) circle [radius=.05];
		\filldraw[fill=white, draw=black] (.5, 3.5) circle [radius=.05];

		\filldraw[fill=white, draw=black] (1, .5) circle [radius=.05];
		\filldraw[fill=white, draw=black] (1, 1) circle [radius=.05];
		\filldraw[fill=white, draw=black] (1, 1.5) circle [radius=.05];
		\filldraw[fill=white, draw=black] (1, 2) circle [radius=.05];
		\filldraw[fill=white, draw=black] (1, 2.5) circle [radius=.05];
		\filldraw[fill=white, draw=black] (1, 3) circle [radius=.05];
		\filldraw[fill=white, draw=black] (1, 3.5) circle [radius=.05];
		
		\filldraw[fill=white, draw=black] (1.5, .5) circle [radius=.05];
		\filldraw[fill=white, draw=black] (1.5, 1) circle [radius=.05];
		\filldraw[fill=white, draw=black] (1.5, 1.5) circle [radius=.05];
		\filldraw[fill=white, draw=black] (1.5, 2) circle [radius=.05];
		\filldraw[fill=white, draw=black] (1.5, 2.5) circle [radius=.05];
		\filldraw[fill=white, draw=black] (1.5, 3) circle [radius=.05];
		\filldraw[fill=white, draw=black] (1.5, 3.5) circle [radius=.05];
		
		\filldraw[fill=white, draw=black] (2, .5) circle [radius=.05];
		\filldraw[fill=white, draw=black] (2, 1) circle [radius=.05];
		\filldraw[fill=white, draw=black] (2, 1.5) circle [radius=.05];
		\filldraw[fill=white, draw=black] (2, 2) circle [radius=.05];
		\filldraw[fill=white, draw=black] (2, 2.5) circle [radius=.05];
		\filldraw[fill=white, draw=black] (2, 3) circle [radius=.05];
		\filldraw[fill=white, draw=black] (2, 3.5) circle [radius=.05];
		
		\filldraw[fill=white, draw=black] (2.5, .5) circle [radius=.05];
		\filldraw[fill=white, draw=black] (2.5, 1) circle [radius=.05];
		\filldraw[fill=white, draw=black] (2.5, 1.5) circle [radius=.05];
		\filldraw[fill=white, draw=black] (2.5, 2) circle [radius=.05];
		\filldraw[fill=white, draw=black] (2.5, 2.5) circle [radius=.05];
		\filldraw[fill=white, draw=black] (2.5, 3) circle [radius=.05];
		\filldraw[fill=white, draw=black] (2.5, 3.5) circle [radius=.05];
		
		\filldraw[fill=white, draw=black] (3, .5) circle [radius=.05];
		\filldraw[fill=white, draw=black] (3, 1) circle [radius=.05];
		\filldraw[fill=white, draw=black] (3, 1.5) circle [radius=.05];
		\filldraw[fill=white, draw=black] (3, 2) circle [radius=.05];
		\filldraw[fill=white, draw=black] (3, 2.5) circle [radius=.05];
		\filldraw[fill=white, draw=black] (3, 3) circle [radius=.05];
		\filldraw[fill=white, draw=black] (3, 3.5) circle [radius=.05];
		
		\filldraw[fill=white, draw=black] (3.5, .5) circle [radius=.05];
		\filldraw[fill=white, draw=black] (3.5, 1) circle [radius=.05];
		\filldraw[fill=white, draw=black] (3.5, 1.5) circle [radius=.05];
		\filldraw[fill=white, draw=black] (3.5, 2) circle [radius=.05];
		\filldraw[fill=white, draw=black] (3.5, 2.5) circle [radius=.05];
		\filldraw[fill=white, draw=black] (3.5, 3) circle [radius=.05];
		\filldraw[fill=white, draw=black] (3.5, 3.5) circle [radius=.05];
		
		\filldraw[fill=white, draw=black] (4, .5) circle [radius=.05];
		\filldraw[fill=white, draw=black] (4, 1) circle [radius=.05];
		\filldraw[fill=white, draw=black] (4, 1.5) circle [radius=.05];
		\filldraw[fill=white, draw=black] (4, 2) circle [radius=.05];
		\filldraw[fill=white, draw=black] (4, 2.5) circle [radius=.05];
		\filldraw[fill=white, draw=black] (4, 3) circle [radius=.05];
		\filldraw[fill=white, draw=black] (4, 3.5) circle [radius=.05];

		\draw[->,black, thick] (9.95, 2.2) -- (9.95, 4);
		\draw[->,black, thick] (10.05, 2.2) -- (10.05, 4);
		\draw[->,black, thick] (9.85, 0) -- (9.85, 1.8);
		\draw[->,black, thick] (9.95, 0) -- (9.95, 1.8);
		\draw[->,black, thick] (10.05, 0) -- (10.05, 1.8);
		\draw[->,black, thick] (10.15, 0) -- (10.15, 1.8);
		\draw[->,black, thick] (8, 1.9) -- (9.8, 1.9);
		\draw[->,black, thick] (8, 2) -- (9.8, 2);
		\draw[->,black, thick] (8, 2.1) -- (9.8, 2.1);
		\draw[->,black, thick] (10.2, 1.85) -- (12, 1.85); 
		\draw[->,black, thick] (10.2, 1.925) -- (12, 1.925); 
		\draw[->,black, thick] (10.2, 2) -- (12, 2); 
		\draw[->,black, thick] (10.2, 2.075) -- (12, 2.075); 
		\draw[->,black, thick] (10.2, 2.15) -- (12, 2.15);

		\filldraw[fill=white, draw=black] (10, 2) circle [radius=.2];

		\end{tikzpicture}
		
	\end{center}	
	
	\caption{\label{vertexedgespaths} A directed path ensemble is depicted above and to the left. A vertex with arrow configuration $(i_1, j_1; i_2, j_2) = (4, 3; 2, 5)$ is depicted above and to the right.} 
\end{figure}

In particular, each vertex $(a, b) \in \mathcal{D}$ has some number $i_1 = i_1 (a, b)$ of incoming vertical edges (namely, directed edges from $(a, b - 1)$ to $(a, b)$); some number $j_1 = j_1 (a, b)$ of incoming horizontal edges (from $(a - 1, b)$ to $(a, b)$); $i_2 = i_2 (a, b)$ outgoing vertical edges (from $(a, b)$ to $(a, b + 1)$); and $j_2 = j_2 (a, b)$ outgoing horizontal edges (from $(a, b)$ to $(a + 1, b)$). We refer to the right side of Figure \ref{vertexedgespaths} for an example (where there we allow multiple arrows to occupy edges). The quadruple $\big(i_1, j_1; i_2, j_2 \big)$ is called the \emph{arrow configuration} associated with the vertex $(a, b) \in \mathcal{D}$. Observe that the same number of arrows enter and exit each vertex, that is, $i_1 + j_1 = i_2 + j_2$; this is sometimes referred to as \emph{arrow (spin) conservation}.

We consider the quantities $S_{1; 1}^{\text{ell}} \big(i_1, j_1; i_2, j_2 \b| \lambda, v; x, y \big)$ as weights associated with a vertex $u \in \mathbb{Z}_{> 0}^2$ with arrow configuration $(i_1, j_1; i_2, j_2)$; here, $x$ is viewed as a \emph{rapidity parameter} associated with the row of $u$ and $y$ as one associated with the column of $u$. Observe in the present case that $S_{1; 1}^{\text{ell}}$ is equal to $0$ unless $i_1, j_1, i_2, j_2 \in \{ 0, 1 \}$; this condition will be removed in later examples. The complex numbers $\lambda = \lambda (u)$ and $v = v(u)$ are \emph{dynamical parameters} that change between vertices according to the identities 
\begin{flalign*}
& \lambda (a + 1, b) = \lambda (a, b) + 2 \eta (i_2 - 1); \qquad \lambda (a, b - 1) = \lambda (a, b) + 2 \eta (2j_1 - 1); \\
& v (a - 1, b) = v(a, b) - 2 \eta i_1 (a, b); \qquad  v (a, b + 1) = v(a, b) - 2 \eta j_2 (a, b),
\end{flalign*}

\noindent where $(i_1, j_1; i_2, j_2)$ is the arrow configuration associated with the vertex $u = (a, b)$; we refer to Figure \ref{vdynamicalvertex1elliptic1} for depictions of these identities, where for each vertex $u$, the dynamical parameters $\lambda (u)$ and $v(u)$ there are drawn in the upper-left face and lower-right containing $u$, respectively.

While at first glance it might seem as if one would need to keep track of two height functions in order to reformulate the above system as a face (or SOS, or IRF) model, with weights given by products of the contributions of plaquettes on the dual lattice, it actually suffices to only track one. Indeed, the identities according to which the dynamical parameters $\lambda$ and $v$ change across the lattice are determined by the directed path ensemble (or, equivalently, by the configuration of arrows entering and exiting each vertex), which can be encoded by a single height function on the dual lattice. Thus \Cref{1jsdefinition}, as well as its more general fused versions described below, can be viewed as an SOS model in the usual sense. 	

The $S_{1; 1}^{\text{ell}}$ weights above are in fact the $J = 1 = \Lambda$ special cases of the \emph{stochasticized elliptic fused weights} $S_{J; \Lambda}$ given by \Cref{selliptic} and \eqref{sjlambda1} below, which are (both vertically and horizontally) fused versions of the $S_{1; 1}^{\text{ell}}$ weights. We will not describe them here since their form, given by ${_{12} v_{11}}$ very well-poised, balanced elliptic hypergeometric series, is a bit lengthy. However, let us mention that several special cases of these fused weights that simplify considerably include their higher spin $S_{1; \Lambda}$ and elliptic Hahn degenerations, given by \Cref{1js} and \Cref{svlambdaelliptic}, respectively.

As indicated by \Cref{dynamicalequationellipticstochastic} below, these $S_{J; \Lambda}$ weights are stochastic and satisfy the dynamical Yang-Baxter equation. The $J = \Lambda = T = 1$ special case of that result is given by the following theorem.

\begin{thm}
	
	\label{lambda1equation}
	
	For any fixed $x, y, \lambda, v  \in \mathbb{C}$ and $i_1, j_1 \in \{ 0, 1 \}$, we have that 
	\begin{flalign*} 
	\displaystyle\sum_{i_2, j_2 \in \{ 0, 1 \}} S \big( i_1, j_1; i_2, j_2 \b| \lambda, v; x, y \big) = 1,
	\end{flalign*} 
	
	\noindent where we abbreviated $S \big( i_1, j_1; i_2, j_2 \b| \lambda, v; x, y \big) = S_{1; 1}^{\text{\emph{ell}}} \big( i_1, j_1; i_2, j_2 \b| \lambda, v; x, y \big)$.  
	
	Furthermore, for any fixed $i_1, j_1, k_1, i_3, j_3, k_3 \in \{ 0, 1 \}$ and $\lambda, v, x, y, z \in \mathbb{C}$, we have that 	
	\begin{flalign*}
	\displaystyle\sum_{i_2, j_2, k_2 \in \{ 0, 1 \}} & S \big( i_1, j_1; i_2, j_2 \b| \lambda; v - 2 \eta k_1; x, y \big) S \big( k_1, j_2; k_2, j_3 \b| \lambda + 2 \eta (2i_2 - 1); v; x, z \big) \\
	&  \times  S \big( k_2, i_2; k_3, i_3 \b| \lambda; v - 2 \eta j_3; y, z \big) \\
	\quad = \displaystyle\sum_{i_2, j_2, k_2 \in \{ 0, 1 \} } & S \big( k_1, i_1; k_2, i_2 \b| \lambda + 2 \eta (2j_1 - 1); v; y, z \big) S \big( k_2, j_1; k_3, j_2 \b| \lambda; v - 2 \eta i_2; x, z \big) 	\\
	& \times S \big( i_2, j_2; i_3, j_3 \b| \lambda + 2 \eta (2k_3 - 1); v; x, y \big).
	\end{flalign*}

\end{thm}

\begin{figure}

	\begin{center}

		\begin{tikzpicture}[
		>=stealth,
		auto,
		style={
			scale = 1.8
		}
		]

		\filldraw[fill=white, draw=black] (1, 0) circle [radius=.1] node[scale = .7, below = 37, right = 4]{$\lambda + 2 \eta (2i_2 + 2j_2 - 2)$} node[scale = 1, above = 28, left = 18]{$\lambda$};
		
		\filldraw[fill=white, draw=black] (1, 1) circle [radius=0] node[scale = .8, below = 30, right = 4]{$\lambda + 2 \eta (2i_2 - 1)$};
		
		\filldraw[fill=white, draw=black] (0, 0) circle [radius=0];

		\draw[->, black, thick] (1, -.9) -- (1, -.1) node[below = 40, scale = .8]{$i_1$} node[below = 57, scale = .8]{$y$} node[below = 67, scale = .8]{}; 
		
		\draw[->, black, thick] (1, .1) -- (1, .9) node[scale = .8, above = 0]{$i_2$}; 
		
		\draw[->, black, thick] (.1, 0) -- (.9, 0)  node[left = 40, scale = .8]{$j_1$} node[left = 57, scale = .8]{$x$} node[left = 67, scale = .8]{} node[scale = .8, below = 33, left = 2]{$\lambda + 2 \eta (2j_1 - 1)$}; 
		
		\draw[->, black, thick] (1.1, 0) -- (1.9, 0) node[right = 0, scale = .8]{$j_2$};

		\filldraw[fill=white, draw=black] (6, 0) circle [radius=.1] node[below = 25, right = 15]{$v$} node[scale = .75, above = 33, left = 2]{$v - 2 \eta (i_1 + j_1)$};
		
		\filldraw[fill=white, draw=black] (6, 1) circle [radius=0] node[scale = .9, below = 28, right = 5]{$v -2 \eta j_2$};
		
		\filldraw[fill=white, draw=black] (5, 0) circle [radius=0] node[scale = .9, below = 28, right = 7]{$v -2 \eta i_1$};

		\draw[->, black, thick] (6	, -.9) -- (6, -.1) node[below = 40, scale = .8]{$i_1$}; 
		
		\draw[->, black, thick] (6, .1) -- (6, .9) node[scale = .8, above = 0]{$i_2$}; 
		
		\draw[->, black, thick] (5.1, 0) -- (5.9, 0) node[left = 40, scale = .8]{$j_1$}; 
		
		\draw[->, black, thick] (6.1, 0) -- (6.9, 0) node[right = 0, scale = .8]{$j_2$};

		\end{tikzpicture}
		
	\end{center}

	\caption{\label{vdynamicalvertex1elliptic1} Depicted above is the way in which the dynamical parameters $\lambda$ and $v$ change between faces in the elliptic, stochastic spin $\frac{1}{2}$ model.  }
\end{figure}
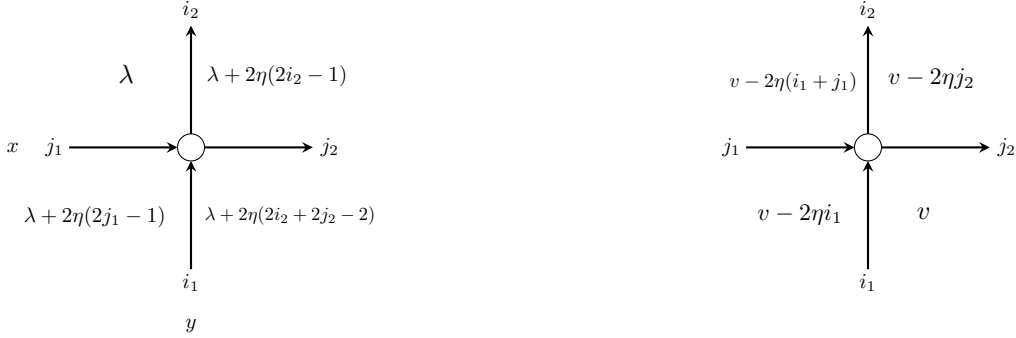

The feature that makes the weights from \Cref{1jsdefinition} and the Yang-Baxter equation from \Cref{lambda1equation} different from what had been considered earlier is their dependence on the new dynamical parameter $v$. It is possible to remove their dependence on this parameter by first taking the \emph{trigonometric limit} of the Jacobi theta function, obtained by letting the parameter $\tau$ (recall \eqref{theta1}) tend to $\textbf{i} \infty$ so that  $f(z)$ will converge to $\sin z$ (after suitable rescaling), and then by letting $v$ tend to $\infty$. This will recover some of the dynamical stochastic higher spin vertex models studied in \cite{DSHSVM,ERSF,FW}.

\subsection{A Stochastic, Dynamical Higher Rank Solution}

\label{HigherRankEquation}

In this section we provide a solution to a dynamical variant of the Yang-Baxter equation that yields a stochastic, dynamical, higher rank (colored) vertex model. 

The solution to the Yang-Baxter equation that gives rise to this stochastic dynamical solution was originally proposed by Jimbo \cite{QMGS}. That solution was not stochastic, although it was observed in \cite{SM} that stochasticity could be obtained by applying a gauge transformation to its vertex weights; these stochastic solutions were later fused in \cite{STRR} and then analyzed from the perspectives of duality and spectral theory in \cite{CDFSVMD} and \cite{CSVMST}, respectively. 

The stochasticization procedure enables one to produce dynamical deformations of these colored stochastic vertex models, which depends on an additional parameter $v$ that changes between vertices. We begin with the following definition for the vertex weights.

\begin{definition}
	
	\label{l1sdefinition} 
	
	Let $n$ be a positive integer and $q, v, x, y$ be complex numbers. For any integer $i \in [1, n]$ or pair of integers $(i, j)$ with $1 \le i < j \le n$, define 
	\begin{flalign}
	\label{11rankdynamicalsdefinition}
	\begin{aligned} 
	& S_{1; 1}^{\text{col}} \big( i, 0, 0, i \b| x, y; v \big)  = \displaystyle\frac{(1 - q ) x (1 - qvy)}{(x - q y) (1 - qvx)}; \quad S_{1; 1}^{\text{col}} \big( i, 0, i, 0 \b| x, y; v \big) =  \displaystyle\frac{q (x -  y) (1 - vx) }{(x - q y) (1 - qvx)}; \\
	& S_{1; 1}^{\text{col}}  \big( 0, i, 0, i \b| x, y; v \big)  =  \displaystyle\frac{(x - y) (1 - q v y)}{(x - q y) (1 -  v y)}; \qquad  S_{1; 1}^{\text{col}} \big( 0, i, i, 0 \b| x, y; v \big)  = \displaystyle\frac{(1 - q) y (1 - vx)}{(x - q y) (1 - vy)}; \\
	& S_{1; 1}^{\text{col}}  \big( i, j, j, i \b| x, y; v \big) =  \displaystyle\frac{(1 - q)  x}{x - q y}; \qquad \qquad \quad  S_{1; 1}^{\text{col}} \big( j, i, i, j \b| x, y; v \big) = \displaystyle\frac{(1 - q) y}{x - q y}; \\
	& S_{1; 1}^{\text{col}}  \big( i, j, i, j \b| x, y; v \big) =  \displaystyle\frac{(x - y) q}{x - q y}; \qquad \qquad \quad S_{1; 1}^{\text{col}}  \big( j, i, j, i  \b| x, y; v \big)  =  \displaystyle\frac{x - y}{x - q y}; \\
	& \qquad \qquad \qquad \qquad \quad  S_{1; 1}^{\text{col}} \big( 0, 0, 0, 0 \b| x, y; v \big)  = 1 = S_{1; 1}^{\text{col}} \big( j, j, j, j \b| x, y; v \big),
	\end{aligned}
	\end{flalign}
	
	\noindent and $S_{1; 1}^{\text{col}}  \big( A, B, C, D \b| x, y; v \big) = 0$ for any $(A, B, C, D)$ not of the above form. 
	
\end{definition}

As in \Cref{EllipticEquation}, the $S_{1; 1}^{\text{col}}$ weights have diagrammatic interpretations. Specifically, let $\mathcal{D}$ be a finite subset of the graph $\mathbb{Z}_{\ge 0}^2$. An \emph{$n$-colored directed path ensemble} on $\mathcal{D}$ is a family of \emph{colored paths} connecting adjacent vertices of $\mathcal{D}$, each edge of which is either directed up or to the right and is assigned one of $n$ colors (which are labeled by the integers $\{ 1, 2, \ldots , n \}$). We assume that each edge can accommodate at most one path (although this restriction will be weakened later); if an edge does not accommodate a path, then we assign it color $0$. 

The (colored) arrow configuration associated with some vertex $u \in \mathcal{D}$ is defined to be a certain quadruple of integers $(A, B, C, D)$. Here, $A$ denotes the color of the incoming vertical arrow at $u$; $B$ denotes the color of the incoming horizontal arrow; $C$ the color of the outgoing vertical arrow; and $D$ the color of the outgoing horizontal arrow.

Then we view the quantity $S_{1; 1}^{\text{col}} \big( A, B, C, D \b| x, y; v \big)$ as the weight associated with a vertex $u \in \mathcal{D}$ whose colored arrow configuration is $(A, B, C, D)$. Here, $x$ and $y$ denote the rapidity parameters at $u$ in the horizontal and vertical directions, respectively. The parameter $v$ can again be viewed as a dynamical parameter, so we sometimes denote $v = v(u)$ by the value of $v$ at a vertex $u$ in $\mathbb{Z}_{> 0}^2$; the identities governing $v(u)$ are given by 
\begin{flalign*}
v (a - 1, b) = q^{1 - \textbf{1}_{A = 0}} v; \qquad v (a, b + 1) = q^{1 - \textbf{1}_{D = 0}} v,
\end{flalign*}

\noindent for any $(a, b) \in \mathbb{Z}_{> 0}^2$ whose arrow configuration is $(A, B, C, D)$. These identities are depicted in \Cref{configurationrank1}, where $v(u)$ is labeled in the lower-right face containing $u$.

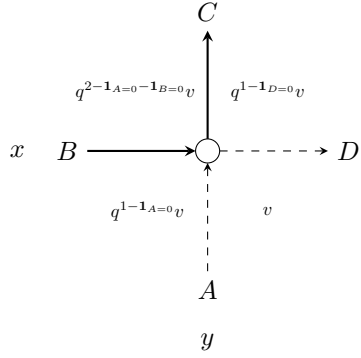
\begin{figure}
	
	\begin{center} 
		
		\begin{tikzpicture}[
		>=stealth,
		scale = .8
		]

		\draw[->,black, thick] (0, .2) -- (0, 2);
		
		\draw[->,black, dashed] (0,-2) -- (0, -.2);
		
		\draw[->,black, thick] (-2, 0) -- (-.2, 0);
		
		\draw[->,black, dashed] (.2, 0) -- (2, 0); 
		
		\draw (0, -2) circle[radius = 0] node[below = 0]{$A$} node[below = 20]{$y$} node[below = 30]{}; 
		\draw (-2, 0) circle[radius = 0] node[left = 0]{$B$} node[left = 20]{$x$} node[left = 30]{}; 
		\draw (0, 2) circle[radius = 0] node[above = 0]{$C$}; 
		\draw (2, 0) circle[radius = 0] node[right = 0]{$D$};

		\draw (-1.2, 1) circle[radius = 0] node[scale = .7]{$q^{2 - \textbf{1}_{A = 0} - \textbf{1}_{B = 0}} v$}; 
		\draw (1, 1) circle[radius = 0] node[scale = .7]{$q^{1 - \textbf{1}_{D = 0}} v$}; 
		\draw (-1, -1) circle[radius = 0] node[scale = .7]{$q^{1 - \textbf{1}_{A = 0}} v$}; 
		\draw (1, -1) circle[radius = 0] node[scale = .7]{$v$};

		\filldraw[fill=white, draw=black] (0, 0) circle [radius=.2];

		\end{tikzpicture}
		
	\end{center}

	\caption{\label{configurationrank1} Depicted above is a colored vertex $u$, where the dashed and solid arrows are associated with colors $1$ and $2$, respectively. The arrow configuration of $u$ is $(A, B, C, D) = (1, 2, 2, 1)$. The dynamical parameters are also labeled in each of the four faces passing through $u$. In the situation depicted above, we have that $\textbf{1}_{A = 0} = \textbf{1}_{B = 0} = \textbf{1}_{C = 0} = \textbf{1}_{D = 0} = 0$. } 
\end{figure}

Similar to in \Cref{EllipticEquation}, the weights given by \Cref{l1sdefinition} are the $L = 1 = M$ special cases of certain fused stochastic weights $S_{L; M}$ given by \Cref{srankc} and \eqref{svrank} below. Those weights are more general, allowing for multiple arrows to exist along vertical and horizontal edges. However, for simplicity, we will not describe them here since their expression is again a bit lengthy. 

As indicated by \Cref{dynamicalrankstochastic} below, these fused weights are stochastic and satisfy a dynamical variant of the Yang-Baxter equation. The $L = M = T = 1$ case of that result states the following.

\begin{thm} 
	
	\label{dynamicalrankstochasticlm1} 
	
	For any integers $a, b \in [0, n]$ and complex numbers $x, y, v \in \mathbb{C}$, we have that 
	\begin{flalign*}
	\displaystyle\sum_{0 \le c, d \le n}  S \big( a, b, c, d \b| x, y; v \big) = 1,
	\end{flalign*} 
	
	\noindent where we abbreviated $S \big( a, b, c, d \b| x, y; v \big) = S_{1; 1}^{\text{\emph{col}}} \big( a, b, c, d \b| x, y; v \big)$.
	
	Furthermore, for any fixed complex numbers $x, y, z, v$ and integers $i_1, j_1, k_1, i_3, j_3, k_3 \in [0, n]$, we have that 
	\begin{flalign*}
	\displaystyle\sum_{0 \le i_2, j_2, k_2 \le n}  & S \big( i_1, j_1; i_2, j_2 \big| x, y; q^{1 - \textbf{\emph{1}}_{k_1 = 0}} v \big) S \big( k_1, j_2; k_2, j_3 \Big| x, z; v \big)  S  \big( k_2, i_2; k_3, i_3  \big| y, z; q^{1 - \textbf{\emph{1}}_{j_3 = 0}} v \big) \\
	\quad = \displaystyle\sum_{0 \le i_2, j_2, k_2 \le n}  & S  \big( k_1, i_1; k_2, i_2  \big| y, z; v \big) S  \big( k_2, j_1; k_3, j_2  \big| x, z; q^{1 - \textbf{\emph{1}}_{i_2 = 0}} v \big)  S  \big( i_2, j_2; i_3, j_3  \big| x, y; v \big).
	\end{flalign*}

\end{thm}

As in \Cref{EllipticEquation}, it is possible to remove the dependence on the new dynamical parameter $v$ from both the weights from \Cref{l1sdefinition} and the Yang-Baxter equation from \Cref{dynamicalrankstochasticlm1}. This can be done by either setting $v = 0$ or letting $v$ tend to $\infty$, in which case one recovers the models introduced in \cite{SM}.

\subsection{A Stochastic Solution of the Tetrahedron Equation}

\label{TetrahedronStochastic}

The \emph{tetrahedron equation} is a three-dimensional analog of the Yang-Baxter equation. While many solutions are known to the latter, few solutions are known to the former. The earliest was predicted by Zamolodchikov in 1980 \cite{TEISS,TERSMSSD} and then confirmed by Baxter in \cite{TE} shortly afterwards. A number of additional solutions to the tetrahedron equation were found over the next thirty years \cite{NSLMD,TEQG,QGLTE,STEITM}, although all of these (including the original one of Zamolodchikov) had at least one vertex weight that was negative. The first family of solutions to the tetrahedron equation with all nonnegative vertex weights was introduced by Mangazeev-Bazhanov-Sergeev in 2013 \cite{ILMPW}; none of these solutions happened to be stochastic. 

By applying our stochasticization procedure to the solution provided in \cite{ILMPW}, we provide the first example of a stochastic solution to a (dynamical variant) of the tetrahedron equation. Its vertex weights are given as follows.

\begin{definition} 
	
	\label{sskweight}
	
	For any $q, v \in \mathbb{C}$ and $n_1, n_2, n_3, n_1', n_2', n_3' \in \mathbb{Z}_{\ge 0}$, define
	\begin{flalign} 
	\label{sn1n2n3vweight}
	\begin{aligned}
	S_{n_1 n_2 n_3}^{n_1' n_2' n_3'} (v) & = q^{(n_1 + n_3 + n_1' + n_3' + 2) n_2 - 2 (n_1' + 1) n_2'} \displaystyle\frac{(q^2; q^2)_{n_1} (q^{-2n_1'}; q^2)_{n_2} (q^2; q^2)_{n_3}}{(q^2; q^2)_{n_1'} (q^2; q^2)_{n_2'} (q^2; q^2)_{n_3'}} \\
	& \quad \times v^{n_2'}  \displaystyle\frac{(v q^{-2 n_1'}; q^2)_{n_3'} }{(v; q^2)_{n_3}} {_2 \varphi_1}  \Bigg( \begin{array}{cc}  q^{-2n_2}, q^{2n_1 + 2} \\ q^{2(n_1' - n_2 + 1)} \end{array}\Bigg| q^2, q^{-2n_3}\Bigg) \textbf{1}_{n_1 + n_2 = n_1' + n_2'} \textbf{1}_{n_2 + n_3 = n_2' + n_3'}, 
	\end{aligned}
	\end{flalign} 
	
	\noindent where $(a; q^2)_k$ denotes the $q$-Pochhammer symbol, as in the second identity in \eqref{productbasicelliptic}, and ${_2 \varphi_1}$ denotes the basic hypergeometric series, as in \eqref{basichypergeometric}.
\end{definition}

The diagrammatic interpretation of the quantities $S_{n_1 n_2 n_3}^{n_1' n_2' n_3'} (v)$ is that they are weights associated with a vertex $u$ in $\mathbb{Z}^3$ with arrow configuration $(n_1, n_2, n_3; n_1', n_2', n_3')$; here, this means that $u$ has $n_1$ incoming arrows parallel to one direction (say the $x$-axis); $n_2$ incoming arrows parallel to the $y$-axis; $n_3$ incoming arrows parallel to the $z$-axis; $n_1'$ outgoing arrows parallel to the $x$-axis; $n_2'$ outgoing arrows parallel to the $y$-axis; and $n_3'$ outgoing arrows parallel to the $z$-axis. Observe that, with this understanding, the $S$ weights do not satisfy arrow conservation, in that they are not supported on arrow configurations $(n_1, n_2, n_3; n_1', n_2', n_3')$ satisfying $n_1 + n_2 + n_3 = n_1' + n_2' + n_3'$. The parameter $v$ is again a dynamical parameter that changes between vertices; we will not explicitly state the identities governing this parameter here but instead refer to \Cref{3equationdimension} for a depiction. 

The following theorem states that these $S$ weights are stochastic and satisfy a dynamical variant of the tetrahedron equation, which is depicted in \Cref{3equationdimension}; observe that there are two non-interacting dynamical parameters (namely $v$ and $w$) in this equation. The proof of this theorem will appear in \Cref{EquationStochastic3} below.

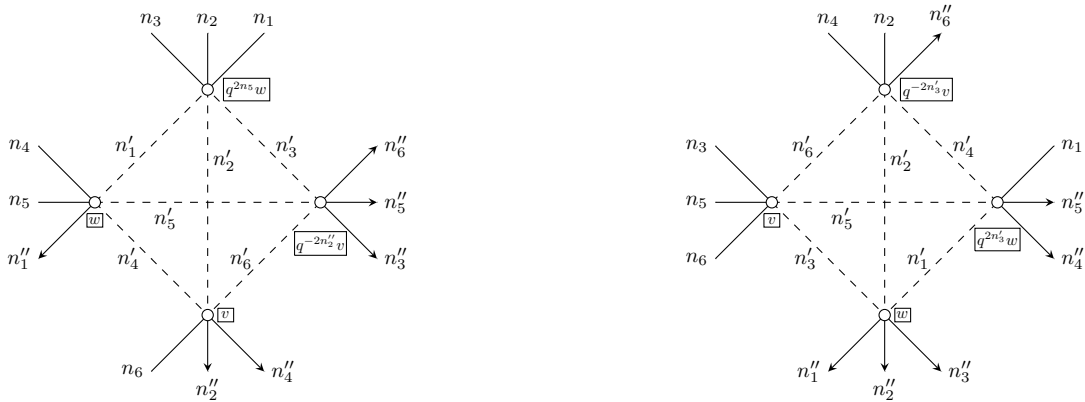
\begin{figure}

	\begin{center}

		\begin{tikzpicture}[
		>=stealth,
		auto,
		style={
			scale = 1.5	
		}
		]
		
		\draw[-, black] (2.5, 1.5) -- (2, 1);
		\draw[-, black] (2, 1.5) -- (2, 1);
		\draw[-, black] (1.5, 1.5) -- (2, 1);
		\draw[-, black] (.5, .5) -- (1, 0);
		\draw[-, black] (.5, 0) -- (1, 0);
		\draw[-, black] (1.5, -1.5) -- (2, -1);
		
		\draw[->, black] (1, 0) -- (.5, -.5);
		\draw[->, black] (2, -1) -- (2, -1.5);
		\draw[->, black] (3, 0) -- (3.5, -.5);
		\draw[->, black] (2, -1) -- (2.5, -1.5);
		\draw[->, black] (3, 0) -- (3.5, 0);
		\draw[->, black] (3, 0) -- (3.5, .5);

		\draw[-, black, dashed] (1, 0) -- (3, 0); 
		\draw[-, black, dashed] (1, 0) -- (2, 1); 
		\draw[-, black, dashed] (1, 0) -- (2, -1); 
		\draw[-, black, dashed] (3, 0) -- (2, 1); 
		\draw[-, black, dashed] (3, 0) -- (2, -1); 
		\draw[-, black, dashed] (2, 1) -- (2, -1);

		\filldraw[fill = white] (1, 0) circle[radius = .05] node[scale = .6, below = 6]{$w$};
		\filldraw[fill = white] (3, 0) circle[radius = .05] node[scale = .6, below = 17]{$q^{-2n_2''} v$};
		\filldraw[fill = white] (2, 1) circle[radius = .05] node[scale = .6, right = 8]{$q^{2n_5} w$};
		\filldraw[fill = white] (2, -1) circle[radius = .05] node[scale = .6, right = 5]{$v$};

		\filldraw[fill = black] (2.5, 1.5) circle[radius = 0] node[scale = .8, above = 0]{$n_1$};
		\filldraw[fill = black] (2, 1.5) circle[radius = 0] node[scale = .8, above = 0]{$n_2$};
		\filldraw[fill = black] (1.5, 1.5) circle[radius = 0] node[scale = .8, above = 0]{$n_3$};
		\filldraw[fill = black] (.5, .5) circle[radius = 0] node[scale = .8, left = 0]{$n_4$};
		\filldraw[fill = black] (.5, 0) circle[radius = 0] node[scale = .8, left = 0]{$n_5$};
		\filldraw[fill = black] (1.5, -1.5) circle[radius = 0] node[scale = .8, left = 0]{$n_6$};

		\filldraw[fill = black] (.5, -.5) circle[radius = 0] node[scale = .8, left = 0]{$n_1''$};
		\filldraw[fill = black] (2, -1.5) circle[radius = 0] node[scale = .8, below = 0]{$n_2''$};
		\filldraw[fill = black] (3.5, -.5) circle[radius = 0] node[scale = .8, right = 0]{$n_3''$};
		\filldraw[fill = black] (2.5, -1.5) circle[radius = 0] node[scale = .8, right = 0]{$n_4''$};
		\filldraw[fill = black] (3.5, 0) circle[radius = 0] node[scale = .8, right = 0]{$n_5''$};
		\filldraw[fill = black] (3.5, .5) circle[radius = 0] node[scale = .8, right = 0]{$n_6''$};

		\filldraw[fill = black] (1.5, .5) circle[radius = 0] node[scale = .8, left = 3]{$n_1'$};
		\filldraw[fill = black] (2, .375) circle[radius = 0] node[scale = .8, right = -1]{$n_2'$};
		\filldraw[fill = black] (2.5, .5) circle[radius = 0] node[scale = .8, right = 2]{$n_3'$};
		\filldraw[fill = black] (1.5, -.5) circle[radius = 0] node[scale = .8, left = 2]{$n_4'$};
		\filldraw[fill = black] (1.625, 0) circle[radius = 0] node[scale = .8, below = -1]{$n_5'$};
		\filldraw[fill = black] (2.5, -.5) circle[radius = 0] node[scale = .8, left = 2]{$n_6'$};

		\draw (.93, -.1) -- (.93, -.22) -- (1.07, -.22) -- (1.07, -.1) -- (.93, -.1);
		\draw (2.09, -.94) -- (2.09, -1.06) -- (2.23, -1.06) -- (2.23, -.94) -- (2.09, -.94); 
		\draw (2.14, .9) -- (2.14, 1.1) -- (2.55, 1.1) -- (2.55, .9) -- (2.14, .9); 	
		\draw (2.77, -.47) -- (2.77, -.26) -- (3.23, -.26) -- (3.23, -.47) -- (2.77, -.47);

		\draw[->, black] (8, 1) -- (8.5, 1.5);
		\draw[-, black] (8, 1.5) -- (8, 1);
		\draw[-, black] (7.5, 1.5) -- (8, 1);
		\draw[-, black] (6.5, .5) -- (7, 0);
		\draw[-, black] (6.5, 0) -- (7, 0);
		\draw[->, black] (8, -1) -- (7.5, -1.5);
		
		\draw[-, black] (6.5, -.5) -- (7, 0);
		\draw[->, black] (8, -1) -- (8, -1.5);
		\draw[->, black] (9, 0) -- (9.5, -.5);
		\draw[->, black] (8, -1) -- (8.5, -1.5);
		\draw[->, black] (9, 0) -- (9.5, 0);
		\draw[-, black] (9.5, .5) -- (9, 0);

		\draw[-, black, dashed] (7, 0) -- (9, 0); 
		\draw[-, black, dashed] (7, 0) -- (8, 1); 
		\draw[-, black, dashed] (7, 0) -- (8, -1); 
		\draw[-, black, dashed] (9, 0) -- (8, 1); 
		\draw[-, black, dashed] (9, 0) -- (8, -1); 
		\draw[-, black, dashed] (8, 1) -- (8, -1);

		\filldraw[fill = white] (7, 0) circle[radius = .05] node[scale = .6, below = 6]{$v$};
		\filldraw[fill = white] (9, 0) circle[radius = .05] node[scale = .6, below = 14]{$q^{2n_3'} w$};
		\filldraw[fill = white] (8, 1) circle[radius = .05] node[scale = .6, right = 8]{$q^{-2n_3'} v$};
		\filldraw[fill = white] (8, -1) circle[radius = .05] node[scale = .6, right = 4]{$w$};
		
		\draw (6.93, -.1) -- (6.93, -.22) -- (7.07, -.22) -- (7.07, -.1) -- (6.93, -.1);
		\draw (8.09, -.94) -- (8.09, -1.06) -- (8.23, -1.06) -- (8.23, -.94) -- (8.09, -.94); 
		\draw (8.14, .9) -- (8.14, 1.1) -- (8.6, 1.1) -- (8.6, .9) -- (8.14, .9); 	
		\draw (8.8, -.43) -- (8.8, -.23) -- (9.2, -.23) -- (9.2, -.43) -- (8.8, -.43); 
		
		\filldraw[fill = black] (9.5, .5) circle[radius = 0] node[scale = .8, right = 0]{$n_1$};
		\filldraw[fill = black] (8, 1.5) circle[radius = 0] node[scale = .8, above = 0]{$n_2$};
		\filldraw[fill = black] (6.5, .5) circle[radius = 0] node[scale = .8, left = 0]{$n_3$};
		\filldraw[fill = black] (7.5, 1.5) circle[radius = 0] node[scale = .8, above = 0]{$n_4$};
		\filldraw[fill = black] (6.5, 0) circle[radius = 0] node[scale = .8, left = 0]{$n_5$};
		\filldraw[fill = black] (6.5, -.5) circle[radius = 0] node[scale = .8, left = 0]{$n_6$};

		\filldraw[fill = black] (7.5, -1.5) circle[radius = 0] node[scale = .8, left = 0]{$n_1''$};
		\filldraw[fill = black] (8, -1.5) circle[radius = 0] node[scale = .8, below = 0]{$n_2''$};
		\filldraw[fill = black] (8.5, -1.5) circle[radius = 0] node[scale = .8, right = 0]{$n_3''$};
		\filldraw[fill = black] (9.5, -.5) circle[radius = 0] node[scale = .8, right = 0]{$n_4''$};
		\filldraw[fill = black] (9.5, 0) circle[radius = 0] node[scale = .8, right = 0]{$n_5''$};
		\filldraw[fill = black] (8.5, 1.5) circle[radius = 0] node[scale = .8, above = 0]{$n_6''$};
		
		\filldraw[fill = black] (8.5, -.5) circle[radius = 0] node[scale = .8, left = 2]{$n_1'$};
		\filldraw[fill = black] (8, .375) circle[radius = 0] node[scale = .8, right = -1]{$n_2'$};
		\filldraw[fill = black] (7.5, -.5) circle[radius = 0] node[scale = .8, left = 2]{$n_3'$};
		\filldraw[fill = black] (8.5, .5) circle[radius = 0] node[scale = .8, right = 2]{$n_4'$};
		\filldraw[fill = black] (7.625, 0) circle[radius = 0] node[scale = .8, below = -1]{$n_5'$};
		\filldraw[fill = black] (7.5, .5) circle[radius = 0] node[scale = .8, left = 3]{$n_6'$};

		\end{tikzpicture}
		
	\end{center}

	\caption{\label{3equationdimension} The dynamical tetrahedron equation, as in \Cref{stochastic3sequation}, is depicted above; the numbers of arrows on the dashed lines is summed over. The dynamical parameters of the form $q^m u$ and $q^m w$ are boxed.}

\end{figure}

\begin{thm}
	
	\label{stochastic3sequation}
	
	Fix $v \in \mathbb{C}$ and $(n_1, n_2, n_3) \in \mathbb{Z}_{\ge 0}^3$. Then,  
	\begin{flalign}
	\label{stochasticsumsv}
	\displaystyle\sum_{\textbf{\emph{n}}'} S_{n_1 n_2 n_3}^{n_1' n_2' n_3'} (v) = 1,
	\end{flalign}
	
	\noindent where the sum is over all $\textbf{\emph{n}}' = (n_1', n_2', n_3') \in \mathbb{Z}_{\ge 0}^3$. 
	
	Furthermore, fix complex numbers $v$ and $w$, as well as nonnegative integers $n_1, n_2, n_3, n_4, n_5, n_6$ and $n_1'', n_2'', n_3'', n_4'', n_5'', n_6''$. Then, 
	\begin{flalign}
	\label{dimensiondynamical}
	\begin{aligned} 
	& \displaystyle\sum_{\textbf{\emph{n}}'} S_{n_1 n_2 n_3}^{n_1' n_2' n_3'} (q^{2 n_5} w) S_{n_1' n_4 n_5}^{n_1'' n_4' n_5'} (w) S_{n_3' n_5' n_6'}^{n_3'' n_5'' n_6''} (q^{- 2 n_2''} v) S_{n_2' n_4' n_6}^{n_2'' n_4'' n_6'} (v) \\
	& \qquad = \displaystyle\sum_{\textbf{\emph{n}}'} S_{n_3 n_5 n_6}^{n_3' n_5' n_6'} (v) S_{n_2 n_4 n_6'}^{n_2' n_4' n_6''} (q^{- 2 n_3'} v) S_{n_1 n_4' n_5'}^{n_1' n_4'' n_5''} (q^{2 n_3'} w) S_{n_1' n_2' n_3'}^{n_1'' n_2'' n_3''} (w), 
	\end{aligned}
	\end{flalign} 
	
	\noindent where the sum on both sides of \eqref{dimensiondynamical} is over nonnegative integers $\textbf{\emph{n}}' = (n_1', n_2', n_3', n_4', n_5', n_6')$.
\end{thm} 

Once again, it is possible to remove the dynamical feature of both the stochastic weights \eqref{sn1n2n3vweight} and the tetrahedron equation \eqref{dimensiondynamical}. This can be done by letting $q$ tend to $1$, under which degeneration the new weights (denoted by $T_{n_1 n_2 n_3}^{n_1' n_2' n_3'} (v)$) simplify considerably and are given explicitly by 
\begin{flalign*}
T_{n_1 n_2 n_3}^{n_1' n_2' n_3'} (v) & = v^{n_2'} (1 - v)^{n_2 - n_2'} \binom{n_2}{n_2'} \textbf{1}_{n_2' \le n_2} \textbf{1}_{n_1 + n_2 = n_1' + n_2'} \textbf{1}_{n_2 + n_3 = n_2' + n_3'}.
\end{flalign*}   

\noindent This serves as a (seemingly new) solution to the non-dynamical tetrahedron equation; we will describe this degeneration in more detail in \Cref{EquationDimension2}.

\subsection{Outline}

\label{Outline}

The remainder of this paper is organized as follows. We begin in Section \ref{VertexModel} with a preliminary example of the stochasticization procedure by applying it to the six-vertex solution to the Yang-Baxter equation to recover a dynamical variant of the stochastic six-vertex model that was introduced in \cite{FSSF}. Then, in Section \ref{Domain}, we explain the general version of stochasticization, which we implement in Section \ref{StochasticElliptic}, Section \ref{RankDynamical}, and Section \ref{DynamicalDimension} to establish the results from \Cref{EllipticEquation}, \Cref{HigherRankEquation}, and \Cref{TetrahedronStochastic}, respectively.

\subsection*{Acknowledgments}

The work of Amol Aggarwal was partially funded by the NSF Graduate Research Fellowship under grant number DGE1144152. The work of Alexei Borodin was partially supported by the NSF grant DMS-1607901 and DMS-1664619.

\section{A Dynamical Stochastic Six-Vertex Model}	

\label{VertexModel}

In this section we provide a preliminary example of the stochasticization procedure to recover the dynamical stochastic six-vertex model introduced as Definition 7.1 of \cite{FSSF}. Specifically, in Section \ref{SixVertexEquation} we present the solution to the Yang-Baxter equation that we will stochasticize, and in Section \ref{DynamicalVertex1} we discuss properties of the stochasticized weights and how they give rise to a dynamical six-vertex model. Throughout this section, we fix $q \in \mathbb{C}$.

\subsection{Stochasticizing a Six-Vertex Weight}

\label{SixVertexEquation} 	

In this section we recall the (higher spin) six-vertex solution to the Yang-Baxter equation and explain how it can be ``stochasticized.'' We begin by defining certain vertex weights, given by the following definition, which appears as Definition 2.1 of \cite{FSRF} with their $w_u$ replaced by our $\chi_s$ and their $u$ by our $\frac{x}{y}$.

\begin{definition}[{\cite[Definition 2.1]{FSRF}}]
	
	\label{spin12spins}
	
	For any $x, y, s \in \mathbb{C}$, define the quantities  $\chi_s (i_1, j_1; i_2, j_2) = \chi_s \big( i_1, j_2; i_2, j_2 \b| x, y \big)$ by
	\begin{flalign*}
	& \chi_s (k, 0; k, 0) = \displaystyle\frac{y - s q^k x}{y - s x}; \qquad \qquad \quad \chi_s (k, 0; k - 1, 1) = \displaystyle\frac{(1 - s^2 q^{k - 1} ) x}{y - s x}, \\
	& \chi_s (k, 1; k + 1, 0) = \displaystyle\frac{(1 - q^{k + 1}) y}{y - s x}; \qquad \chi_s (k, 1; k, 1) = \displaystyle\frac{x - s q^k y}{y - s x},
	\end{flalign*}
	
	\noindent for any nonnegative integer $k$, and $\chi_s (i_1, j_1; i_2, j_2) = 0$ for any $(i_1, j_1; i_2, j_2)$ not of the above form. 
	
	Further define the quantities  $w (i_1, j_1; i_2, j_2) = w \big(i_1, j_2; i_2, j_2 \b| x, y \big)$ by
	\begin{flalign*}
	& w (1, 0; 1, 0) = \displaystyle\frac{q (x - y)}{x - qy}; \qquad \quad w (1, 0; 0, 1) = \displaystyle\frac{(1 - q) x}{x - qy}, \\
	& w (0, 1; 1, 0) = \displaystyle\frac{(1 - q) y}{x - qy}; \qquad \quad w (0, 1; 0, 1) = \displaystyle\frac{x - y}{x - qy}, \\
	& \qquad \qquad \qquad w (0, 0; 0, 0) = 1 = (1, 1; 1, 1). 
	\end{flalign*}
	
	\noindent and $w (i_1, j_1; i_2, j_2) = 0$ for any $(i_1, j_1; i_2, j_2)$ not of the above form. 
	
\end{definition}

To provide a diagrammatic interpretation of the quantities $w (i_1, j_1; i_2, j_2)$ and $\chi_s (i_1, j_1; i_2, j_2)$, let $\mathcal{D}$ be a finite subset of the graph $\mathbb{Z}_{\ge 0}^2$, and consider a directed path ensemble on $\mathcal{D}$, as in \Cref{EllipticEquation} (see Figure \ref{vertexedgespaths} for an example). We consider the quantities $w \big(i_1, j_1; i_2, j_2 \b| x, y \big)$ and $\chi_s \big(i_1, j_1; i_2, j_2 \b| x, y \big)$ as weights associated with a vertex $u \in \mathbb{Z}_{> 0}^2$ with arrow configuration $(i_1, j_1; i_2, j_2)$; here, $x$ is viewed as a rapidity parameter associated with the row of $u$ and $y$ as one associated with the column of $u$. Observe in the present case that both $w$ and $\chi_s$ are equal to $0$ unless $j_1, j_2 \in \{ 0, 1 \}$; this condition will be removed in later examples.

The following proposition, which is a restatement of Proposition 2.5 of \cite{FSRF} (see also Section 4 and Section 5 of \cite{ESVM}), indicates that the $w$ and $\chi$ weights satisfy the ($RLL = LLR$ version of the) Yang-Baxter equation. 

\begin{prop}[{\cite[Proposition 2.5]{FSRF}}]
	
	\label{equationspin} 
	
	Fix $x, y, z \in \mathbb{C}$; $i_1, j_1, i_3, j_3 \in \{ 0, 1\}$ and $k_1, k_3 \in \mathbb{Z}_{\ge 0}$. Then, 
	\begin{flalign} 
	\label{equationspin1}
	\begin{aligned}
	\displaystyle\sum_{i_2, j_2, k_2 \in \mathbb{Z}_{\ge 0}}  & w \big( i_1, j_1; i_2, j_2 \b| x, y \big) \chi_s \big( k_1, j_2; k_2, j_3 \b| x, z \big)  \chi_s \big( k_2, i_2; k_3, i_3 \b| y, z \big) \\
	\qquad = \displaystyle\sum_{i_2, j_2, k_2 \in \mathbb{Z}_{\ge 0}} & \chi_s \big( k_1, i_1; k_2, i_2 \b| y, z \big)  \chi_s \big( k_2, j_1; k_3, j_2 \b| x, z \big)  w \big( i_2, j_2; i_3, j_3 \b| x, y \big). 
	\end{aligned}
	\end{flalign}
	
	\noindent The same statement holds if the $\chi_s$ weights in \eqref{equationspin1} are replaced by the $w$ weights.

\end{prop}

\begin{figure}

	\begin{center}

		\begin{tikzpicture}[
		>=stealth,
		auto,
		style={
			scale = 1.2
		}
		]

		\draw[->, black] (-.87, -.5) -- (-.087, -.05) node[above = 20, left = 28, scale = .8]{$j_1$} node[above = 20, left = 40, scale = .8]{$x$}; 
		\draw[->, black, dotted] (-.87, .5) -- (-.087, .05) node[below = 20, left = 25, scale = .8]{$i_1$} node[below = 20, left = 40, scale = .8]{$y$};
			
		\draw[->, black, dashed] (.087, -.05) -- (.783, -.45) node[above = 3, left = 11, scale = .8]{$j_2$}; 
		\draw[->, black, dashed] (.087, .05) -- (.783, .45) node[below = 3, left = 11, scale = .8]{$i_2$}; 
		
		\draw[->, black] (.87, -1.5) -- (.87, -.6) node[below = 38, right = -4, scale = .8]{$k_1$} node[below = 50, right = -4, scale = .8]{$z$};

		\draw[->, black, dashed] (.87, -.5) -- (.87, .4) node[below = 16, right = -1, scale = .8]{$k_2$}; 
		
		\draw[->, black] (.87, .5) -- (1.87, .5) node[right = 0, scale = .8]{$i_3$}; 
		
		\draw[->, black, dotted] (.87, -.5) -- (1.87, -.5) node[right = 0, scale = .8]{$j_3$};

		\draw[->, black] (.87, .5) -- (.87, 1.5) node[above = 5, right = -4, scale = .8]{$k_3$};

		\draw[black] (1.4, -.92) -- (1.58, -.92) -- (1.58, -1.12) -- (1.4, -1.12) -- (1.4, -.92); 
		\draw[black] (1.28, .12) -- (1.7, .12) -- (1.7, -.16) -- (1.28, -.16) -- (1.28, .12); 
		\draw[black] (-.23, -.62) -- (.23, -.62) -- (.23, -.92) -- (-.23, -.92) -- (-.23, -.62);

		\filldraw[fill=white, draw=black] (0, 0) circle [radius=.1] node[scale = .7, left = 6]{$w$} node[scale = .7, below = 29]{$q^{k_1} v$};
		
		\filldraw[fill=gray!50!white, draw=black] (.87, .5) circle [radius=.1] node[scale = .7, above = 6, right = 2]{$\chi_s$};
		
		\filldraw[fill=gray!50!white, draw=black] (.87, -.5) circle [radius=.1] node[scale = .7, below = 7, right = 2]{$\chi_s$}  node[scale = .7, right = 30, below = 20]{$v$}  node[scale = .7, right = 30, above = 15]{$q^{j_3} v$}; 
		
		\draw[->, black] (4.87, -1.5) -- (4.87, -.6) node[below = 38, right = -4, scale = .8]{$k_1$} node[below = 50, right = -4, scale = .8]{$z$}; 
		
		\draw[->, black, dashed] (4.87, -.5) -- (4.87, .4) node[below = 14, left = -1, scale = .8]{$k_2$}; 
		
		\draw[->, black, dotted] (3.87, .5) -- (4.77, .5) node[left = 34, scale = .8]{$j_1$}; 
		
		\draw[->, black] (3.87, -.5) -- (4.77, -.5) node[left = 34, scale = .8]{$i_1$};

		\draw[->, black] (4.87, .5) -- (4.87, 1.5) node[above = 5, right = -4, scale = .8]{$k_3$};

		\draw[->, black, dashed] (4.957, -.45) -- (5.653, -.05) node[above = 17	, left = 5, scale = .8]{$j_2$}; 
		\draw[->, black, dashed] (4.957, .45) -- (5.653, .05) node[below = 17, left = 5, scale = .8]{$i_2$};

		\draw[->, black] (5.827, -.05) -- (6.567, -.45) node[right = 0, scale = .8]{$j_3$} node[right = 12, scale = .8]{$x$}; 
		\draw[->, black, dotted] (5.827, .05) -- (6.567, .45) node[right = 0, scale = .8]{$i_3$} node[right = 12, scale = .8]{$y$};

		\filldraw[fill=white, draw=black] (5.74, 0) circle [radius=.1] node[scale = .7, right = 3]{$w$}  node[scale = .7, left = 15]{$q^{i_2} v$}  node[scale = .7, below = 20]{$v$};
		
		\draw[black] (5.65, -.42) -- (5.83, -.42) -- (5.83, -.62) -- (5.65, -.62) -- (5.65, -.42); 
		
		\draw[black] (4.98, .14) -- (5.4, .14) -- (5.4, -.14) -- (4.98, -.14) -- (4.98, .14); 
		
		\filldraw[fill=gray!50!white, draw=black] (4.87, .5) circle [radius=.1] node[scale = .7, above = 6, left = 1]{$\chi_s$};
		
		\filldraw[fill=gray!50!white, draw=black] (4.87, -.5) circle [radius=.1] node[scale = .7, below	 = 7, left = 1]{$\chi_s$};
		
		\filldraw[fill=white, draw=black] (2.87, 0) circle [radius=0] node{$=$};

		\end{tikzpicture}
		
	\end{center}

	\caption{\label{equationpaths} The vertex interpretation of the Yang-Baxter relation is depicted above; here, the indices along the dashed arrows (which are $i_2, j_2, k_2$) are summed over and the remaining indices (which are $i_1, j_1, k_1, i_3, j_3, k_3$) remain fixed. The dynamical parameters of the form $q^m v$ are boxed above and appear in Proposition \ref{sdynamicalequation1}.}
\end{figure}
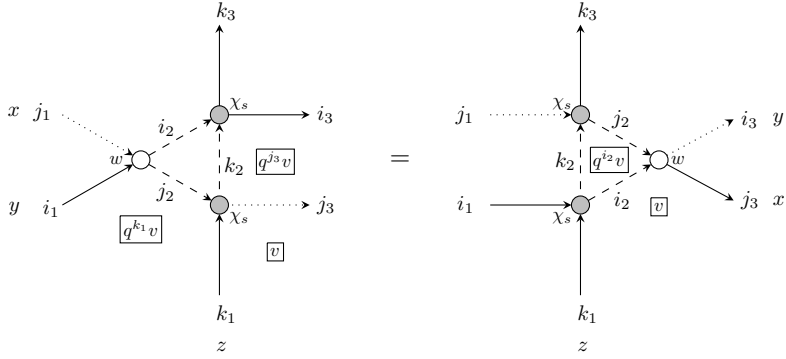

The Yang-Baxter equation \eqref{equationspin1} is sometimes diagrammatically interpreted as ``moving'' a line through a cross; see Figure \ref{equationpaths} for a depiction.\footnote{Although arbitrarily many arrows are allowed along vertical edges through shaded vertices, no edge in \Cref{equationpaths} accommodates more than one arrow. The reason for this is that \Cref{equationpaths} will be reused to depict \eqref{dynamicalequationspin1}, where at most one arrow is allowed per edge. } There, the unshaded vertices are evaluated with respect to the weight $w$ and the shaded vertices with respect to $\chi_s$.

Now, we would like to ``stochasticize'' the weights $w$, that is, produce weights $S(i_1, j_1; i_2, j_2)$, which are (manageable) deformations of the original $w$ weights that are also stochastic, meaning that $\sum_{i_2, j_2} S (i_1, j_1; i_2, j_2) = 1$ for each fixed $i_1, j_1 \in \{ 0, 1 \}$. One should also impose the nonnegativity constraint $S (i_1, j_1; i_2, j_2) \ge 0$ for all $(i_1, j_1; i_2, j_2)$, but we will typically not mention this here since in the examples we consider it will be possible to specialize the underlying parameters in such a way to ensure nonnegativity of the stochasticized weights.

As it happens, the weights $w$ given by Definition \ref{spin12spins} are already stochastic. However, this is not a general property of solutions to the Yang-Baxter equation and we would therefore like to find a definition of these $S$ weights that can be shown to be stochastic only through (some variant of) the Yang-Baxter equation \eqref{equationspin1} and the spin conservation property. 

This is indeed doable. To that end, first observe that the spin conservation property implies that $w (i_1, j_1; 0, 0)$ is equal to zero unless $i_1 = 0 = j_1$. In general, an outgoing pair $(i_2, j_2)$ that admits a unique incoming pair $(i_1, j_1)$ such that $(i_1, j_1; i_2, j_2)$ is nonzero will be called a \emph{frozen boundary condition}; the stochasticization procedure to be presented in Section \ref{DefinitionWeights} will not fully assume spin conservation but rather the existence of a frozen boundary, which in this case is $(0, 0)$. Thus, if we insert $i_3 = j_3 = 0$ in \Cref{equationspin}, there is at most one nonzero summand on the right side of \eqref{equationspin1} (which arises when $i_2 = 0 = j_2$ and $k_2 = k_1 + i_1$), as depicted in Figure \ref{vertexmodelfrozen}. This implies the following corollary.

\begin{cor}
	
	\label{equationspin2} 
	
	Fix $x, y, z \in \mathbb{C}$; $i_1, j_1 \in \{ 0, 1\}$ and $k \in \mathbb{Z}_{\ge 0}$. If $\chi_s \big( k, i_1; k + i_1, 0 \b| y, z \big)$ and $\chi_s \big( k + i_1, j_1; k + i_1 + j_1, 0 \b| x, z \big)$ are nonzero, then 
	\begin{flalign} 
	\label{equationspin3}
	\begin{aligned}
	& \displaystyle\sum_{i_2, j_2 \in \{ 0, 1 \}}   \displaystyle\frac{w \big( i_1, j_1; i_2, j_2 \b| x, y \big) \chi_s \big( k, j_2; k + j_2, 0 \b| x, z \big)  \chi_s \big( k + j_2, i_2; k + i_2 + j_2, 0 \b| y, z \big)}{ \chi_s \big( k, i_1; k + i_1, 0 \b| y, z \big)  \chi_s \big( k + i_1, j_1; k + i_1 + j_1, 0 \b| x, z \big)  w \big( 0, 0; 0, 0 \b| x, y \big)} = 1. 
	\end{aligned}
	\end{flalign}

\end{cor}

In view of \Cref{equationspin2} (and the fact that $w \big( 0, 0; 0, 0 \b| x, y \big) = 1$), we have the following definition.

\begin{definition}
	
	\label{stochastic1}
	
	For each $x, y, s \in \mathbb{C}$; $i_1, j_1, i_2, j_2 \in \{ 0, 1 \}$; and $k \in \mathbb{Z}_{\ge 0}$, define the \emph{stochasticized weight} 
	\begin{flalign}
	\label{si1j1i2j2vertex}
	\begin{aligned}
	S & \big( i_1, j_1; i_2, j_2 \b| x, y; s; k, z \big) \\
	& = w \big( i_1, j_1; i_2, j_2 \b| x, y \big)\displaystyle\frac{ \chi_s \big(k, j_2; k + j_2, 0 \b| x, z \big) }{\chi_s \big(k + i_1, j_1; k + i_1 + j_1	, 0 \b| x, z \big)} \displaystyle\frac{ \chi_s \big(k + j_2, i_2; k + i_2 + j_2, 0 \b| y, z \big)}{\chi_s \big( k, i_1; k + i_1, 0 \b| y, z \big)}.
	\end{aligned}
	\end{flalign}
	
	\noindent The fact that the $S$ weights are indeed stochastic follows from \Cref{equationspin2}. 
	
\end{definition}

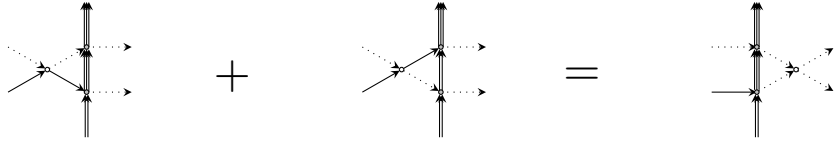
\begin{figure}

	\begin{center}

		\begin{tikzpicture}[
		>=stealth,
		auto,
		style={
			scale = .6
		}
		]

		\draw[->, black] (-.87, .5) -- (-.043, .975); 
		\draw[->, black, dotted] (-.87, 1.5) -- (-.043, 1.025);
		
		\draw[->, black, dotted] (.043, .975) -- (.827, .525); 
		\draw[->, black] (.043, 1.025) -- (.827, 1.475); 
		
		\draw[->, black] (.84, -.5) -- (.84, .45); 
		\draw[->, black] (.9, -.5) -- (.9, .45);
		
		\draw[->, black] (.84, .5) -- (.84, 1.45);
		\draw[->, black] (.9, .5) -- (.9, 1.45); 
		
		\draw[->, black, dotted] (.87, 1.5) -- (1.87, 1.5); 
		
		\draw[->, black, dotted] (.87, .5) -- (1.87, .5);

		\draw[->, black] (.82, 1.5) -- (.82, 2.5); 
		\draw[->, black] (.87, 1.5) -- (.87, 2.5);  
		\draw[->, black] (.92, 1.5) -- (.92, 2.5); 
	
		\filldraw[fill=white, draw=black] (0, 1) circle [radius=.05];
		
		\filldraw[fill=gray!50!white, draw=black] (.87, 1.5) circle [radius=.05];
		
		\filldraw[fill=gray!50!white, draw=black] (.87, .5) circle [radius=.05];

		\draw[black] (4, .85) circle[radius=0] node[scale = 1.8]{$=$};

		\draw[black] (-3.75, .85) circle[radius=0] node[scale = 1.8]{$+$};

		\draw[->, black] (-8.72, .5) -- ( -7.893, .975); 
		\draw[->, black, dotted] (-8.72, 1.5) -- (-7.893, 1.025);
		
		\draw[->, black] (-7.807, .975) -- (-7.023, .525); 
		\draw[->, black, dotted] (-7.807, 1.025) -- (-7.023, 1.475); 
		
		\draw[->, black] (-7.01, -.5) -- (-7.01, .45); 
		\draw[->, black] (-6.95, -.5) -- (-6.95, .45);
		
		\draw[->, black] (-7.03, .5) -- (-7.03, 1.45);
		\draw[->, black] (-6.98, .5) -- (-6.98, 1.45); 
		\draw[->, black] (-6.93, .5) -- (-6.93, 1.45);  
		
		\draw[->, black, dotted] (-6.98, 1.5) -- (-5.98, 1.5); 
		
		\draw[->, black, dotted] (-6.98, .5) -- (-5.98, .5);

		\draw[->, black] (-7.03, 1.5) -- (-7.03, 2.5); 
		\draw[->, black] (-6.98, 1.5) -- (-6.98, 2.5); 
		\draw[->, black] (-6.93, 1.5) -- (-6.93, 2.5); 
		
		\filldraw[fill=white, draw=black] (-7.85, 1) circle [radius=.05];
		
		\filldraw[fill=gray!50!white, draw=black] (-6.98, 1.5) circle [radius=.05];
		
		\filldraw[fill=gray!50!white, draw=black] (-6.98, .5) circle [radius=.05];

		\draw[->, black] (7.84, -.5) -- (7.84, .45); 
		\draw[->, black] (7.9, -.5) -- (7.9, .45);

		\draw[->, black] (7.82, .55) -- (7.82, 1.45); 
		\draw[->, black] (7.87, .55) -- (7.87, 1.45); 
		\draw[->, black] (7.92, .55) -- (7.92, 1.45);

		\draw[->, black, dotted] (6.87, 1.5) -- (7.87, 1.5); 
		
		\draw[->, black] (6.87, .5) -- (7.87, .5);

		\draw[->, black] (7.82, 1.5) -- (7.82, 2.5); 
		\draw[->, black] (7.87, 1.5) -- (7.87, 2.5); 
		\draw[->, black] (7.92, 1.5) -- (7.92, 2.5);

		\draw[->, black, dotted] (7.913, .525) -- (8.697, .975); 
		\draw[->, black, dotted] (7.913, 1.475) -- (8.697, 1.025);

		\draw[->, black, dotted] (8.783, .975) -- (9.567, .525); 
		\draw[->, black, dotted] (8.783, 1.025) -- (9.567, 1.475);

		\filldraw[fill=white, draw=black] (8.74, 1) circle [radius=.05];
		
		\filldraw[fill=gray!50!white, draw=black] (7.87, 1.5) circle [radius=.05];
		
		\filldraw[fill=gray!50!white, draw=black] (7.87, .5) circle [radius=.05];

		\end{tikzpicture}
		
	\end{center}

	\caption{\label{vertexmodelfrozen} An instance of the Yang-Baxter relation with a frozen right boundary is depicted above. }
\end{figure}

The following proposition explicitly evaluates the $S$ weights from \Cref{stochastic1}; its proof follows from inserting the weights from \Cref{spin12spins} into \eqref{si1j1i2j2vertex}.

\begin{prop}
	
	\label{dynamicalweights1}
	
	Follow the notation of Definition \ref{stochastic1}, and denote $v = s q^k z^{-1}$. Then, the stochasticized weights $S ( i_1, j_1; i_2, j_2) = S \big( i_1, j_1; i_2, j_2 \b| x, y; v \big) = S \big( i_1, j_1; i_2, j_2 \b| x, y; s; k, z \big)$ are given by
	\begin{flalign}
	\label{s01dynamicalweights}
	\begin{aligned}
	& S (1, 0; 1, 0)  =  \displaystyle\frac{q(x - y) (1 - v x)}{(x - qy) (1 - qv x)} ; \qquad S (1, 0; 0, 1) = \displaystyle\frac{(1 - q) x (1 - q v y)}{(x - qy) (1 - qv x)}; \\
	& S (0, 1; 1, 0)  =  \displaystyle\frac{(1 - q) y (1  - v x)}{(x - qy) (1 - v y)} ; \qquad \quad S (0, 1; 0, 1) = \displaystyle\frac{(x - y) (1 - qv y)}{(x - qy) (1 - v y)} \\
	& \qquad \qquad \qquad \qquad \quad S (0, 0; 0, 0) = 1 = S (1, 1; 1, 1),
	\end{aligned}
	\end{flalign}
	
	\noindent and $S(i_1, j_1; i_2, j_2) = 0$ if $(i_1, j_1; i_2, j_2)$ is not of the above form. 
	
\end{prop}

Observe that these weights depend on $s$, $k$, and $z$ only through $v$; in particular, if $v = 0$ then the $S$ weights reduce to the original $w$ weights from \Cref{spin12spins}. In the next section we will explain in what sense the $S$ weights still satisfy the Yang-Baxter equation for nonzero $v$.

\subsection{Stochasticizing a Path Ensemble}

\label{DynamicalVertex1} 

As seen by \Cref{dynamicalweights1}, the stochasticization procedure deforms the original $w$ weights by an additional parameter $v$. If we wish to stochasticize a directed path ensemble, we therefore require an assignment of this parameter $v$ to each vertex of our domain $\mathcal{D}$. Under an arbitrary such assignment, the model will not retain its integrability, that is, the Yang-Baxter equation will no longer hold. However, in this section we will describe a specific assignment that preserves the Yang-Baxter equation.

To that end, let us begin with a more diagrammatic interpretation of \Cref{stochastic1}, which is depicted in \Cref{arrowstochastic}. Suppose we would like to stochasticize some vertex $u = (a, b)$ with arrow configuration $(i_1, j_1; i_2, j_2)$. We draw this vertex as unshaded in the plane; see the left side of \Cref{arrowstochastic}, in which $(i_1, j_1; i_2, j_2) = (1, 0; 0, 1)$. We then attach two shaded vertices to $u$, one above it (denoted by $u_2$) and one to its right (denoted by $u_1$). As depicted in \Cref{arrowstochastic}, $k$ arrows (which are drawn as dashed) enter through $u_1$ from the bottom, and any arrows exiting $u$ are ``collected'' by this dashed curve; in particular, this means that $k + j_2$ arrows are directed from $u_1$ to $u_2$, and $k + i_2 + j_2$ exit through $u_2$. We call this dashed curve the \emph{stochasticization curve}, to which we assign the rapidity parameter $z$.

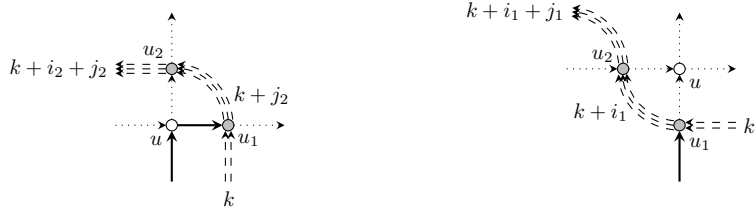
\begin{figure}

	\begin{center}

		\begin{tikzpicture}[
		>=stealth,
		auto,
		style={
			scale = .75
		}
		]

		\draw[->, black, thick] (0, -1) -- (0, -.1);
		\draw[->, black, thick] (.1, 0) -- (.9, 0);
		
		\draw[->, black, dotted] (-1, 0) -- (-.1, 0);
		\draw[->, black, dotted] (0, .1) -- (0, .9); 
		
		\draw[->, black, dotted] (0, 1.1) -- (0, 2); 
		\draw[->, black, dotted] (1.1, 0) -- (2, 0); 
		
		\draw[->, black, dashed] (.95, -1) -- (.95, -.1); 
		\draw[->, black, dashed] (1.05, -1) -- (1.05, -.1) node[scale = .8, left = 1, below = 25]{$k$}; 
		
		\draw[->, black, dashed] (-.1, .92) -- (-1, .92); 
		\draw[->, black, dashed] (-.1, 1) -- (-1, 1) node[scale = .8, left = 0]{$k + i_2 + j_2$};
		\draw[->, black, dashed] (-.1, 1.08) -- (-1, 1.08);
		
		\draw[->, black, dashed] (.92, .1) arc(5:85:.92);
		\draw[->, black, dashed] (1, .1) arc(5:85:1) node[scale = .8, right = 40, below = 5]{$k + j_2$};
		\draw[->, black, dashed] (1.08, .1) arc(5:85:1.08);

		\filldraw[fill=white, draw=black] (0, 0) circle [radius = .1] node[scale = .8, below = 7, left = 1]{$u$};
		\filldraw[fill=gray!50!white, draw=black] (0, 1) circle [radius = .1] node[scale = .8, above = 8, left = 0]{$u_2$};
		\filldraw[fill=gray!50!white, draw=black] (1, 0) circle [radius = .1] node[scale = .8, below = 7, right = 1]{$u_1$};

		\draw[->, black, dotted] (9.1, 1) -- (10, 1);
		\draw[->, black, dotted] (9, .1) -- (9, .9);

		\draw[->, black, dotted] (9, 1.1) -- (9, 2); 
		\draw[->, black, dotted] (8.1, 1) -- (8.9, 1); 
		\draw[->, black, dotted] (7, 1) -- (7.9, 1);
		
		\draw[->, black, thick] (9, -1) -- (9, -.1); 
		
		\draw[->, black, dashed] (10, -.05) -- (9.1, -.05) node[scale = .8, above = 1, right = 24]{$k$}; 
		\draw[->, black, dashed] (10, .05) -- (9.1, .05);

		\draw[->, black, dashed] (7.92, 1.1) arc(5:85:.92);
		\draw[->, black, dashed] (8, 1.1) arc(5:85:1) node[scale = .8, left = 0]{$k + i_1 + j_1$};
		\draw[->, black, dashed] (8.08, 1.1) arc(5:85:1.08);

		\draw[->, black, dashed] (8.9, .08) arc(265:185:.92);
		\draw[->, black, dashed] (8.9, 0) arc(265:185:1) node[scale = .8, below = 18, left = -6]{$k + i_1$};
		\draw[->, black, dashed] (8.9, -.08) arc(265:185:1.08);

		\filldraw[fill=white, draw=black] (9,1) circle [radius = .1] node[scale = .8, below = 7, right = 1]{$u$};
		\filldraw[fill=gray!50!white, draw=black] (8, 1) circle [radius = .1] node[scale = .8, above = 6, left = 1]{$u_2$};
		\filldraw[fill=gray!50!white, draw=black] (9, 0) circle [radius = .1] node[scale = .8, below = 9, right = 1]{$u_1$};

		\end{tikzpicture}
		
	\end{center}
	
	\caption{\label{arrowstochastic} The process of the (dashed) stochasticization curve being ``pushed'' through a vertex is depicted above.  }
\end{figure}

We then ``push'' the stochasticization curve through the vertex $u$, as shown on the right side of \Cref{arrowstochastic}. Due to the boundary conditions, all arrows that would have otherwise entered $u$ are now collected by the stochasticization curve; in particular, no arrows enter or exit the vertex $u$ after the pushing procedure. 

If we weight all unshaded vertices by $w$ and all shaded vertices by $\chi_s$, then \Cref{stochastic1} states that the stochasticized weight $S(i_1, j_1; i_2, j_2)$ of the original vertex $u$ (with arrow configuration $(i_1, j_1; i_2, j_2)$) is the weight of the left diagram in \Cref{arrowstochastic} divided by that of the right one. 

This weight depends on the number $k$ of arrows that are originally placed in the stochasticization curve. Thus, if we would like to stochasticize each vertex in a path ensemble (in a domain $\mathcal{D}$), we require a way to assign $k$ to each vertex of $\mathcal{D}$. This can be done by applying the same pushing procedure to a directed path ensemble instead of to a single vertex; see \Cref{arrowstochasticvertices}.

\begin{figure}

	\begin{center}

		\begin{tikzpicture}[
		>=stealth,
		auto,
		style={
			scale = .6
		}
		]

		\draw[->, black, thick] (0, -1) -- (0, -.1);
		
		\draw[->, black, thick] (1, -1) -- (1, -.1);
		\draw[->, black, thick] (1, .1) -- (1, .9);
		\draw[->, black, thick] (1, 1.1) -- (1, 1.9);
		
		\draw[->, black, thick] (2, .1) -- (2, .9);
		
		\draw[->, black, thick] (3, 1.1) -- (3, 1.9);
		
		\draw[->, black, thick] (4, -1) -- (4, -.1); 
		\draw[->, black, thick] (4, .1) -- (4, .9);

		\draw[->, black, thick] (.1, 0) -- (.9, 0);
		
		\draw[->, black, thick] (1.1, 0) -- (1.9, 0);
		
		\draw[->, black, thick] (2.1, 1) -- (2.9, 1);
		
		\draw[->, black, thick] (4.1, 1) -- (4.9, 1);

		\draw[->, black, dashed] (4.95, -.9) -- (4.95, -.1) node[scale = .8, right = 1, below = 18]{$k$}; 
		\draw[->, black, dashed] (5.05, -.9) -- (5.05, -.1);
		
		\draw[->, black, dashed] (4.95, .1) -- (4.95, .9); 
		\draw[->, black, dashed] (5.05, .1) -- (5.05, .9); 
		
		\draw[->, black, dashed] (4.92, 1.1) arc(5:85:.92);
		\draw[->, black, dashed] (5, 1.1) arc(5:85:1);
		\draw[->, black, dashed] (5.08, 1.1) arc(5:85:1.08);
		
		\draw[->, black, dashed] (3.9, 1.92) -- (3.1, 1.92); 
		\draw[->, black, dashed] (3.9, 2) -- (3.1, 2);
		\draw[->, black, dashed] (3.9, 2.08) -- (3.1, 2.08);
		
		\draw[->, black, dashed] (2.9, 1.91) -- (2.1, 1.91); 
		\draw[->, black, dashed] (2.9, 1.97) -- (2.1, 1.97);
		\draw[->, black, dashed] (2.9, 2.03) -- (2.1, 2.03);
		\draw[->, black, dashed] (2.9, 2.09) -- (2.1, 2.09);
		
		\draw[->, black, dashed] (1.9, 1.91) -- (1.1, 1.91); 
		\draw[->, black, dashed] (1.9, 1.97) -- (1.1, 1.97);
		\draw[->, black, dashed] (1.9, 2.03) -- (1.1, 2.03);
		\draw[->, black, dashed] (1.9, 2.09) -- (1.1, 2.09);

		\draw[->, black, dashed] (.9, 1.9) -- (.1, 1.9); 
		\draw[->, black, dashed] (.9, 1.95) -- (.1, 1.95);
		\draw[->, black, dashed] (.9, 2) -- (.1, 2);
		\draw[->, black, dashed] (.9, 2.05) -- (.1, 2.05);
		\draw[->, black, dashed] (.9, 2.1) -- (.1, 2.1);

		\draw[->, black, dashed] (-.1, 1.9) -- (-.9, 1.9); 
		\draw[->, black, dashed] (-.1, 1.95) -- (-.9, 1.95);
		\draw[->, black, dashed] (-.1, 2) -- (-.9, 2);
		\draw[->, black, dashed] (-.1, 2.05) -- (-.9, 2.05);
		\draw[->, black, dashed] (-.1, 2.1) -- (-.9, 2.1);

		\filldraw[fill=white, draw=black] (0, 0) circle [radius = .1] node[scale = .6, below = 5, left = 3]{$(1, 1)$};
		\filldraw[fill=white, draw=black] (0, 1) circle [radius = .1];
		\filldraw[fill=white, draw=black] (1, 0) circle [radius = .1];
		\filldraw[fill=white, draw=black] (1, 1) circle [radius = .1];
		\filldraw[fill=white, draw=black] (2, 0) circle [radius = .1];
		\filldraw[fill=white, draw=black] (2, 1) circle [radius = .1];
		\filldraw[fill=white, draw=black] (3, 0) circle [radius = .1];
		\filldraw[fill=white, draw=black] (3, 1) circle [radius = .1];
		\filldraw[fill=white, draw=black] (4, 0) circle [radius = .1];
		\filldraw[fill=white, draw=black] (4, 1) circle [radius = .1];
		
		\filldraw[fill=gray!50!white, draw=black] (5, 0) circle [radius = .1];
		\filldraw[fill=gray!50!white, draw=black] (0, 2) circle [radius = .1];
		
		\filldraw[fill=gray!50!white, draw=black] (5, 1) circle [radius = .1];
		\filldraw[fill=gray!50!white, draw=black] (1, 2) circle [radius = .1];
		
		\filldraw[fill=gray!50!white, draw=black] (2, 2) circle [radius = .1];
		\filldraw[fill=gray!50!white, draw=black] (3, 2) circle [radius = .1];
		
		\filldraw[fill=gray!50!white, draw=black] (4, 2) circle [radius = .1];

		\draw[->, black, thick] (9, -1) -- (9, -.1);
		
		\draw[->, black, thick] (10, -1) -- (10, -.1);
		\draw[->, black, thick] (10, .1) -- (10, .9);
		\draw[->, black, thick] (10, 1.1) -- (10, 1.9);
		
		\draw[->, black, thick] (11, .1) -- (11, .9);
		
		\draw[->, black, thick] (12, 1.1) -- (12, 1.9);
		
		\draw[->, black, thick] (13, -1) -- (13, -.1); 
		\draw[->, black, thick] (13, .1) -- (13, .4);

		\draw[->, black, thick] (9.1, 0) -- (9.9, 0);
		
		\draw[->, black, thick] (10.1, 0) -- (10.9, 0);
		
		\draw[->, black, thick] (11.1, 1) -- (11.9, 1);

		\draw[->, black, dashed] (13.95, -.9) -- (13.95, -.1); 
		\draw[->, black, dashed] (14.05, -.9) -- (14.05, -.1);

		\draw[->, black, dashed] (13, .58) arc(265:185:.42);
		\draw[->, black, dashed] (13, .5) arc(265:185:.5);
		\draw[->, black, dashed] (13, .42) arc(265:185:.58);

		\draw[->, black, dashed] (13.95, .1) arc(50:88:1.45);
		\draw[->, black, dashed] (14.05, .1) arc(47:89:1.45);

		\draw[->, black, dashed] (12.42, 1.1) arc(2:40:1.42);
		\draw[->, black, dashed] (12.5, 1.1) arc(2:43:1.45);
		\draw[->, black, dashed] (12.58, 1.1) arc(2:46:1.48);

		\draw[->, black, dashed] (11.9, 1.91) -- (11.1, 1.91); 
		\draw[->, black, dashed] (11.9, 1.97) -- (11.1, 1.97);
		\draw[->, black, dashed] (11.9, 2.03) -- (11.1, 2.03);
		\draw[->, black, dashed] (11.9, 2.09) -- (11.1, 2.09);
		
		\draw[->, black, dashed] (10.9, 1.91) -- (10.1, 1.91); 
		\draw[->, black, dashed] (10.9, 1.97) -- (10.1, 1.97);
		\draw[->, black, dashed] (10.9, 2.03) -- (10.1, 2.03);
		\draw[->, black, dashed] (10.9, 2.09) -- (10.1, 2.09);

		\draw[->, black, dashed] (9.9, 1.9) -- (9.1, 1.9); 
		\draw[->, black, dashed] (9.9, 1.95) -- (9.1, 1.95);
		\draw[->, black, dashed] (9.9, 2) -- (9.1, 2);
		\draw[->, black, dashed] (9.9, 2.05) -- (9.1, 2.05);
		\draw[->, black, dashed] (9.9, 2.1) -- (9.1, 2.1);

		\draw[->, black, dashed] (8.9, 1.9) -- (8.1, 1.9); 
		\draw[->, black, dashed] (8.9, 1.95) -- (8.1, 1.95);
		\draw[->, black, dashed] (8.9, 2) -- (8.1, 2);
		\draw[->, black, dashed] (8.9, 2.05) -- (8.1, 2.05);
		\draw[->, black, dashed] (8.9, 2.1) -- (8.1, 2.1);

		\filldraw[fill=white, draw=black] (9, 0) circle [radius = .1];
		\filldraw[fill=white, draw=black] (9, 1) circle [radius = .1];
		\filldraw[fill=white, draw=black] (10, 0) circle [radius = .1];
		\filldraw[fill=white, draw=black] (10, 1) circle [radius = .1];
		\filldraw[fill=white, draw=black] (11, 0) circle [radius = .1];
		\filldraw[fill=white, draw=black] (11, 1) circle [radius = .1];
		\filldraw[fill=white, draw=black] (12, 0) circle [radius = .1];
		\filldraw[fill=white, draw=black] (12, 1) circle [radius = .1];
		\filldraw[fill=white, draw=black] (13, 0) circle [radius = .1];
		\filldraw[fill=white, draw=black] (13, 1) circle [radius = .1];
		
		\filldraw[fill=gray!50!white, draw=black] (14, 0) circle [radius = .1];
		\filldraw[fill=gray!50!white, draw=black] (9, 2) circle [radius = .1];
		
		\filldraw[fill=gray!50!white, draw=black] (10, 2) circle [radius = .1];
		
		\filldraw[fill=gray!50!white, draw=black] (11, 2) circle [radius = .1];
		\filldraw[fill=gray!50!white, draw=black] (12, 2) circle [radius = .1];

		\filldraw[fill=gray!50!white, draw=black] (13, .5) circle [radius = .1];

		\filldraw[fill=gray!50!white, draw=black] (12.5, 1) circle [radius = .1];

		\draw[->, black, thick] (18, -1) -- (18, -.1);
		
		\draw[->, black, thick] (19, -1) -- (19, -.1);
		\draw[->, black, thick] (19, .1) -- (19, .9);
		\draw[->, black, thick] (19, 1.1) -- (19, 1.9);
		
		\draw[->, black, thick] (20, .1) -- (20, .9);

		\draw[->, black, thick] (22, -1) -- (22, -.1); 
		\draw[->, black, thick] (22, .1) -- (22, .4);

		\draw[->, black, thick] (18.1, 0) -- (18.9, 0);
		
		\draw[->, black, thick] (19.1, 0) -- (19.9, 0);
		
		\draw[->, black, thick] (20.1, 1) -- (20.4, 1);

		\draw[->, black, dashed] (22.95, -.9) -- (22.95, -.1); 
		\draw[->, black, dashed] (23.05, -.9) -- (23.05, -.1);
		
		\draw[->, black, dashed] (19.9, 1.91) -- (19.1, 1.91); 
		\draw[->, black, dashed] (19.9, 1.97) -- (19.1, 1.97);
		\draw[->, black, dashed] (19.9, 2.03) -- (19.1, 2.03);
		\draw[->, black, dashed] (19.9, 2.09) -- (19.1, 2.09);

		\draw[->, black, dashed] (17.9, 1.9) -- (17.1, 1.9); 
		\draw[->, black, dashed] (17.9, 1.95) -- (17.1, 1.95);
		\draw[->, black, dashed] (17.9, 2) -- (17.1, 2);
		\draw[->, black, dashed] (17.9, 2.05) -- (17.1, 2.05);
		\draw[->, black, dashed] (17.9, 2.1) -- (17.1, 2.1);

		\draw[->, black, dashed] (18.9, 1.9) -- (18.1, 1.9); 
		\draw[->, black, dashed] (18.9, 1.95) -- (18.1, 1.95);
		\draw[->, black, dashed] (18.9, 2) -- (18.1, 2);
		\draw[->, black, dashed] (18.9, 2.05) -- (18.1, 2.05);
		\draw[->, black, dashed] (18.9, 2.1) -- (18.1, 2.1);

		\draw[->, black, dashed] (21, .58) arc(265:185:.42);
		\draw[->, black, dashed] (21, .5) arc(265:185:.5);
		\draw[->, black, dashed] (21, .42) arc(265:185:.58);

		\draw[->, black, dashed] (22.95, .1) arc(50:88:1.45);
		\draw[->, black, dashed] (23.05, .1) arc(47:89:1.45);
		
		\draw[->, black, dashed] (21.9, .58) -- (21.1, .58); 
		\draw[->, black, dashed] (21.9, .5) -- (21.1, .5); 
		\draw[->, black, dashed] (21.9, .42) -- (21.1, .42); 
		
		\draw[->, black, dashed] (20.41, 1.1) arc(2:39:1.41);
		\draw[->, black, dashed] (20.47, 1.1) arc(2:42:1.44);
		\draw[->, black, dashed] (20.53, 1.1) arc(2:45:1.47);
		\draw[->, black, dashed] (20.59, 1.1) arc(2:48:1.5);

		\filldraw[fill=white, draw=black] (18, 0) circle [radius = .1];
		\filldraw[fill=white, draw=black] (18, 1) circle [radius = .1];
		\filldraw[fill=white, draw=black] (19, 0) circle [radius = .1];
		\filldraw[fill=white, draw=black] (19, 1) circle [radius = .1];
		\filldraw[fill=white, draw=black] (20, 0) circle [radius = .1];
		\filldraw[fill=white, draw=black] (20, 1) circle [radius = .1];
		\filldraw[fill=white, draw=black] (21, 0) circle [radius = .1];
		\filldraw[fill=white, draw=black] (21, 1) circle [radius = .1];
		\filldraw[fill=white, draw=black] (22, 0) circle [radius = .1];
		\filldraw[fill=white, draw=black] (22, 1) circle [radius = .1];
		
		\filldraw[fill=gray!50!white, draw=black] (23, 0) circle [radius = .1];
		\filldraw[fill=gray!50!white, draw=black] (18, 2) circle [radius = .1];
		
		\filldraw[fill=gray!50!white, draw=black] (19, 2) circle [radius = .1];
		
		\filldraw[fill=gray!50!white, draw=black] (20, 2) circle [radius = .1];

		\filldraw[fill=gray!50!white, draw=black] (22, .5) circle [radius = .1];

		\filldraw[fill=gray!50!white, draw=black] (20.5, 1) circle [radius = .1];

		\filldraw[fill=gray!50!white, draw=black] (21, .5) circle [radius = .1];

		\draw[->, black, thick] (0, -6) -- (0, -5.1);
		
		\draw[->, black, thick] (1, -6) -- (1, -5.1);
		\draw[->, black, thick] (1, -4.9) -- (1, -4.6);
		
		\draw[->, black, thick] (2, -4.9) -- (2, -4.6);
		
		\draw[->, black, thick] (4, -4.9) -- (4, -4.6);
		\draw[->, black, thick] (4, -6) -- (4, -5.1);

		\draw[->, black, thick] (.1, -5) -- (.9, -5);
		
		\draw[->, black, thick] (1.1, -5) -- (1.9, -5);

		\draw[->, black, dashed] (4.95, -5.9) -- (4.95, -5.1); 
		\draw[->, black, dashed] (5.05, -5.9) -- (5.05, -5.1);
		
		\draw[->, black, dashed] (2.9, -4.58) -- (2.1, -4.58);
		\draw[->, black, dashed] (2.9, -4.5) -- (2.1, -4.5);
		\draw[->, black, dashed] (2.9, -4.42) -- (2.1, -4.42);
		
		\draw[->, black, dashed] (1.9, -4.59) -- (1.1, -4.59); 
		\draw[->, black, dashed] (1.9, -4.53) -- (1.1, -4.53);
		\draw[->, black, dashed] (1.9, -4.47) -- (1.1, -4.47);
		\draw[->, black, dashed] (1.9, -4.41) -- (1.1, -4.41);
		
		\draw[->, black, dashed] (.9, -4.6) -- (.1, -4.6); 
		\draw[->, black, dashed] (.9, -4.55) -- (.1, -4.55);
		\draw[->, black, dashed] (.9, -4.5) -- (.1, -4.5);
		\draw[->, black, dashed] (.9, -4.45) -- (.1, -4.45);
		\draw[->, black, dashed] (.9, -4.4) -- (.1, -4.4);
		
		\draw[->, black, dashed] (0, -4.4) arc(265:185:.4);
		\draw[->, black, dashed] (0, -4.45) arc(265:185:.43);
		\draw[->, black, dashed] (0, -4.5) arc(265:185:.46);
		\draw[->, black, dashed] (0, -4.55) arc(265:185:.49);
		\draw[->, black, dashed] (0, -4.6) arc(265:185:.52);

		\draw[->, black, dashed] (4.95, -4.9) arc(50:88:1.45);
		\draw[->, black, dashed] (5.05, -4.9) arc(47:89:1.45);
		
		\draw[->, black, dashed] (3.9, -4.42) -- (3.1, -4.42); 
		\draw[->, black, dashed] (3.9, -4.5) -- (3.1, -4.5); 
		\draw[->, black, dashed] (3.9, -4.58) -- (3.1, -4.58); 
		
		\draw[->, black, dashed] (-.6, -3.9) arc(2:39:1.4);
		\draw[->, black, dashed] (-.55, -3.9) arc(2:42:1.43);
		\draw[->, black, dashed] (-.5, -3.9) arc(2:45:1.46);
		\draw[->, black, dashed] (-.45, -3.9) arc(2:48:1.49);
		\draw[->, black, dashed] (-.4, -3.9) arc(2:48:1.52);

		\filldraw[fill=white, draw=black] (0, -5) circle [radius = .1];
		\filldraw[fill=white, draw=black] (0, -4) circle [radius = .1];
		\filldraw[fill=white, draw=black] (1, -5) circle [radius = .1];
		\filldraw[fill=white, draw=black] (1, -4) circle [radius = .1];
		\filldraw[fill=white, draw=black] (2, -5) circle [radius = .1];
		\filldraw[fill=white, draw=black] (2, -4) circle [radius = .1];
		\filldraw[fill=white, draw=black] (3, -5) circle [radius = .1];
		\filldraw[fill=white, draw=black] (3, -4) circle [radius = .1];
		\filldraw[fill=white, draw=black] (4, -5) circle [radius = .1];
		\filldraw[fill=white, draw=black] (4, -4) circle [radius = .1];
		
		\filldraw[fill=gray!50!white, draw=black] (5, -5) circle [radius = .1];
		
		\filldraw[fill=gray!50!white, draw=black] (0, -4.5) circle [radius = .1];
		\filldraw[fill=gray!50!white, draw=black] (1, -4.5) circle [radius = .1];
		\filldraw[fill=gray!50!white, draw=black] (2, -4.5) circle [radius = .1];
		\filldraw[fill=gray!50!white, draw=black] (3, -4.5) circle [radius = .1];
		\filldraw[fill=gray!50!white, draw=black] (4, -4.5) circle [radius = .1];
		
		\filldraw[fill=gray!50!white, draw=black] (-.5, -4) circle [radius = .1];

		\draw[->, black, thick] (9, -6) -- (9, -5.1);
		
		\draw[->, black, thick] (10, -6) -- (10, -5.1);
		\draw[->, black, thick] (10, -4.9) -- (10, -4.6);
		
		\draw[->, black, thick] (11, -4.9) -- (11, -4.6);

		\draw[->, black, thick] (13, -6) -- (13, -5.6);

		\draw[->, black, thick] (9.1, -5) -- (9.9, -5);
		
		\draw[->, black, thick] (10.1, -5) -- (10.9, -5);

		\draw[->, black, dashed] (11.9, -4.58) -- (11.1, -4.58);
		\draw[->, black, dashed] (11.9, -4.5) -- (11.1, -4.5);
		\draw[->, black, dashed] (11.9, -4.42) -- (11.1, -4.42);
		
		\draw[->, black, dashed] (10.9, -4.59) -- (10.1, -4.59); 
		\draw[->, black, dashed] (10.9, -4.53) -- (10.1, -4.53);
		\draw[->, black, dashed] (10.9, -4.47) -- (10.1, -4.47);
		\draw[->, black, dashed] (10.9, -4.41) -- (10.1, -4.41);

		\draw[->, black, dashed] (9.9, -4.6) -- (9.1, -4.6); 
		\draw[->, black, dashed] (9.9, -4.55) -- (9.1, -4.55);
		\draw[->, black, dashed] (9.9, -4.5) -- (9.1, -4.5);
		\draw[->, black, dashed] (9.9, -4.45) -- (9.1, -4.45);
		\draw[->, black, dashed] (9.9, -4.4) -- (9.1, -4.4);

		\draw[->, black, dashed] (8.6, -3.9) -- (8.6, -3.1);
		\draw[->, black, dashed] (8.55, -3.9) -- (8.55, -3.1);
		\draw[->, black, dashed] (8.5, -3.9) -- (8.5, -3.1);
		\draw[->, black, dashed] (8.45, -3.9) -- (8.45, -3.1);
		\draw[->, black, dashed] (8.4, -3.9) -- (8.4, -3.1);

		\draw[->, black, dashed] (13, -5.42) arc(265:185:.42);
		\draw[->, black, dashed] (13, -5.5) arc(265:185:.5);
		\draw[->, black, dashed] (13, -5.58) arc(265:185:.58);

		\draw[->, black, dashed] (12.42, -5) arc(5:85:.42);
		\draw[->, black, dashed] (12.5, -5) arc(5:85:.5);
		\draw[->, black, dashed] (12.58, -5) arc(5:85:.58);

		\draw[->, black, dashed] (13.9, -5.45) -- (13.1, -5.45); 
		\draw[->, black, dashed] (13.9, -5.55) -- (13.1, -5.55);

		\draw[->, black, dashed] (9, -4.4) arc(265:185:.4);
		\draw[->, black, dashed] (9, -4.45) arc(265:185:.43);
		\draw[->, black, dashed] (9, -4.5) arc(265:185:.46);
		\draw[->, black, dashed] (9, -4.55) arc(265:185:.49);
		\draw[->, black, dashed] (9, -4.6) arc(265:185:.52);

		\filldraw[fill=white, draw=black] (9, -5) circle [radius = .1];
		\filldraw[fill=white, draw=black] (9, -4) circle [radius = .1];
		\filldraw[fill=white, draw=black] (10, -5) circle [radius = .1];
		\filldraw[fill=white, draw=black] (10, -4) circle [radius = .1];
		\filldraw[fill=white, draw=black] (11, -5) circle [radius = .1];
		\filldraw[fill=white, draw=black] (11, -4) circle [radius = .1];
		\filldraw[fill=white, draw=black] (12, -5) circle [radius = .1];
		\filldraw[fill=white, draw=black] (12, -4) circle [radius = .1];
		\filldraw[fill=white, draw=black] (13, -5) circle [radius = .1];
		\filldraw[fill=white, draw=black] (13, -4) circle [radius = .1];
		
		\filldraw[fill=gray!50!white, draw=black] (8.5, -4) circle [radius = .1];
		\filldraw[fill=gray!50!white, draw=black] (9, -4.5) circle [radius = .1];
		
		\filldraw[fill=gray!50!white, draw=black] (10, -4.5) circle [radius = .1];
		
		\filldraw[fill=gray!50!white, draw=black] (11, -4.5) circle [radius = .1];

		\filldraw[fill=gray!50!white, draw=black] (13, -5.5) circle [radius = .1];
		
		\filldraw[fill=gray!50!white, draw=black] (12.5, -5) circle [radius = .1];
		
		\filldraw[fill=gray!50!white, draw=black] (12, -4.5) circle [radius = .1];

		\draw[->, black, thick] (18, -6) -- (18, -5.6);
		
		\draw[->, black, thick] (19, -6) -- (19, -5.6);
		
		\draw[->, black, thick] (22, -6) -- (22, -5.6);

		\draw[->, black, dashed] (18.9, -5.59) -- (18.1, -5.59);
		\draw[->, black, dashed] (18.9, -5.53) -- (18.1, -5.53);
		\draw[->, black, dashed] (18.9, -5.47) -- (18.1, -5.47);
		\draw[->, black, dashed] (18.9, -5.41) -- (18.1, -5.41);
		
		\draw[->, black, dashed] (19.9, -5.58) -- (19.1, -5.58); 
		\draw[->, black, dashed] (19.9, -5.5) -- (19.1, -5.5);
		\draw[->, black, dashed] (19.9, -5.42) -- (19.1, -5.42);
		
		\draw[->, black, dashed] (20.9, -5.58) -- (20.1, -5.58);
		\draw[->, black, dashed] (20.9, -5.5) -- (20.1, -5.5);
		\draw[->, black, dashed] (20.9, -5.42) -- (20.1, -5.42);
		
		\draw[->, black, dashed] (21.9, -5.42) -- (21.1, -5.42); 
		\draw[->, black, dashed] (21.9, -5.5) -- (21.1, -5.5); 
		\draw[->, black, dashed] (21.9, -5.58) -- (21.1, -5.58); 
		
		\draw[->, black, dashed] (22.9, -5.55) -- (22.1, -5.55);
		\draw[->, black, dashed] (22.9, -5.45) -- (22.1, -5.45);

		\draw[->, black, dashed] (18, -5.4) arc(265:185:.4);
		\draw[->, black, dashed] (18, -5.45) arc(265:185:.43);
		\draw[->, black, dashed] (18, -5.5) arc(265:185:.46);
		\draw[->, black, dashed] (18, -5.55) arc(265:185:.49);
		\draw[->, black, dashed] (18, -5.6) arc(265:185:.52);
		
		\draw[->, black, dashed] (17.4, -4.9) -- (17.4, -4.1);
		\draw[->, black, dashed] (17.45, -4.9) -- (17.45, -4.1);
		\draw[->, black, dashed] (17.5, -4.9) -- (17.5, -4.1);
		\draw[->, black, dashed] (17.55, -4.9) -- (17.55, -4.1);
		\draw[->, black, dashed] (17.6, -4.9) -- (17.6, -4.1);

		\draw[->, black, dashed] (17.4, -3.9) -- (17.4, -3.1);
		\draw[->, black, dashed] (17.45, -3.9) -- (17.45, -3.1);
		\draw[->, black, dashed] (17.5, -3.9) -- (17.5, -3.1);
		\draw[->, black, dashed] (17.55, -3.9) -- (17.55, -3.1);
		\draw[->, black, dashed] (17.6, -3.9) -- (17.6, -3.1);

		\filldraw[fill=white, draw=black] (18, -5) circle [radius = .1];
		\filldraw[fill=white, draw=black] (18, -4) circle [radius = .1];
		\filldraw[fill=white, draw=black] (19, -5) circle [radius = .1];
		\filldraw[fill=white, draw=black] (19, -4) circle [radius = .1];
		\filldraw[fill=white, draw=black] (20, -5) circle [radius = .1];
		\filldraw[fill=white, draw=black] (20, -4) circle [radius = .1];
		\filldraw[fill=white, draw=black] (21, -5) circle [radius = .1];
		\filldraw[fill=white, draw=black] (21, -4) circle [radius = .1];
		\filldraw[fill=white, draw=black] (22, -5) circle [radius = .1];
		\filldraw[fill=white, draw=black] (22, -4) circle [radius = .1];

		\filldraw[fill=gray!50!white, draw=black] (18, -5.5) circle [radius = .1];
		\filldraw[fill=gray!50!white, draw=black] (19, -5.5) circle [radius = .1];
		\filldraw[fill=gray!50!white, draw=black] (20, -5.5) circle [radius = .1];
		\filldraw[fill=gray!50!white, draw=black] (21, -5.5) circle [radius = .1];
		\filldraw[fill=gray!50!white, draw=black] (22, -5.5) circle [radius = .1];
		
		\filldraw[fill=gray!50!white, draw=black] (17.5, -4) circle [radius = .1];
		\filldraw[fill=gray!50!white, draw=black] (17.5, -5) circle [radius = .1];

		\end{tikzpicture}
		
	\end{center}
	
	\caption{\label{arrowstochasticvertices} The process of the stochasticization curve being ``pushed'' through a path ensemble is depicted above.  }
\end{figure}
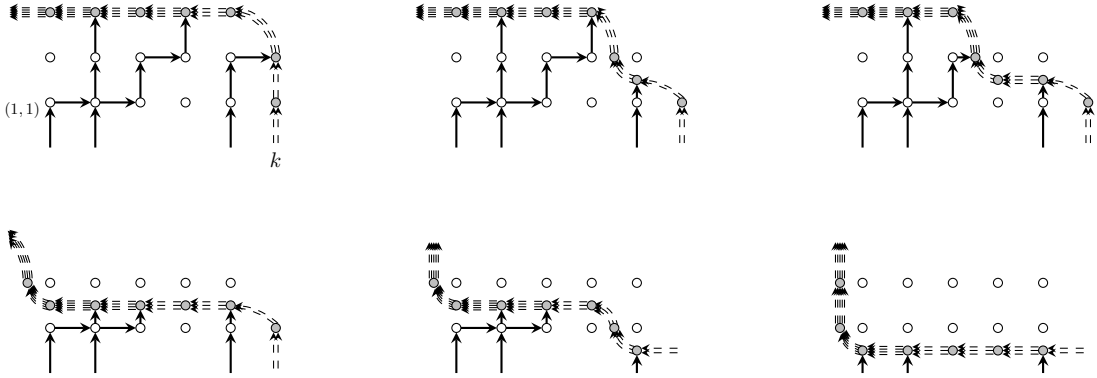

Specifically, suppose that we have a path ensemble $\mathcal{E}$ on domain $\mathcal{D}$, whose vertices we depict as unshaded (in \Cref{arrowstochasticvertices}, $\mathcal{D}$ is a $5 \times 2$ rectangle). We then attach a path of shaded vertices along the top-right boundary of $\mathcal{D}$, which are connected by dashed arrows, as shown in the top-left diagram of \Cref{arrowstochasticvertices}. We again refer to this path of dashed arrows as the \emph{stochasticization curve}, which begins with $k$ arrows entering the bottom-right shaded vertex and collects all arrows that exit $\mathcal{D}$.  

Next, we push the stochasticization curve through $\mathcal{D}$, one vertex at a time, as described above in \Cref{arrowstochastic}. For example, the top-middle diagram in \Cref{arrowstochasticvertices} depicts the stochasticization curve after being pushed through the top-right vertex $(5, 2)$ of $\mathcal{D}$; the bottom-left diagram there shows it after being pushed through the top row of $\mathcal{D}$; and the bottom-right diagram shows the stochasticization curve after being pushed through all of $\mathcal{D}$.

Now, unit squares whose vertices are lattice points in $\mathbb{Z}_{\ge 0}^2$ are faces of the graph $\mathbb{Z}_{> 0}^2$; if one of the four vertices of some face $F$ belongs to $\mathcal{D}$ we say that $F$ is a \emph{face of $\mathcal{D}$}. The pushing procedure described above provides a way to assign a nonnegative integer $k (F)$ to each face $F$ of $\mathcal{D}$ by setting $k (F)$ equal to the number of arrows in the stochasticization curve when it passes through $F$.

 By identifying $F$ with its top-left corner, this assigns an integer $k(u)$ to each vertex $u \in \mathcal{D}$. For example, if $u$ is the bottom-right vertex of $\mathcal{D}$ (which is $(5, 1)$ in \Cref{arrowstochasticvertices}), then $k(u) = k = 2$. Moreover, the top-right and bottom-middle diagrams in \Cref{arrowstochasticvertices} indicate that if $u = (3, 2)$ or $u = (4, 1)$, then $k (u) = 3$; similarly, one can verify that $k (2, 2) = 4$ and $k (1, 2) = 5$.

This provides a way to assign an integer $k(u)$ to each $u \in \mathcal{D}$, which depends on the number of arrows $k$ initially entering the stochasticization curve, as well as on the dircted path ensemble that is being stochasticized. Under this assignment, each $u \in \mathcal{D}$ can be stochasticized as in \eqref{si1j1i2j2vertex}, with weight given explicitly by \eqref{s01dynamicalweights}, where the $v = v(u)$ there is $s q^{k(u)} z^{-1}$. 

Since these parameters $k(u)$ and $v(u)$ change between vertices, they are sometimes referred to as \emph{dynamical parameters}. Based on the above description, one can quickly determine the identities governing these dynamical parameters. Specifically, let $u = (a, b)$ denote a vertex with arrow configuration $(i_1, j_1; i_2, j_2)$; denote $k(u) = k$ and $v(u) = v$. Then, the dynamical parameters at the vertices $(a - 1, b)$ and $(a, b + 1)$ are given by
\begin{flalign}
\label{ukv}
k (a - 1, b) = k + i_1; \quad k (a, b + 1) = k + j_2; \quad v (a - 1, b) = q^{i_1} v; \quad v (a, b + 1) = q^{j_2} v. 
\end{flalign}

\noindent We refer to \Cref{vdynamicalvertex1} for a depiction of these identities, which shows dynamical parameters corresponding to four adjacent faces (and therefore four adjacent vertices, since each face is associated with its top-left vertex). In this way, if we fix a path ensemble on some domain $\mathcal{D}$ and the dynamical parameter $k$ or $v$ at one vertex in $\mathcal{D}$, then the dynamical parameters $k (u)$ and $v (u)$ are defined at all $u \in \mathcal{D}$ through \eqref{ukv} and the fact that $v(u) = s q^{k(u)} z^{-1}$. This procedure of assigning dynamical parameters is not restricted to subdomains of $\mathbb{Z}^2$; it can be applied to any oriented, planar graph admitting a stochasticization curve that can be passed through it using the Yang-Baxter equation (see \Cref{DefinitionWeights}).

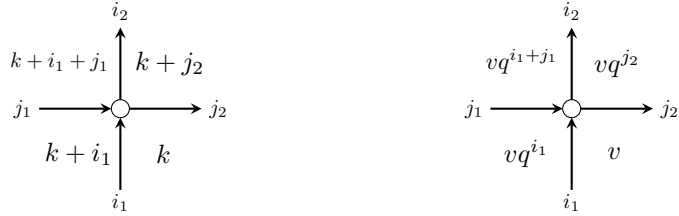
\begin{figure}

	\begin{center}

		\begin{tikzpicture}[
		>=stealth,
		auto,
		style={
			scale = 1.2
		}
		]

		\filldraw[fill=white, draw=black] (1, 0) circle [radius=.1] node[below = 17, right = 10]{$k$} node[scale = .8, above = 22, left = 2]{$k + i_1 + j_1$};
		
		\filldraw[fill=white, draw=black] (1, 1) circle [radius=0] node[below = 17, right = 2]{$k + j_2$};
		
		\filldraw[fill=white, draw=black] (0, 0) circle [radius=0] node[below = 17, right = 2]{$k + i_1$};

		\draw[->, black, thick] (1, -.9) -- (1, -.1) node[below = 26, scale = .8]{$i_1$}; 
		
		\draw[->, black, thick] (1, .1) -- (1, .9) node[scale = .8, above = 0]{$i_2$}; 
		
		\draw[->, black, thick] (.1, 0) -- (.9, 0)  node[left = 27, scale = .8]{$j_1$}; 
		
		\draw[->, black, thick] (1.1, 0) -- (1.9, 0) node[right = 0, scale = .8]{$j_2$};

		\filldraw[fill=white, draw=black] (6, 0) circle [radius=.1] node[below = 17, right = 10]{$v$} node[scale = .9, above = 20, left = 2]{$v q^{i_1 + j_1}$};
		
		\filldraw[fill=white, draw=black] (6, 1) circle [radius=0] node[below = 17, right = 5]{$v q^{j_2}$};
		
		\filldraw[fill=white, draw=black] (5, 0) circle [radius=0] node[below = 17, right = 5]{$v q^{i_1}$};

		\draw[->, black, thick] (6, -.9) -- (6, -.1) node[below = 26, scale = .8]{$i_1$}; 
		
		\draw[->, black, thick] (6, .1) -- (6, .9) node[scale = .8, above = 0]{$i_2$}; 
		
		\draw[->, black, thick] (5.1, 0) -- (5.9, 0) node[left = 27, scale = .8]{$j_1$}; 
		
		\draw[->, black, thick] (6.1, 0) -- (6.9, 0) node[right = 0, scale = .8]{$j_2$};

		\end{tikzpicture}
		
	\end{center}

	\caption{\label{vdynamicalvertex1} Depicted above is the way in which the dynamical parameters $k$ and $v$ change between faces.  }
\end{figure}

The following proposition states that, under the above choice of assignment of dynamical parameter $v(u)$, the $S$ weights satisfy the Yang-Baxter equation; see \Cref{equationpaths}. 

\begin{prop}
	
	\label{sdynamicalequation1}
	
	Fix $i_1, j_1, k_1, i_3, j_3, k_3 \in \{ 0, 1 \}$ and $x, y, z, v \in \mathbb{C}$. Then, 
	\begin{flalign} 
	\label{dynamicalequationspin1}
	\begin{aligned}
	\displaystyle\sum_{i_2, j_2, k_2 \in \{ 0, 1 \}}  & S \big( i_1, j_1; i_2, j_2 \b| x, y; q^{k_1} v \big) S \big( k_1, j_2; k_2, j_3 \b| x, z; v \big)  S \big( k_2, i_2; k_3, i_3 \b| y, z; q^{j_3} v \big) \\
	\qquad = \displaystyle\sum_{i_2, j_2, k_2 \in \{ 0, 1 \} } & S \big( k_1, i_1; k_2, i_2 \b| y, z; v \big)  S \big( k_2, j_1; k_3, j_2 \b| x, z; q^{i_2} v \big)  S \big( i_2, j_2; i_3, j_3 \b| x, y; v \big). 
	\end{aligned}
	\end{flalign}
\end{prop}

\begin{proof} 
	
	We will derive \eqref{dynamicalequationspin1} from the Yang-Baxter equation for the $w$ weights given by the last statement of \Cref{equationspin}. To that end, set $v = s q^r$ for some integer $r > 0$ and complex number $s$. In view of \eqref{si1j1i2j2vertex} (with the $k$ there replaced by our $r$ here and the $z$ there equal to $1$ here), the summand on the left side of \eqref{dynamicalequationspin1} divided by $w \big( i_1, j_1; i_2, j_2 \b| x, y \big) w \big( k_1, j_2; k_2, j_3 \b| x, z \big)  w \big( k_2, i_2; k_3, i_3 \b| y, z \big)$ is equal to
	\begin{flalign}
	\label{stochastic1dynamical}
	\begin{aligned}
	&  \displaystyle\frac{\chi_s \big(  r + k_1, j_2; r + k_1 + j_2, 0 \b| x, 1 \big)}{\chi_s \big( r + k_1 + i_1, j_1; r + i_1 + j_1 + k_1; 0 \b| x, 1 \big)} \displaystyle\frac{\chi_s \big( r + k_1 + j_2, i_2; r + i_2 + j_2 + k_1, 0 \b| y, 1 \big)}{\chi_s \big( r + k_1, i_1; r + k_1 + i_1; 0 \b| y, 1 \big)} \\
	& \quad \times\displaystyle\frac{\chi_s \big( r, j_3; r + j_3, 0 \b| x, 1 \big)}{\chi_s \big( r + k_1, j_2; r + k_1 + j_2; 0 \b| x, 1 \big)} \displaystyle\frac{\chi_s \big( r + j_3, k_2; r + k_2 + j_3, 0 \b| z, 1 \big)}{\chi_s \big( r, k_1; r + k_1; 0 \b| z, 1 \big)} \\
	& \quad \times \displaystyle\frac{\chi_s \big( r + j_3, i_3; r + i_3 + j_3, 0 \b| y, 1 \big)}{\chi_s \big( r + j_3 + k_2, i_2; r + i_2 + j_3 + k_2; 0 \b| y, 1 \big)} \displaystyle\frac{\chi_s \big( r + i_3 + j_3, k_3; r + i_3 + j_3 + k_3, 0 \b| z, 1 \big)}{\chi_s \big( r + j_3, k_2; r + j_3 + k_2; 0 \b| z, 1 \big)} \\
	& \qquad  =  \displaystyle\frac{\chi_s \big( r, j_3; r + j_3, 0 \b| x, 1 \big)}{\chi_s \big( r + k_1 + i_1, j_1; r + i_1 + j_1 + k_1; 0 \b| x, 1 \big)} \displaystyle\frac{\chi_s \big( r + j_3, i_3; r + i_3 + j_3, 0 \b| y, 1 \big)}{\chi_s \big( r + k_1, i_1; r + k_1 + i_1; 0 \b| y, 1 \big)} \\
	& \qquad \qquad \times  \displaystyle\frac{\chi_s \big( r + i_3 + j_3, k_3; r + i_3 + j_3 + k_3, 0 \b| z, 1 \big)}{\chi_s \big( r, k_1; r + k_1; 0 \b| z, 1 \big)},
	\end{aligned}
	\end{flalign}
		
	\noindent where we are using the fact that replacing $v$ by $q^a v$ corresponds to replacing $r$ by $r + a$ for any integer $a$, and also that $k_1 + j_2 = k_2 + j_3$ (due to arrow conservation). By similar reasoning, the right side of \eqref{dynamicalequationspin1} divided by $w \big( k_1, i_1; k_2, i_2 \b| y, z \big) w \big( k_2, j_1; k_3, j_2 \b| x, z \big) w \big( i_2, j_2; i_3, j_3 \b| x, y \big)$ is also equal to the right side of \eqref{stochastic1dynamical}. 
	
	Since that quantity is only dependent on the fixed boundary parameters $i_1, j_1, k_1, i_3, j_3, k_3$ (equivalently, they are independent of the interior, summed parameters $i_2$, $j_2$, and $k_2$), \eqref{dynamicalequationspin1} follows from the last statement of \Cref{equationspin}. 
\end{proof}

\begin{rem} 

\label{dynamicalsixvertexequationparameter}

Observe that the proof of \Cref{sdynamicalequation1} did not require the explicit values of the $\chi_s$ and $w$ weights given by \Cref{spin12spins}. Instead, it used the fact that \eqref{stochastic1dynamical} is independent of the intermediate indices $i_2, j_2, k_2$ (which are the ones that are summed over in \eqref{dynamicalequationspin1}), which was in turn due to the fact that each appearance of a $\chi_s$ weight in the numerator of the left side of \eqref{stochastic1dynamical} is matched by one in the denominator. 

In fact, by imposing this matching, one could guess the way in which the dynamical parameters should appear in \eqref{dynamicalequationspin1} without initially having the diagrammatic interpretation for these parameters given by \Cref{arrowstochasticvertices}. To that end, one would search for integers $a, b, c$ such that the quantity 
\begin{flalign}
\label{swi2j2k2}
\displaystyle\frac{  S \big( i_1, j_1; i_2, j_2 \b| x, y; q^a v \big) S \big( k_1, j_2; k_2, j_3 \b| x, z; q^b v \big)  S \big( k_2, i_2; k_3, i_3 \b| y, z; q^c v \big) }{ w \big( i_1, j_1; i_2, j_2 \b| x, y \big) w \big( k_1, j_2; k_2, j_3 \b| x, z	 \big)  w \big( k_2, i_2; k_3, i_3 \b| y, z \big) }
\end{flalign}

\noindent is independent of $i_2$, $j_2$, and $k_2$. Following the proof of \Cref{sdynamicalequation1}, we find that \eqref{swi2j2k2} is equal to 
\begin{flalign}
\label{stochastic1dynamical1}
\begin{aligned}
&  \displaystyle\frac{\chi_s \big(  r + a, j_2; r + a + j_2, 0 \b| x, 1 \big)}{\chi_s \big( r + a + i_1, j_1; r + i_1 + j_1 + a; 0 \b| x, 1 \big)} \displaystyle\frac{\chi_s \big( r + a + j_2, i_2; r + i_2 + j_2 + a, 0 \b| y, 1 \big)}{\chi_s \big( r + a, i_1; r + a + i_1; 0 \b| y, 1 \big)} \\
& \quad \times\displaystyle\frac{\chi_s \big( r + b, j_3; r + b + j_3, 0 \b| x, 1 \big)}{\chi_s \big( r + b + k_1, j_2; r + b + k_1 + j_2; 0 \b| x, 1 \big)} \displaystyle\frac{\chi_s \big( r + b + j_3, k_2; r + b + k_2 + j_3, 0 \b| z, 1 \big)}{\chi_s \big( r + b, k_1; r + k_1; 0 \b| z, 1 \big)} \\
& \quad \times \displaystyle\frac{\chi_s \big( r + c, i_3; r + c + i_3, 0 \b| y, 1 \big)}{\chi_s \big( r + c + k_2, i_2; r + c + i_2 + k_2; 0 \b| y, 1 \big)} \displaystyle\frac{\chi_s \big( r + c + i_3, k_3; r + c + i_3 + k_3, 0 \b| z, 1 \big)}{\chi_s \big( r + c, k_2; r + c + k_2; 0 \b| z, 1 \big)}.
\end{aligned}
\end{flalign}

As explained above, one way of guaranteeing that \eqref{stochastic1dynamical1} be independent of $i_2, j_2, k_2$ would be to impose a matching between any term in the numerator involving one of these indices with one in the denominator. For instance, this would require us to match 
\begin{flalign*}
\chi_s \big(  r + a, j_2; r + a + j_2, 0 \b| x, 1 \big) = \chi_s \big( r + b + k_1, j_2; r + b + k_1 + j_2; 0 \b| x, 1 \big),
\end{flalign*}

\noindent which would be true if $a = b + k_1$. Similarly, one equates 
\begin{flalign*}
& \chi_s \big( r + a + j_2, i_2; r + i_2 + j_2 + a, 0 \b| y, 1 \big) = \chi_s \big( r + c + k_2, i_2; r + c + i_2 + k_2; 0 \b| y, 1 \big); \\
& \chi_s \big( r + b + j_3, k_2; r + b + k_2 + j_3, 0 \b| z, 1 \big) = \chi_s \big( r + c, k_2; r + c + k_2; 0 \b| z, 1 \big),
\end{flalign*} 

\noindent which are satisfied if $a + j_2 = c + k_2$ and $c = b + j_3$, respectively. In particular, these equations are all satisfied for $(a, b, c) = (k_1, 0, j_3)$, which is indeed the choice of parameters on the left side of \eqref{dynamicalequationspin1}. One could then similarly guess the way in which the dynamical parameters should appear on the right side of \eqref{dynamicalequationspin1}. 

This form of reasoning is not quite necessary in the present situation, since the identities \eqref{ukv} governing the dynamical parameter $v$ can be obtained through the diagrammatic procedure depicted in \Cref{arrowstochasticvertices}. However, we will implement a variant of this analysis in the proof of \Cref{3equationdimensiondynamical} in \Cref{EquationStochastic3}, where such a diagram appears to be unavailable. 
\end{rem}

\section{The Stochasticization Procedure}

\label{Domain}

In this section we provide a general procedure for stochasticizing solutions to the (two-dimensional) Yang-Baxter equation; we explain it in Section \ref{DefinitionWeights} and then establish some properties about the stochasticized weights in Section \ref{PropertiesWeights}. 
	
\subsection{Stochasticization of Face Weights}

\label{DefinitionWeights}

In this section we describe the general (two-dimensional version of the) stochasticization procedure that produces a stochastic solution to the Yang-Baxter equation starting from a nonnegative one. The setting here will be that of an \emph{interaction round-a-face (IRF) model} as opposed to of a vertex model of the type studied in \Cref{VertexModel}, and so our weights will be with respect to faces (and not vertices) of a graph. This will cause the notation here to be slightly different from what was described in \Cref{VertexModel} but, as we will explain further below, these two perspectives are essentially equivalent by dualizing the underlying graph. 

We begin with a finite family $\mathcal{L}$ of directed lines\footnote{For convenience, we will sometimes draw these lines as curves in the plane to simplify diagrams.} in $\mathbb{R}^2$, that is, each line $\ell \in \mathcal{L}$ has both a positive and negative direction. We assume that these directions are \emph{consistent}, meaning that there exists an open half-plane $H \subset \mathbb{R}^2$ that intersects each $\ell \in \mathcal{L}$ in its positive direction. Stated equivalently, there do not exist three lines in $\mathfrak{L}$ whose positive and negative directions ``interlace.'' We refer to Figure \ref{directions} for examples.

\begin{figure}

	\begin{center}

		\begin{tikzpicture}[
		>=stealth,
		auto,
		style={
			scale = 1
		}
		]

		\fill[fill=white!70!gray] (-1.5, .5) -- (-1.5, 2.5) -- (2.5, 2.5) -- (2.5, -.5) -- (-1.5, .5);

		\draw[->, black] (0, -1)  -- (0, 2); 
		\draw[->, black] (1, -1) -- (-1, 2); 
		\draw[->, black] (-1, -.5) -- (2, .5);
		\draw[->, black] (.5,-1) -- (2, 2); 
		\draw[->, black] (-1, 0) -- (1, 2);

		\draw[->, black] (7, 2) -- (7, -1); 
		\draw[->, black] (8, -1) -- (6, 2); 
		\draw[->, black] (6, -.5) -- (9, .5);

		\end{tikzpicture}
		
	\end{center}

	\caption{\label{directions} Shown to the left is a set of five lines whose directions are consistent. Shown to the right are three lines whose directions are inconsistent since their directions interlace.}

\end{figure}
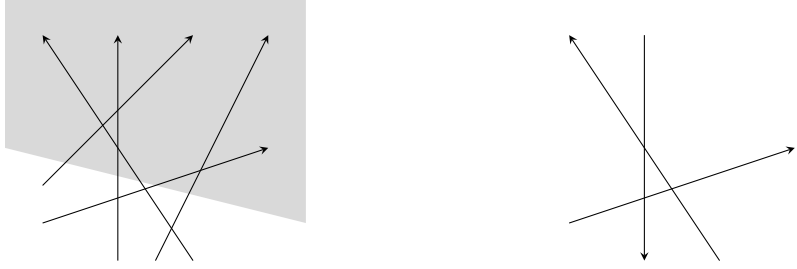

The set $\mathcal{L}$ induces a directed, planar graph $\mathcal{D}$, whose vertex set consists of the intersection points between elements of $\mathcal{L}$ and whose edge set consists of segments of the (directed) lines in $\mathcal{L}$ connecting adjacent vertices. Let $\mathcal{G}$ denote the dual graph to $\mathcal{D}$, that is, a vertex of $\mathcal{G}$ is a face of $\mathcal{D}$ and two vertices of $\mathcal{G}$ are connected by an edge if the corresponding faces in $\mathcal{D}$ share an edge.

We call the faces of $\mathcal{G}$ \emph{plaquettes}; they correspond to the vertices of $\mathcal{D}$. Each plaquette $P$ has four vertices $(v_1, v_2, v_3, v_4)$, which we order counterclockwise and in such a way that the two edges adjacent to $v_1$ are the segments closer to the negative directions of the two lines in $\mathcal{L}$ that intersect to form $P$; see the left side of Figure \ref{frozenvertices} for an example. We will call the edges in $\mathcal{G}$ connecting $(v_1, v_2)$ and $(v_3, v_4)$ \emph{horizontal} and the ones connecting $(v_2, v_3)$ and $(v_1, v_4)$ \emph{vertical} (although these edges might not be parallel to the axes). 

Denoting the vertex set of $\mathcal{G}$ by $V$, fix some complex vector space $\mathfrak{h}$ and two functions $\mathcal{F}: V \rightarrow \mathfrak{h}$ and $w: \mathfrak{h} \times \mathfrak{h} \times \mathfrak{h} \times \mathfrak{h} \rightarrow \mathbb{C}$; we refer to $\mathcal{F}$ as a \emph{height function} and $w$ as a \emph{weight function}. We view $w (a, b, c, d)$ as the weight of a plaquette $P = (v_1, v_2, v_3, v_4)$ that satisfies $\big( \mathcal{F} (v_1), \mathcal{F} (v_2), \mathcal{F} (v_3), \mathcal{F} (v_4) \big) = (a, b, c, d)$. 

\begin{rem} 
	
\label{vertexweightsmodel}	
	
Let us explain how the setting described here degenerates to that of \Cref{VertexModel}. There, the weights were with respect to vertices of a domain $\mathcal{D}$. By taking $\mathcal{G}$ to be the dual graph of $\mathcal{D}$ as above, these vertex weights now become plaquette weights on $\mathcal{G}$. 

The quadruple $\big( \mathcal{F} (a), \mathcal{F} (b), \mathcal{F} (c), \mathcal{F} (d) \big)$ corresponding to a plaquette here is the analog of the arrow configuration of a vertex defined in \Cref{SixVertexEquation}. Explicitly, we may define the height function $\mathfrak{h}$ on $\mathcal{G}$ by setting $\mathfrak{h} (v)$ to be the number of paths to the right of the vertex $v$ of $\mathcal{G}$ (or, equivalently, of the corresponding face in $\mathcal{D}$). Then, the height function at the four corners of a plaquette $P$ in $\mathcal{G}$ determine the arrow configuration of the vertex in $\mathcal{D}$ corresponding to $P$ through the identities 
\begin{flalign}
\label{fi1j1i2j2}
i_1 = \mathcal{F} (v_1) - \mathcal{F} (v_2); \qquad j_1 = \mathcal{F} (v_4) - \mathcal{F} (v_1); \qquad i_2 = \mathcal{F} (v_4) - \mathcal{F} (v_3);  \qquad j_2 = \mathcal{F} (v_3) - \mathcal{F} (v_2). 
\end{flalign}

Thus, by dualizing, the weight of a vertex in $u \in \mathcal{D}$ (dependent on the arrow configuration of $u$) becomes the weight of the corresponding plaquette $P$ in $\mathcal{G}$ (dependent on the height function at the four vertices of $P$). 

In the examples to be discussed in \Cref{StochasticElliptic}, \Cref{RankDynamical}, and \Cref{DynamicalDimension}, we will adopt this vertex perspective as opposed to the face model one described here (including in situations where we analyze face models, such as in \Cref{StochasticElliptic}). However, as indicated above, these two points of view are equivalent, and so we will omit comments of this type in later sections.

\end{rem} 

The first and second definitions below provide notation on when the weight function is nonzero or zero, respectively.

\begin{definition}

\label{admissible}

 We call a quadruple $(a, b, c, d) \in \mathfrak{h} \times \mathfrak{h} \times \mathfrak{h} \times \mathfrak{h}$ \emph{admissible (with respect to $w$)} if $w(a, b, c, d) \ne 0$. Let $\Adm_1 (a, b, c; w), \Adm_3 (a, b, c; w),\subset \mathfrak{h}$ denote the sets of $d \in \mathfrak{h}$ such that $(d, a, b, c)$ and $(a, b, d, c)$ are admissible, respectively. Also define $\Adm (a, b, c; w) =  \Adm_1 (a, b, c; w) \cup \Adm_3 (a, b, c; w)$.

\end{definition}

\begin{definition}
	
	\label{triplefrozen} 
	
We  call a triple $(b, c, d) \in \mathfrak{h} \times \mathfrak{h} \times \mathfrak{h}$ \emph{frozen} (with respect to $w$) if $\big| \Adm_1 (b, c, d) \big| = 1$. Furthermore, we say that a plaquette $P = (v_1, v_2, v_3, v_4)$ has an \emph{frozen boundary} (with respect to a height function $\mathcal{F}$) if $\big( \mathcal{F} (v_2), \mathcal{F} (v_3), \mathcal{F} (v_4) \big)$ is frozen; we depict this diagrammatically by placing squares in the corners of any plaquette with a frozen boundary (see the right side of Figure \ref{frozenvertices}). 

\end{definition}

\begin{rem}
	
\label{vertexadmissiblefrozen} 

Again, let us take a moment to explain the analogs of these two notions in the setting of \Cref{VertexModel}. As indicated in \Cref{vertexweightsmodel}, the height function $\mathfrak{h}$ evaluated at the four vertices of a plaquette $P = (v_1, v_2, v_3, v_4)$ determines the arrow configuration $(i_1, j_1; i_2, j_2)$ at the vertex $u \in \mathcal{D}$ corresponding to $P$ through the identities \eqref{fi1j1i2j2}. The admissibility constraint here means that $(a, b, c, d) = \big( \mathcal{F} (v_1), \mathcal{F} (v_2), \mathcal{F} (v_3), \mathcal{F} (v_4) \big)$ are chosen such that $w (i_1, j_1; i_2, j_2) \ne 0$. This requires $i_1, j_1, i_2, j_2 \ge 0$, which means that $a, c \in [b, d]$; in particular, $\big| \Adm (x, y, z) \big| < \infty$ for any $x, y, z \in \mathbb{Z}$. 

Now, \eqref{fi1j1i2j2} implies that fixing the last three elements of the height function quadruple $(a, b, c, d)$ of a plaquette in $\mathcal{G}$ fixes the numbers $(i_2, j_2)$ of outgoing vertical and horizontal arrows at the corresponding vertex of $\mathcal{D}$. Thus, the statement that a triple $(b, c, d)$ freezes is equivalent to that the corresponding pair $(i_2, j_2)$ admits only one pair $(i_1, j_1) \in \mathbb{Z}_{\ge 0}^2$ such that $w (i_1, j_1; i_2, j_2) \ne 0$. As mentioned in \Cref{SixVertexEquation}, this is guaranteed if $i_2 = 0 = j_2$, in which case $b = c = d$. 

\end{rem}

\begin{figure}

	\begin{center}

		\begin{tikzpicture}[
		>=stealth,
		auto,
		style={
			scale = 1.5
		}
		]

		\draw[black] (-5, 0) -- (-5, 1); 
		\draw[black] (-5, 0) -- (-4, 0);
		\draw[black] (-4, 0) -- (-4, 1); 
		\draw[black] (-5, 1) -- (-4, 1);

		\draw[->, black, dotted] (-5.5, .4) -- (-3.5, .6);
		
		\draw[->, black, dotted] (-4.6, -.5) -- (-4.4, 1.5);

		\fill[black] (-5, 0) circle [radius=.03] node [scale = .8, black, below = 0]{$v_1$};
		\fill[black] (-4, 0) circle [radius=.03] node [scale = .8, black, below = 0]{$v_2$};
		\fill[black] (-4, 1) circle [radius=.03] node [scale = .8, black, above = 0]{$v_3$};
		\fill[black] (-5, 1) circle [radius=.03] node [scale = .8, black, above = 0]{$v_4$};

		\draw[black] (0, 0) -- (0, 1); 
		\draw[black] (0, 0) -- (1, 0);
		\draw[black] (1, 0) -- (1, 1); 
		\draw[black] (0, 1) -- (1, 1);

		\fill[black] (-.08, -.08) rectangle (.08, .08) node [scale = .8, black, left = 2, below = 9] {$v_1$};
		\fill[black] (.92, -.08) rectangle (1.08, .08) node [scale = .8, black, left = 2, below = 9] {$v_2$};
		\fill[black] (.92, .92) rectangle (1.08, 1.08) node [scale = .8, black, left = 2, above = 0] {$v_3$};
		\fill[black] (-.08, .92) rectangle (.08, 1.08) node [scale = .8, black, left = 2, above = 0] {$v_4$};

		\end{tikzpicture}
		
	\end{center}

	\caption{\label{frozenvertices} To the left is the labeling of vertices on a plaquette in $\mathcal{G}$; the dotted lines there are elements of $\mathcal{L}$. To the right is a frozen plaquette, which we indicate by placing squares in its corners. }

\end{figure}
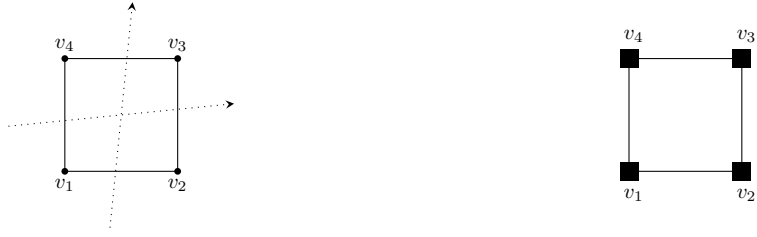

Now let us fix an additional function $\chi: \mathfrak{h} \times \mathfrak{h} \times \mathfrak{h} \times \mathfrak{h} \rightarrow \mathbb{C}$. The following definition stipulates certain conditions on the pair $(w, \chi)$ indicating when it is possible to use $\chi$ to create a stochastic solution to the Yang-Baxter relation from $w$. 

\begin{definition}
	
	\label{wchiconditions}
	
	Adopting the notation above, we say that $w$ is \emph{stochasticizable} with respect to $\chi$ if the following three conditions hold.

\begin{enumerate}
	
	\item For each $a, b, c \in \mathfrak{h}$, both	 $\big| \Adm (a, b, c; w) \big|$ and $\big| \Adm (a, b, c; \chi) \big|$ are finite. 
	
	\item The weights $w$ satisfy the \emph{Yang-Baxter equation}, which states that 
	\begin{flalign}
	\label{unshadedequation1}
	\begin{aligned}
	\displaystyle\sum_{c \in \mathfrak{h}} &	w (a, b, c, d)  w (b, e, f, c) w (c, f, g, d) = \displaystyle\sum_{c \in \mathfrak{h}}	w (b, e, c, a) w (a, c, g, d) w (c, e, f, g),
	\end{aligned}
	\end{flalign}
	
	\noindent for any fixed $a, b, d, e, f, g \in \mathfrak{h}$. Diagrammatically, this is the equation
		\begin{flalign}
		\label{unshadedequation}
		\begin{aligned}
			\begin{tikzpicture}[
			>=stealth,
			auto,
			style={
				scale = .75
			}
			]
			\draw (-2, 0) -- (-1, 1) --  (-1, 2) -- (-2, 1) -- (-2, 0); 
			\draw (-3, 1) -- (-2, 2) -- (-1, 2) -- (-2, 1) -- (-3, 1); 
			\draw[-, black] (-3, 0) -- (-2,  0); 
			\draw[-, black] (-3, 0) -- (-3,  1); 
			\draw[-, black] (-3, 1) -- (-2,  1); 
			\draw[-, black] (-2, 0) -- (-2,  1); 
			\draw[->, black, dotted] (-2.6, -.5) arc(180:135:4) node[scale = .6, above = 0]{$\ell_1$};
			\draw[->, black, dotted] (-1.5, .3) arc(0:70:2) node[scale = .6, above = 0]{$\ell_3$};
			\draw[->, black, dotted] (-3.4, .35) arc(270:315:4) node[scale = .6, above = 0]{$\ell_2$};
			\draw[->, black, dotted] (4.7, -.4) arc(-45:0:4) node[scale = .6, above = 0]{$\ell_1$};
			\draw[->, black, dotted] (6.2, .5) arc(270:197:2) node[scale = .6, above = 0]{$\ell_3$};
			\draw[->, black, dotted] (4, .6) arc(135:92:4) node[scale = .6, left = 1, right = 0]{$\ell_2$};
			\filldraw[fill=black, draw=black] (-3, 0) circle [radius=.03] node [scale = .75, black, below = 0] {$a$};
			\filldraw[fill=black, draw=black] (-2,  0) circle [radius=.03] node [scale = .75, black, left = 1, below = 0] {$b$};
			\filldraw[fill=black, draw=black] (-3, 1) circle [radius=.03] node [scale = .75, black, above = 2 ] {$d$};
			\fill[black] (-2, 1) circle [radius=.03] node [scale = .75, black, right = 0, below = 6, left = 0] {$c$};
			\fill[black] (-1, 1) circle [radius=.03] node [scale = .75, black, below = 2, right = 0] {$e$};
			\fill[black] (-2, 2) circle [radius=.03] node [scale = .75, black, above = 0] {$g$};
			\fill[black] (-1, 2) circle [radius=.03] node [scale = .75, black, above = 0] {$f$};		
			\draw[black] (.8, .85) circle[radius=0] node[scale = 1.8]{$=$}; 	
			\filldraw[fill=black, draw=black] (-3.5, .75) circle [radius=0] node [scale = 1.8, black, left = 1] {$\displaystyle\sum_c$};
			\draw (4.25, 0) -- (5.25, 1) --  (5.25, 2) -- (4.25, 1) -- (4.25, 0); 
			\draw (4.25, 0) -- (5.25, 1) -- (6.25, 1) -- (5.25, 0) -- (4.25, 0); 
			\draw[-, black] (5.25, 1) -- (6, 1); 
			\draw[-, black] (5.25, 2) -- (6.25, 2); 
			\draw[-, black] (6.25, 1) -- (6.25, 2); 
			\filldraw[fill=black, draw=black] (4.25, 0) circle [radius=.03] node [scale = .75, black, below = 0] {$a$};
			\filldraw[fill=black, draw=black] (5.25, 0) circle [radius=.03] node [scale = .75, black, left = 1, below = 0] {$b$};
			\filldraw[fill=black, draw=black] (4.25, 1) circle [radius=.03] node [scale = .75, black, above = 2 ] {$d$};
			\fill[black] (5.25, 1) circle [radius=.03] node [scale = .75, black, right = 0, above = 6, right = 0] {$c$};
			\fill[black] (6.25, 1) circle [radius=.03] node [scale = .75, black, below = 2, right = 0] {$e$};
			\fill[black] (5.25, 2) circle [radius=.03] node [scale = .75, black, above = 0] {$g$};
			\fill[black] (6.25, 2) circle [radius=.03] node [scale = .75, black, above = 0] {$f$};
			\filldraw[fill=black, draw=black] (3.75, .75) circle [radius=0] node [scale = 1.8, black, left = 1] {$\displaystyle\sum_c$};
			\end{tikzpicture}
		\end{aligned}
		\end{flalign}

	\item Let $P = (v_1, v_2, v_3, v_4)$ be any plaquette, and denote $a = \mathcal{F} (v_1)$; $b = \mathcal{F} (v_2)$; $c = \mathcal{F} (v_3)$; and $d = \mathcal{F} (v_4)$. We can associate with $P$ a \emph{shaded triple}\footnote{This triple need not be unique. In fact, it usually will not be, and this non-uniqueness will give rise to additional, dynamical parameters.} $(x, y, z) \in \mathfrak{h} \times \mathfrak{h} \times \mathfrak{h}$, which is independent of the value of $a = \mathcal{F} (v_1)$, such that the following conditions hold. In what follows, we view this shaded triple as ``attaching'' two new plaquettes to $P$ along the positive direction. We distinguish these new plaquettes by shading them; see the left side of  \Cref{assignment}. We can also view this procedure as introducing a frozen plaquette and then attaching two shaded plaquettes to it along the negative direction, as on the right side of \Cref{assignment}. In either case, the weights of the original or frozen plaquettes will be with respect to $w$, and weights of the shaded ones will be with respect to $\chi$.

	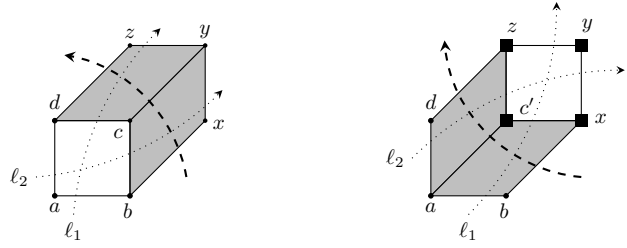
\begin{figure} 
		
		\begin{center}

			\begin{tikzpicture}[
			>=stealth,
			auto,
			style={
				scale = 1
			}
			]

			\filldraw[fill=gray!50!white] (1, 0) -- (2, 1) --  (2, 2) -- (1, 1) -- (1, 0); 
			\filldraw[fill=gray!50!white] (0, 1) -- (1, 2) -- (2, 2) -- (1, 1) -- (0, 1); 
			\draw[-, black] (0, 0) -- (1,  0); 
			\draw[-, black] (0, 0) -- (0,  1); 
			\draw[-, black] (0, 1) -- (1,  1); 
			\draw[-, black] (1, 0) -- (1,  1); 
			\filldraw[fill=black, draw=black] (0, 0) circle [radius=.03] node [scale = .8, black, below = 0] {$a$};
			\filldraw[fill=black, draw=black] (1,  0) circle [radius=.03] node [scale = .8, black, left = 1, below = 0] {$b$};
			\filldraw[fill=black, draw=black] (0, 1) circle [radius=.03] node [scale = .8, black, above = 2 ] {$d$};
			\fill[black] (1, 1) circle [radius=.03] node [scale = .8, black, right = 0, below = 6, left = 0] {$c$};
			\fill[black] (2, 1) circle [radius=.03] node [scale = .8, black, below = 2, right = 0] {$x$};
			\fill[black] (1, 2) circle [radius=.03] node [scale = .8, black, above = 0] {$z$};
			\fill[black] (2	, 2) circle [radius=.03] node [scale = .8, black, above = 0] {$y$};

			\filldraw[fill=gray!50!white] (5, 0) -- (6, 1) --  (6, 2) -- (5, 1) -- (5, 0); 
			\filldraw[fill=gray!50!white] (5, 0) -- (6, 1) -- (7, 1) -- (6, 0) -- (5, 0); 
			\draw[-, black] (6, 1) -- (7,  1); 
			\draw[-, black] (6, 1) -- (6,  2); 
			\draw[-, black] (6, 2) -- (7,  2); 
			\draw[-, black] (7, 1) -- (7,  2); 
			\filldraw[fill=black, draw=black] (5, 0) circle [radius=.03] node [scale = .8, black, below = 0] {$a$};
			\filldraw[fill=black, draw=black] (6,  0) circle [radius=.03] node [scale = .8, black, left = 1, below = 0] {$b$};
			\filldraw[fill=black, draw=black] (5, 1) circle [radius=.03] node [scale = .8, black, above = 2 ] {$d$};
			\fill[black] (5.92, .92) rectangle (6.08, 1.08) node [scale = .8, black, right = 0, above = 4, right = 0] {$c'$};
			\fill[black] (6.92, .92) rectangle (7.08, 1.08) node [scale = .8, black, below = 2, right = 0] {$x$};
			\fill[black] (5.92, 1.92) rectangle (6.08, 2.08) node [scale = .8, black, above = 0] {$z$};
			\fill[black] (6.92, 1.92) rectangle (7.08, 2.08) node [scale = .8, black, above = 0] {$y$};

			\draw[->, black, dashed, thick] (1.75, .25) arc(10:80:2);
			\draw[->, black, dotted] (-.25, .25) arc(275:315:4);
			\draw[->, black, dotted] (.25, -.25) arc(175:135:4);

			\draw[->, black, dashed, thick] (7, .25) arc(265:185:2);
			\draw[->, black, dotted] (4.75, .5) arc(135:90:4);
			\draw[->, black, dotted] (5.5, -.25) arc(-45:0:4);
			
			\draw (5.5, -.25) circle[radius = 0] node[scale = .8, below = 0]{$\ell_1$};
			
			\draw (.25, -.25) circle[radius = 0] node[scale = .8, below = 0]{$\ell_1$};
			
			\draw (4.75, .5) circle[radius = 0] node[scale = .8, left = 0]{$\ell_2$};
			
			\draw (-.25, .25) circle[radius = 0] node[scale = .8, left = 0]{$\ell_2$};
			
			\end{tikzpicture}
			
		\end{center}

		\caption{\label{assignment} We assign to each plaquette in our original graph two shaded plaquettes (as shown to the left), such that the top-right square (shown to the right) has a frozen boundary.}
	\end{figure}

	\begin{enumerate}  
		
		\item The shaded triple $(x, y, z)$ is frozen (with respect to $w$); let $c'$ be such that $(c', x, y, z)$ is admissible. 
		
		\item Both $\chi (b, x, c', a)$ and $\chi (a, c', z, d)$ are nonzero. 
		
		\item The special case of the Yang-Baxter equation given by 
		\begin{flalign}
		\label{wchirelation}
		\displaystyle\sum_{c \in \mathfrak{h}} w(a, b, c, d) \chi (b, x, y, c) \chi (c, y, z, d) =  \chi (b, x, c', a) \chi (a, c', z, d) w (c', x, y, z)
		\end{flalign}
		
		\noindent holds; the right side of \eqref{wchirelation} has one term since $(x, y, z)$ is frozen. Diagrammatically, this is the equation
		\begin{flalign*}
				\begin{tikzpicture}[
				>=stealth,
				auto,
				style={
					scale = .5
				}
				]
				\filldraw[fill=gray!50!white, draw=black] (-2, 0) -- (-1, 1) --  (-1, 2) -- (-2, 1) -- (-2, 0); 
				\filldraw[fill=gray!50!white, draw=black] (-3, 1) -- (-2, 2) -- (-1, 2) -- (-2, 1) -- (-3, 1); 
				\draw[-, black] (-3, 0) -- (-2,  0); 
				\draw[-, black] (-3, 0) -- (-3,  1); 
				\draw[-, black] (-3, 1) -- (-2,  1); 
				\draw[-, black] (-2, 0) -- (-2,  1); 
				\filldraw[fill=black, draw=black] (-3, 0) circle [radius=.03] node [scale = .5, black, below = 0] {$a$};
				\filldraw[fill=black, draw=black] (-2,  0) circle [radius=.03] node [scale = .5, black, left = 1, below = 0] {$b$};
				\filldraw[fill=black, draw=black] (-3, 1) circle [radius=.03] node [scale = .5, black, left = 0] {$d$};
				\fill[black] (-2, 1) circle [radius=.03] node [scale = .5, black, right = 0, below = 6, left = 0] {$c$};
				\fill[black] (-1, 1) circle [radius=.03] node [scale = .5, black, below = 2, right = 0] {$x$};
				\fill[black] (-2, 2) circle [radius=.03] node [scale = .5, black, above = 0] {$z$};
				\fill[black] (-1, 2) circle [radius=.03] node [scale = .5, black, above = 0] {$y$};
				\draw[black] (.925, .85) circle[radius=0] node[scale = 1.2]{$=$}; 		
				\filldraw[fill=black, draw=black] (-3.75, .75) circle [radius=0] node [scale = 1.2, black, left = 1] {$\displaystyle\sum_{c}$};
				\filldraw[fill=gray!50!white, draw=black] (2.75, 0) -- (3.75, 1) --  (3.75, 2) -- (2.75, 1) -- (2.75, 0); 
				\filldraw[fill=gray!50!white, draw=black] (2.75, 0) -- (3.75, 1) -- (4.75, 1) -- (3.75, 0) -- (2.75, 0); 
				\draw[-, black] (3.75, 1) -- (4.75, 1); 
				\draw[-, black] (3.75, 2) -- (4.75, 2); 
				\draw[-, black] (4.75, 1) -- (4.75, 2); 
				\filldraw[fill=black, draw=black] (2.75, 0) circle [radius=.03] node [scale = .5, black, below = 0] {$a$};
				\filldraw[fill=black, draw=black] (3.75, 0) circle [radius=.03] node [scale = .5, black, left = 1, below = 0] {$b$};
				\filldraw[fill=black, draw=black] (2.75, 1) circle [radius=.03] node [scale = .5, black, left = 0] {$d$};
				\fill[black] (3.67, .92) rectangle (3.83, 1.08) node [scale = .5, black, right = 4, above = 0] {$c'$};
				\fill[black] (4.67, .92) rectangle (4.83, 1.08) node [scale = .5, black, below = 2, right = 0] {$x$};
				\fill[black] (3.67, 1.92) rectangle (3.83, 2.08) node [scale = .5, black, above = 0] {$z$};
				\fill[black] (4.67, 1.92) rectangle (4.83, 2.08) node [scale = .5, black, above = 0] {$y$};
				\end{tikzpicture}
		\end{flalign*}

		\item \emph{Consistency Condition:} If $P$ and $Q$ are adjacent plaquettes, then the new, attached plaquettes satisfy the consistency relations depicted in Figure \ref{consistent}. Explicitly, suppose that $P$ and $Q$ share a vertical edge and that the height functions on their vertices are as in the top of Figure \ref{consistent}. Denote the shaded triple associated with $P$ by $(x, y, z)$ and the one associated with $Q$ by $(x', y', z')$ (again, as in the top of Figure \ref{consistent}). Then, we have that $(b, x, y, c) = (b, f, z', c)$ (so $x = f$ and $y = z'$). Similarly, if $P$ and $Q$ share a horizontal edge and the height function on the vertices of $P$ and $Q$ are labeled as in the bottom of Figure \ref{consistent}, then $(c, y, z, d) = (c, x', e, d)$ (so $y = x'$ and $z = e$).

		\begin{figure} 
			
			\begin{center}

				\begin{tikzpicture}[
				>=stealth,
				auto,
				style={
					scale = 1
				}
				]

				\draw[-, black] (-4, .5) -- (-5, .5) node[below] {$a$}; 
				\draw[-, black] (-4, 1.5) -- (-4, .5) node[below] {$b$};
				\draw[-, black] (-3, 1.5) -- (-3, .5) node[below] {$e$};
				
				\draw[-, black] (-5, .5) -- (-5, 1.5) node[above] {$d$};
				\draw[-, black] (-5, 1.5) -- (-4, 1.5) node[above] {$c$};
				\draw[-, black] (-4, 1.5) -- (-3, 1.5) node[above] {$f$};
				\draw[-, black] (-3, .5) -- (-4, .5);

				\fill[black] (-3, .5) circle [radius=.03] node [scale = .8, black, right = 0, below = 6, left = 0] {};
				\fill[black] (-3, 1.5) circle [radius=.03] node [scale = .8, black, below = 2, right = 0] {};
				\fill[black] (-4, .5) circle [radius=.03] node [scale = .8, black, above = 0] {};
				\fill[black] (-4, 1.5) circle [radius=.03] node [scale = .8, black, above = 0] {};
				\fill[black] (-5, .5) circle [radius=.03] node [scale = .8, black, above = 0] {};
				\fill[black] (-5, 1.5) circle [radius=.03] node [scale = .8, black, above = 0] {};
				
				\filldraw[fill=gray!50!white] (1, 0) -- (2, 1) --  (2, 2) -- (1, 1) -- (1, 0); 
				\filldraw[fill=gray!50!white] (0, 1) -- (1, 2) -- (2, 2) -- (1, 1) -- (0, 1); 
				\draw[-, black] (0, 0) -- (1,  0); 
				\draw[-, black] (0, 0) -- (0,  1); 
				\draw[-, black] (0, 1) -- (1,  1); 
				\draw[-, black] (1, 0) -- (1,  1); 
				\filldraw[fill=black, draw=black] (0, 0) circle [radius=.03] node [scale = .8, black, below = 0] {$a$};
				\filldraw[fill=black, draw=black] (1,  0) circle [radius=.03] node [scale = .8, black, left = 1, below = 0] {$b$};
				\filldraw[fill=black, draw=black] (0, 1) circle [radius=.03] node [scale = .8, black, above = 2 ] {$d$};
				\fill[black] (1, 1) circle [radius=.03] node [scale = .8, black, right = 0, below = 6, left = 0] {$c$};
				\fill[black] (2, 1) circle [radius=.03] node [scale = .8, black, below = 2, right = 0] {$x$};
				\fill[black] (1, 2) circle [radius=.03] node [scale = .8, black, above = 0] {$z$};
				\fill[black] (2	, 2) circle [radius=.03] node [scale = .8, black, above = 0] {$y$};

				\filldraw[fill=black, draw=black] (1,  0) circle [radius=.03] node [scale = .8, black, right = 18, above = 29] {$=$};

				\filldraw[fill=gray!50!white] (4.5,0) -- (5.5,1) --  (5.5, 2) -- (4.5, 1) -- (4.5, 0); 
				\filldraw[fill=gray!50!white] (4.5, 0) -- (5.5, 1) -- (6.5, 1) -- (5.5, 0) -- (4.5, 0); 
				\draw[-, black] (5.5, 1) -- (6.5, 1); 
				\draw[-, black] (5.5, 1) -- (5.5, 2); 
				\draw[-, black] (5.5, 2) -- (6.5, 2); 
				\draw[-, black] (6.5, 1) -- (6.5, 2);

				\filldraw[fill=black, draw=black] (4.5,0) circle [radius=.03] node [scale = .8, black, below = 0] {$b$};
				\filldraw[fill=black, draw=black] (5.5, 0) circle [radius=.03] node [scale = .8, black, left = 1, below = 0] {$e$};
				\fill[black] (5.5,1) circle [radius=.03] node [scale = .8, black, right = 0, above = 6, right = 0] {$f$};
				\filldraw[fill=black, draw=black] (4.5,1) circle [radius=.03] node [scale = .8, black, above = 2 ] {$c$};
				\fill[black] (6.5,1) circle [radius=.03] node [scale = .8, black, below = 2, right = 0] {$x'$};
				\fill[black] (5.5,2) circle [radius=.03] node [scale = .8, black, above = 0] {$z'$};
				\fill[black] (6.5, 2) circle [radius=.03] node [scale = .8, black, above = 0] {$y'$};

				\filldraw[fill=black, draw=black] (4.5, 0) circle [radius=.03] node [scale = .8, black, above = 35, right = 10] {$=$};

				\draw[-, black] (-4.5, -2) -- (-4.5, -3) node[below] {$a$}; 
				\draw[-, black] (-4.5, -3) -- (-3.5, -3) node[below] {$b$};
				\draw[-, black] (-3.5, -2) -- (-4.5, -2) node[left] {$d$};
				\draw[-, black] (-3.5, -3) -- (-3.5, -2) node[right] {$c$};
				\draw[-, black] (-3.5, -2) -- (-3.5, -1) node[above] {$e$};
				\draw[-, black] (-3.5, -1) -- (-4.5, -1) node[above] {$f$};
				\draw[-, black] (-4.5, -2) -- (-4.5, -1);

				\filldraw[fill=gray!50!white] (1, -3) -- (2, -2) --  (2, -1) -- (1, -2) -- (1, -3); 
				\filldraw[fill=gray!50!white] (0, -2) -- (1, -1) -- (2, -1) -- (1, -2) -- (0, -2); 
				\draw[-, black] (0, -3) -- (1,  -3); 
				\draw[-, black] (0, -3) -- (0,  -2); 
				\draw[-, black] (0, -2) -- (1,  -2); 
				\draw[-, black] (1, -3) -- (1,  -2); 
				\filldraw[fill=black, draw=black] (0, -3) circle [radius=.03] node [scale = .8, black, below = 0] {$a$};
				\filldraw[fill=black, draw=black] (1,  -3) circle [radius=.03] node [scale = .8, black, left = 1, below = 0] {$b$};
				\filldraw[fill=black, draw=black] (0, -2) circle [radius=.03] node [scale = .8, black, above = 2 ] {$d$};
				\fill[black] (1, -2) circle [radius=.03] node [scale = .8, black, right = 0, below = 6, left = 0] {$c$};
				\fill[black] (2, -2) circle [radius=.03] node [scale = .8, black, below = 2, right = 0] {$x$};
				\fill[black] (1, -1) circle [radius=.03] node [scale = .8, black, above = 0] {$z$};
				\fill[black] (2	, -1) circle [radius=.03] node [scale = .8, black, above = 0] {$y$};

				\fill[black] (1, -2) circle [radius=.03] node [scale = .8, black, right = 0, above = 14] {$=$};

				\filldraw[fill=gray!50!white] (4.5, -3) -- (5.5, -2) --  (5.5, -1) -- (4.5, -2) -- (4.5, -3); 
				\filldraw[fill=gray!50!white] (4.5, -3) -- (5.5, -2) -- (6.5, -2) -- (5.5, -3) -- (4.5, -3); 
				\draw[-, black] (5.5, -2) -- (6.5, -2); 
				\draw[-, black] (5.5, -2) -- (5.5, -1); 
				\draw[-, black] (5.5, -1) -- (6.5, -1); 
				\draw[-, black] (6.5, -2) -- (6.5, -1); 
				\filldraw[fill=black, draw=black] (4.5, -3) circle [radius=.03] node [scale = .8, black, below = 0] {$d$};
				\filldraw[fill=black, draw=black] (5.5, -3) circle [radius=.03] node [scale = .8, black, left = 1, below = 0] {$c$};
				\filldraw[fill=black, draw=black] (4.5, -2) circle [radius=.03] node [scale = .8, black, above = 2 ] {$f$};
				\fill[black] (5.5, -2) circle [radius=.03] node [scale = .8, black, right = 0, above = 6, right = 0] {$e$};
				\fill[black] (6.5, -2) circle [radius=.03] node [scale = .8, black, below = 2, right = 0] {$x'$};
				\fill[black] (5.5, -1) circle [radius=.03] node [scale = .8, black, above = 0] {$z'$};
				\fill[black] (6.5, -1) circle [radius=.03] node [scale = .8, black, above = 0] {$y'$};

				\fill[black] (5.5, -2) circle [radius=.03] node [scale = .8, black, below = 15] {$=$};

				\draw[->, black, dotted] (-5.25, 1) -- (-2.75, 1);
				\draw[->, black, dotted] (-4.5, .25) -- (-4.5, 1.75);
				\draw[->, black, dotted] (-3.5, .25) -- (-3.5, 1.75);

				\draw[->, black, dotted] (-4.75, -2.5) -- (-3.25, -2.5);
				\draw[->, black, dotted] (-4.75, -1.5) -- (-3.25, -1.5);
				\draw[->, black, dotted] (-4, -3.25) -- (-4, -.75);
				
				\fill[black] (-3.5, -3) circle [radius=.03];
				\fill[black] (-4.5, -3) circle [radius=.03];
				\fill[black] (-3.5, -2) circle [radius=.03];
				\fill[black] (-4.5, -2) circle [radius=.03];
				\fill[black] (-3.5, -1) circle [radius=.03];
				\fill[black] (-4.5, -1) circle [radius=.03];
				\end{tikzpicture}

			\end{center}

			\caption{\label{consistent} The shaded plaquettes must satisfy the two consistency relations depicted above. }
		\end{figure}
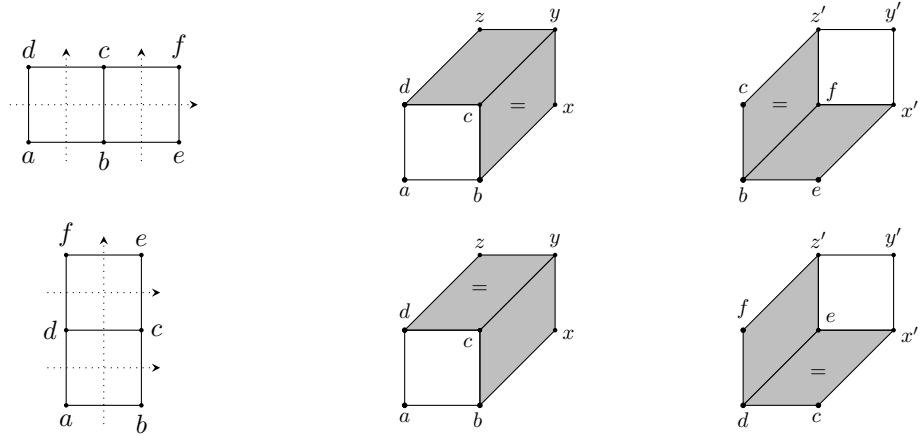

	\end{enumerate}
\end{enumerate}

\end{definition}

\begin{rem}

\label{stochasticwchis}

Once again, let us explain how the discussion from \Cref{VertexModel} can be viewed as a special case of this framework. As explained in \Cref{vertexweightsmodel}, we view the face model on $\mathcal{G}$ as a vertex model on the dual graph $\mathcal{D}$, so that the plaquette weights here become vertex weights. Now, the weights $w$ and $\chi$ here will be the $w$ and $\chi_s$ from \Cref{spin12spins}, respectively. 

We claim that $w$ is stochasticizable with respect to $\chi_s$. The first condition there was verified in \Cref{vertexadmissiblefrozen}, and the second is a consequence of the vertex form of the Yang-Baxter equation, given by \eqref{equationspin1}. To verify the third condition, we must explain how to attach two shaded vertices to each vertex $u \in \mathcal{D}$.

This was essentially done in \Cref{DynamicalVertex1}. Specifically, we begin with a stochasticization curve to the right of $\mathcal{D}$ that starts with $k$ arrows and collects any arrow outputted by $\mathcal{D}$, as in the top-left diagram of \Cref{arrowstochasticvertices}. We then move the stochasticization curve through $\mathcal{D}$, one vertex at a time using the procedure depicted in \Cref{arrowstochastic}. 

The two shaded vertices attached to $u$ are then the two vertices of the stochasticization curve adjacent to $u$ directly before being passed through $u$, as in the left side of \Cref{arrowstochastic}. These vertices are assigned the weight $\chi_s$ and can seen to be independent of the incoming arrow configuration $\big( i_1 (u), j_1 (u) \big)$. The two shaded vertices attached in the negative direction (whose duals are depicted on the right side of \Cref{assignment}) are the two vertices of the stochasticization curve adjacent to $u$ directly after being passed through $u$, as in the right side of \Cref{arrowstochastic}. 

Now one can quickly verify the four parts of the third condition of \Cref{wchiconditions}. The first is a consequence of the content of \Cref{vertexadmissiblefrozen}; the second holds for generic values of $s, x, y, k$ by \Cref{spin12spins}; and the third follows from \eqref{equationspin1}. The fourth (called the consistency condition) follows from the fact that a shaded vertex attached to some $u \in \mathcal{D}$ in the positive direction was attached to a vertex $u' \in \mathcal{D}$ adjacent to $u$ in the negative direction (this can be seen diagrammatically, from \Cref{arrowstochasticvertices}). This confirms that $w$ is stochasticizable with respect to $\chi_s$. 
	
\end{rem} 

\begin{rem} 

Our reason for introducing the consistency condition is similar to what was explained in the beginning of \Cref{DynamicalVertex1}. In the examples of interest to us, there will be a number of ways to ``locally'' assign shaded plaquettes (or vertices) to a given one in our domain. However, under an arbitrary global assignment of plaquettes to each vertex in our domain, the Yang-Baxter equation will typically no longer hold for the stochasticized weights. Imposing the consistency condition is one way of ensuring that the Yang-Baxter equation is preserved; see \Cref{consistencyvertexequation} below. 

\end{rem}

\begin{rem}

In \Cref{DynamicalVertex1}, the attachment of a pair of shaded vertices resulted in a dynamical parameter $v(u)$ (or $k(u)$) associated with each vertex $u$ in our domain. The consistency condition then essentially explains how the dynamical parameter changes between vertices, according to the identities \eqref{ukv} (see also \Cref{dynamicalsixvertexequationparameter}). A similar phenomenon will appear in later examples of stochasticization. 

\end{rem}

Now we can define the stochasticized plaquette weights analogously to \Cref{stochastic1}. 

\begin{definition}

\label{wschiweights}

Suppose that $w$ is stochasticizable with respect to $\chi$ in the sense of Definition \ref{wchiconditions}. Let $P = (v_1, v_2, v_3, v_4)$ be a plaquette and denote $\mathcal{F} (v_1) = a$, $\mathcal{F} (v_2) = b$, $\mathcal{F} (v_3) = c$, and $\mathcal{F} (v_4) = d$. Let $(x, y, z)$ denote the shaded triple associated with $P$ as given by the third part of Definition \ref{wchiconditions}, and let $\Adm_1 (x, y, z) = \{ c' \}$; see the right side of Figure \ref{assignment}. Then, define the \emph{stochasticization} $S$ of $w$ with respect to $\chi$ (and $\mathcal{F}$) by
\begin{flalign} 
\label{ws} 
S_P (a, b, c, d) = \displaystyle\frac{w (a, b, c, d) \chi (b, x, y, c) \chi (c, y, z, d)}{\chi (b, x, c', a) \chi (a, c', z, d) w (c', x, y, z)}. 
\end{flalign}

\noindent Diagrammatically, this is equivalent to
\begin{flalign*}
\begin{tikzpicture}
\node (0, 0) [left = 0	] {$S_P \Bigg( \qquad \Bigg)$}; 
\node (0, 0) [scale = .9, right = -24, above = -13] {\begin{tikzpicture}[
	style={
		scale = .65
	}
	]	
	\draw[-, black] (0, 1) -- (1, 1);
	\draw[-, black] (1, 0) -- (1, 1);
	\draw[-, black] (0, 0) -- (1, 0);
	\draw[-, black] (0, 0) -- (0, 1);
	\filldraw[fill=black, draw=black] (0, 0) circle [radius=.03] node [scale = .8, black, above = 0] {};
	\filldraw[fill=black, draw=black] (0, 1) circle [radius=.03] node [scale = .8, black, above = 0] {};
	\filldraw[fill=black, draw=black] (1, 0) circle [radius=.03] node [scale = .8, black, above = 0] {};
	\filldraw[fill=black, draw=black] (1, 1) circle [radius=.03] node [scale = .8, black, above = 0] {};
	\end{tikzpicture}};
\node (5, 0) [right = -5] {$=$};
\node (7, 0) [right = 5] {$\displaystyle\frac{\text{Weight} \Bigg( \qquad \quad \Bigg)}{\text{Weight} \Bigg( \qquad \quad \Bigg)}$.};
\node (7, 0) [above = 17, right = 45]{\begin{tikzpicture}[
	style={
		scale = .45
	}]
	\filldraw[fill=gray!50!white] (1, 0) -- (2, 1) --  (2, 2) -- (1, 1) -- (1, 0); 
	\filldraw[fill=gray!50!white] (0, 1) -- (1, 2) -- (2, 2) -- (1, 1) -- (0, 1); 
	\draw[-, black] (0, 0) -- (1,  0); 
	\draw[-, black] (0, 0) -- (0,  1); 
	\draw[-, black] (0, 1) -- (1,  1); 
	\draw[-, black] (1, 0) -- (1,  1); 
	\filldraw[fill=black, draw=black] (1,  0) circle [radius=.03] node [scale = .8, black, left = 1] {};
	\filldraw[fill=black, draw=black] (0, 0) circle [radius=.03] node [scale = .8, black, below = 0] {};
	\filldraw[fill=black, draw=black] (0, 1) circle [radius=.03] node [scale = .8, black, above = 0] {};
	\filldraw[fill=black, draw=black] (1,  1) circle [radius=.03] node [scale = .8, black, left = 2] {};
	\filldraw[fill=black, draw=black] (2,  1) circle [radius=.03] node [scale = .8, black, below = 0] {};
	\filldraw[fill=black, draw=black] (1,  2) circle [radius=.03] node [scale = .8, black, right = 1] {};
	\filldraw[fill=black, draw=black] (2, 2) circle [radius=.03] node [scale = .8, black, above = 0] {};
	\end{tikzpicture}}; 
\node (7, 0) [below = 17, right = 45] {\begin{tikzpicture}[
	style={
		scale = .45
	}]
	\filldraw[fill=gray!50!white] (0, 0) -- (1, 1) --  (1, 2) -- (0, 1) -- (0, 0); 
	\filldraw[fill=gray!50!white] (0, 0) -- (1, 1) -- (2, 1) -- (1, 0) -- (0, 0); 
	\draw[-, black] (1, 1) -- (2,  1); 
	\draw[-, black] (1, 1) -- (1,  2); 
	\draw[-, black] (1, 2) -- (2,  2); 
	\draw[-, black] (2, 1) -- (2,  2); 
	\filldraw[fill=black, draw=black] (1,  0) circle [radius=.03] node [scale = .8, black, left = 1] {};
	\filldraw[fill=black, draw=black] (0, 0) circle [radius=.03] node [scale = .8, black, below = 0] {};
	\filldraw[fill=black, draw=black] (0, 1) circle [radius=.03] node [scale = .8, black, above = 0] {};
	\fill[black] (.92, .92) rectangle (1.08, 1.08) node [scale = .8, black, right = 0] {};
	\fill[black] (1.92, .92) rectangle (2.08, 1.08) node [scale = .8, black, below = 2] {};
	\fill[black] (.92, 1.92) rectangle (1.08, 2.08) node [scale = .8, black, right = 1] {};
	\fill[black] (1.92, 1.92) rectangle (2.08, 2.08) node [scale = .8, black, above = 0] {};
	\end{tikzpicture}};
\end{tikzpicture}
\end{flalign*}

\noindent We will occasionally refer to the quantity $\frac{ S_P (a, b, c, d)}{w(a, b, c, d)}$ as a \emph{stochastic correction}.

\end{definition}

\subsection{Properties of Stochasticization} 

\label{PropertiesWeights}

In this section we establish various properties of the stochasticized weights $S$. The first is that they are indeed stochastic. 

\begin{lem} 

\label{stochasticws}

Adopt the notation of Definition \ref{wschiweights}. Then, 
\begin{flalign}
\label{stochasticweights}
  \displaystyle\sum_{m \in \mathfrak{h}} S_P (a, b, m, d) = 1, 
\end{flalign}

\noindent for any plaquette $P = (v_1, v_2, v_3, v_4)$ in our domain and $a = \mathcal{F} (v_1)$, $b = \mathcal{F} (v_2)$, and $d = \mathcal{F} (v_4)$ in $\mathfrak{h}$. Furthermore, if $w(a, b, c, d), \chi (a, b, c, d) \ge 0$ for each $a, b, c, d \in \mathfrak{h}$, then $S_P (a, b, c, d) \ge 0$ for each $a, b, c, d \in \mathfrak{h}$.

\end{lem}

\begin{proof} 

The second statement of the lemma follows from \eqref{ws}. The first, given by \eqref{stochasticweights}, is equivalent to \eqref{wchirelation} in view of \eqref{ws}. 	
\end{proof}

Let us proceed to explain the Yang-Baxter equation that holds for the stochasticized weights. To that end, let $\ell_1, \ell_2, \ell_3 \in \mathcal{L}$ be three lines that are \emph{adjacent}, meaning that they intersect to form three neighboring vertices in $\mathcal{D}$. Denote the corresponding plaquettes in $\mathcal{G}$ by $P$, $Q$, and $R$, and denote the height function are the vertices of $P$, $Q$, and $R$ by $(a, b, c, d)$, $(b, e, f, c)$, and $(c, f, g, d)$, respectively. This is depicted in the summand on the left side of \eqref{unshadedequation}. 

Now, let us perturb $\mathcal{L}$ by shifting $\ell_3$ past $\ell_1$ and $\ell_2$. This forms a family $\mathcal{L}'$ of lines, corresponding to a perturbed graph $\mathcal{D}'$, in which the new $\ell_1, \ell_2, \ell_3 \in \mathcal{L}'$ intersect to form three adjacent vertices. We denote the corresponding plaquettes by $P'$, $Q'$, and $R'$; the height function at the vertices of these plaquettes are $(c, e, f, g)$, $(a, c, g, d)$, and $(b, e, c, a)$, respectively. This is depicted in the summand on the right side of \eqref{unshadedequation}.

Next, we attach shaded plaquettes to $Q$, $R$, $P'$, and $Q'$, as given by the third part of \Cref{wchiconditions}; this is depicted in the top-left and bottom-left diagrams of \Cref{3equationshadedunshaded}. In what follows, we adopt the notation from the figure. In particular, we assume that the same plaquettes are attached in both the original (top-left) graph and the perturbed (bottom-left) one.

\begin{figure} 
	
	\begin{center}

		\begin{tikzpicture}[
		>=stealth,
		auto,
		style={
			scale = 1
		}
		]

		\filldraw[fill = white]  (.3, 5.7) -- (1.7, 5.7) -- (2, 5) -- (1, 5) -- (.3, 5.7);
		\filldraw[fill = white] (.3, 4.7) -- (.3, 5.7) -- (1, 5) -- (1, 4) -- (.3, 4.7);
		\filldraw[fill = white] (1, 4) -- (2, 4) -- (2, 5) -- (1, 5) -- (1, 4); 
		
		\filldraw[fill = white!50!gray] (.3, 5.7) -- (1.7, 5.7) -- (2, 6.4) -- (1, 6.4) -- (.3, 5.7);
		\filldraw[fill = white!50!gray] (1.7, 5.7) -- (2, 5) -- (2.7, 5.7) -- (2, 6.4) -- (1.7, 5.7);	
		\filldraw[fill = white!50!gray] (2, 5) -- (2.7, 5.7) -- (2.7, 4.7) -- (2, 4) -- (2, 5);

		\filldraw[fill = white]  (5.7, 5.7) -- (6.7, 5.7) -- (6, 6.4) -- (5, 6.4) -- (5.7, 5.7);
		\filldraw[fill = white] (4.3, 4.7) -- (4.3, 5.7) -- (5, 5) -- (5, 4) -- (4.3, 4.7);
		\filldraw[fill = white] (5, 4) -- (6, 4) -- (6, 5) -- (5, 5) -- (5, 4); 
		
		\filldraw[fill = white!50!gray] (4.3, 5.7) -- (5, 6.4) -- (5.7, 5.7) -- (5, 5) -- (4.3, 5.7);
		\filldraw[fill = white!50!gray] (5.7, 5.7) -- (5, 5) -- (6, 5) -- (6.7, 5.7) -- (5.7, 5.7);	
		\filldraw[fill = white!50!gray] (6, 5) -- (6.7, 5.7) -- (6.7, 4.7) -- (6, 4) -- (6, 5);

		\filldraw[fill = white]  (9.7, 5.7) -- (10.7, 5.7) -- (10, 6.4) -- (9, 6.4) -- (9.7, 5.7);
		\filldraw[fill = white] (8.3, 4.7) -- (8.3, 5.7) -- (9, 5) -- (9, 4) -- (8.3, 4.7);
		\filldraw[fill = white] (10, 5) -- (9.7, 5.7) -- (10.7, 5.7) -- (10.7, 4.7) -- (10, 5); 
		
		\filldraw[fill = white!50!gray] (8.3, 5.7) -- (9, 6.4) -- (9.7, 5.7) -- (9, 5) -- (8.3, 5.7);
		\filldraw[fill = white!50!gray] (9.7, 5.7) -- (9, 5) -- (9, 4) -- (10, 5) -- (9.7, 5.7);	
		\filldraw[fill = white!50!gray] (10, 5) -- (10.7, 4.7) -- (10, 4) -- (9, 4) -- (10, 5);

		\filldraw[fill = white]  (13.7, 5.7) -- (14.7, 5.7) -- (14, 6.4) -- (13, 6.4) -- (13.7, 5.7);
		\filldraw[fill = white] (13, 5) -- (14, 5) -- (13.7, 5.7) -- (13, 6.4) -- (13, 5);
		\filldraw[fill = white] (14, 5) -- (13.7, 5.7) -- (14.7, 5.7) -- (14.7, 4.7) -- (14, 5); 
		
		\filldraw[fill = white!50!gray] (13, 5) -- (12.3, 4.7) -- (12.3, 5.7)  -- (13, 6.4) -- (13, 5);
		\filldraw[fill = white!50!gray] (13, 5) -- (14, 5) -- (13, 4) -- (12.3, 4.7) -- (13, 5);	
		\filldraw[fill = white!50!gray] (14, 5) -- (14.7, 4.7) -- (14, 4) -- (13, 4) -- (14, 5);

		\filldraw[fill=black, draw=black] (1, 4) circle [radius=.03] node [scale = .8, black, below = 0] {$b$};
		\filldraw[fill=black, draw=black] (2, 4) circle [radius=.03] node [scale = .8, black, below = 0] {$e$};
		\filldraw[fill=black, draw=black] (.3, 4.7) circle [radius=.03] node [scale = .8, black, below = 2] {$a$};
		\filldraw[fill=black, draw=black] (.3, 5.7) circle [radius=.03] node [scale = .8, black, above = 0] {$d$};
		\filldraw[fill=black, draw=black] (1, 6.4) circle [radius=.03] node [scale = .8, black, above = 0] {$z$};
		\filldraw[fill=black, draw=black] (2, 6.4) circle [radius=.03] node [scale = .8, black, above = 0] {$y$};
		\filldraw[fill=black, draw=black] (2.7, 4.7) circle [radius=.03] node [scale = .8, black, right = 0] {$v$};
		\filldraw[fill=black, draw=black] (2.7, 5.7) circle [radius=.03] node [scale = .8, black, right = 0] {$x$};

		\filldraw[fill=black, draw=black] (1, 5) circle [radius=.03] node [scale = .8, black, below = 5, right = 0] {$c$};	
		\filldraw[fill=black, draw=black] (2, 5) circle [radius=.03] node [scale = .8, black, below = 9, left = -2] {$f$};
		\filldraw[fill=black, draw=black] (1.7, 5.7) circle [radius=.03] node [scale = .8, black, left =2, below = 0] {$g$};

		\filldraw[fill = black, draw = black] (6.08, 6.32) -- (5.92, 6.32) -- (5.92, 6.48) -- (6.08, 6.48) -- (6.08, 6.32);
		\filldraw[fill = black, draw = black] (5.08, 6.32) -- (4.92, 6.32) -- (4.92, 6.48) -- (5.08, 6.48) -- (5.08, 6.32);
		\filldraw[fill = black, draw = black] (6.78, 5.62) -- (6.62, 5.62) -- (6.62, 5.78) -- (6.78, 5.78) -- (6.78, 5.62);
		\filldraw[fill = black, draw = black] (5.78, 5.62) -- (5.62, 5.62) -- (5.62, 5.78) -- (5.78, 5.78) -- (5.78, 5.62);
		
		\filldraw[fill=black, draw=black] (6, 4) circle [radius=.03] node [scale = .8, black, below = 0] {$e$};
		\filldraw[fill=black, draw=black] (5, 4) circle [radius=.03] node [scale = .8, black, below = 0] {$b$};
		\filldraw[fill=black, draw=black] (4.3, 4.7) circle [radius=.03] node [scale = .8, black, left = 0] {$a$};
		\filldraw[fill=black, draw=black] (4.3, 5.7) circle [radius=.03] node [scale = .8, black, left = 0] {$d$};
		\filldraw[fill=black, draw=black] (5, 6.4) circle [radius=.03] node [scale = .8, black, above = 2] {$z$};
		\filldraw[fill=black, draw=black] (6, 6.4) circle [radius=.03] node [scale = .8, black, above = 2] {$y$};
		\filldraw[fill=black, draw=black] (6.7, 4.7) circle [radius=.03] node [scale = .8, black, right = 0] {$v$};
		\filldraw[fill=black, draw=black] (6.7, 5.7) circle [radius=.03] node [scale = .8, black, right = 2] {$x$};

		\filldraw[fill=black, draw=black] (5, 5) circle [radius=.03] node [scale = .8, black, below = 5, right = 0] {$c$};
		\filldraw[fill=black, draw=black] (6, 5) circle [radius=.03] node [scale = .8, black, below = 9, left = -2] {$f$};
		\filldraw[fill=black, draw=black] (5.7, 5.7) circle [radius=.03] node [scale = .8, black, above = 0] {$g'$};
		
		\filldraw[fill=black, draw=black] (10, 4) circle [radius=.03] node [scale = .8, black, below = 0] {$e$};
		\filldraw[fill=black, draw=black] (9, 4) circle [radius=.03] node [scale = .8, black, below = 0] {$b$};
		\filldraw[fill=black, draw=black] (8.3, 4.7) circle [radius=.03] node [scale = .8, black, left = 0] {$a$};
		\filldraw[fill=black, draw=black] (8.3, 5.7) circle [radius=.03] node [scale = .8, black, left = 0] {$d$};
		\filldraw[fill=black, draw=black] (9, 6.4) circle [radius=.03] node [scale = .8, black, above = 2] {$z$};
		\filldraw[fill=black, draw=black] (10, 6.4) circle [radius=.03] node [scale = .8, black, above = 2] {$y$};
		\filldraw[fill=black, draw=black] (10.7, 4.7) circle [radius=.03] node [scale = .8, black, right = 2] {$v$};
		\filldraw[fill=black, draw=black] (10.7, 5.7) circle [radius=.03] node [scale = .8, black, right = 2] {$x$};

		\filldraw[fill=black, draw=black] (9, 5) circle [radius=.03] node [scale = .8, black, below = 2, left = 0] {$c$};
		\filldraw[fill=black, draw=black] (10, 5) circle [radius=.03] node [scale = .8, black, above = 5, right = 0] {$f'$};
		\filldraw[fill=black, draw=black] (9.7, 5.7) circle [radius=.03] node [scale = .8, black, above = 0] {$g'$};

		\filldraw[fill = black, draw = black] (10.08, 6.32) -- (9.92, 6.32) -- (9.92, 6.48) -- (10.08, 6.48) -- (10.08, 6.32);
		\filldraw[fill = black, draw = black] (9.08, 6.32) -- (8.92, 6.32) -- (8.92, 6.48) -- (9.08, 6.48) -- (9.08, 6.32);
		\filldraw[fill = black, draw = black] (10.78, 5.62) -- (10.62, 5.62) -- (10.62, 5.78) -- (10.78, 5.78) -- (10.78, 5.62);
		\filldraw[fill = black, draw = black] (10.78, 4.62) -- (10.62, 4.62) -- (10.62, 4.78) -- (10.78, 4.78) -- (10.78, 4.62);
		\filldraw[fill = black, draw = black] (10.08, 4.92) -- (9.92, 4.92) -- (9.92, 5.08) -- (10.08, 5.08) -- (10.08, 4.92);
		\filldraw[fill = black, draw = black] (9.78, 5.62) -- (9.62, 5.62) -- (9.62, 5.78) -- (9.78, 5.78) -- (9.78, 5.62);
		
		\filldraw[fill=black, draw=black] (14, 4) circle [radius=.03] node [scale = .8, black, below = 0] {$e$};
		\filldraw[fill=black, draw=black] (13, 4) circle [radius=.03] node [scale = .8, black, below = 0] {$b$};
		\filldraw[fill=black, draw=black] (12.3, 4.7) circle [radius=.03] node [scale = .8, black, left = 0] {$a$};
		\filldraw[fill=black, draw=black] (12.3, 5.7) circle [radius=.03] node [scale = .8, black, left = 0] {$d$};
		\filldraw[fill=black, draw=black] (13, 6.4) circle [radius=.03] node [scale = .8, black, above = 2] {$z$};
		\filldraw[fill=black, draw=black] (14, 6.4) circle [radius=.03] node [scale = .8, black, above = 2] {$y$};
		\filldraw[fill=black, draw=black] (14.7, 4.7) circle [radius=.03] node [scale = .8, black, below = 2] {$v$};
		\filldraw[fill=black, draw=black] (14.7, 5.7) circle [radius=.03] node [scale = .8, black, above = 2] {$x$};

		\filldraw[fill=black, draw=black] (13, 5) circle [radius=.03] node [scale = .8, black, above = 7, right = 0] {$c'$};
		\filldraw[fill=black, draw=black] (14, 5) circle [radius=.03] node [scale = .8, black, above = 5, right = 0] {$f'$};
		\filldraw[fill=black, draw=black] (13.7, 5.7) circle [radius=.03] node [scale = .8, black, above = 0] {$g'$};

		\filldraw[fill = black, draw = black] (14.08, 6.32) -- (13.92, 6.32) -- (13.92, 6.48) -- (14.08, 6.48) -- (14.08, 6.32);
		\filldraw[fill = black, draw = black] (13.08, 6.32) -- (12.92, 6.32) -- (12.92, 6.48) -- (13.08, 6.48) -- (13.08, 6.32);
		\filldraw[fill = black, draw = black] (14.78, 5.62) -- (14.62, 5.62) -- (14.62, 5.78) -- (14.78, 5.78) -- (14.78, 5.62);
		\filldraw[fill = black, draw = black] (14.78, 4.62) -- (14.62, 4.62) -- (14.62, 4.78) -- (14.78, 4.78) -- (14.78, 4.62);
		\filldraw[fill = black, draw = black] (13.78, 5.62) -- (13.62, 5.62) -- (13.62, 5.78) -- (13.78, 5.78) -- (13.78, 5.62);
		\filldraw[fill = black, draw = black] (13.08, 4.92) -- (12.92, 4.92) -- (12.92, 5.08) -- (13.08, 5.08) -- (13.08, 4.92);
		\filldraw[fill = black, draw = black] (14.08, 4.92) -- (13.92, 4.92) -- (13.92, 5.08) -- (14.08, 5.08) -- (14.08, 4.92);

		\filldraw[fill = white]  (.3, 1.7) -- (.3, .7) -- (1, 1) --  (.9, 1.6) -- (.3, 1.7);
		\filldraw[fill = white] (.3, .7) -- (1, 1) -- (2, 0) -- (1, 0) -- (.3, .7);
		\filldraw[fill = white] (2, 0) -- (2, 1.7) -- (.9, 1.6) -- (1, 1) -- (2, 0); 
		
		\filldraw[fill = white!50!gray] (.3, 1.7) -- (.9, 1.6) -- (2, 2.4) -- (1, 2.4) -- (.3, 1.7);
		\filldraw[fill = white!50!gray] (.9, 1.6) -- (2, 1.7) -- (2.7, 1.7) -- (2, 2.4) -- (.9, 1.6);	
		\filldraw[fill = white!50!gray] (2, 1.7) -- (2.7, 1.7) -- (2.7, .7) -- (2, 0) -- (2, 1.7);

		\filldraw[fill = white]  (4.3, 1.7) -- (4.3, .7) -- (5, 1) --  (4.9, 1.6) -- (4.3, 1.7);
		\filldraw[fill = white] (4.3, .7) -- (5, 1) -- (6, 0) -- (5, 0) -- (4.3, .7);
		\filldraw[fill = white] (6, 1.7) -- (6.7, .7) -- (6.7, 1.7) -- (6, 2.4) -- (6, 1.7); 
		
		\filldraw[fill = white!50!gray] (4.3, 1.7) -- (4.9, 1.6) -- (6, 2.4) -- (5, 2.4) -- (4.3, 1.7);
		\filldraw[fill = white!50!gray] (4.9, 1.6) -- (6, 2.4) -- (6, 1.7) -- (5, 1) -- (4.9, 1.6);	
		\filldraw[fill = white!50!gray] (6, 1.7) -- (6.7, .7) -- (6, 0) -- (5, 1) -- (6, 1.7);

		\filldraw[fill = white] (8.9, 1.6) -- (9, 2.4) -- (10, 2.4) -- (10, 1.7) -- (8.9, 1.6);
		\filldraw[fill = white] (8.3, .7) -- (9, 1) -- (10, 0) -- (9, 0) -- (8.3, .7);
		\filldraw[fill = white] (10, 1.7) -- (10.7, .7) -- (10.7, 1.7) -- (10, 2.4) -- (10, 1.7); 
		
		\filldraw[fill = white!50!gray] (8.3, 1.7) -- (8.3, .7) -- 	(8.9, 1.6) -- (9, 2.4) -- (8.3, 1.7);
		\filldraw[fill = white!50!gray] (8.9, 1.6) -- (8.3, .7) -- (9, 1) -- (10, 1.7) -- (8.9, 1.6);	
		\filldraw[fill = white!50!gray] (10, 1.7) -- (10.7, .7) -- (10, 0) -- (9, 1) -- (10, 1.7);		
		
		\filldraw[fill = white] (12.9, 1.6) -- (13, 2.4) -- (14, 2.4) -- (14, 1.7) -- (12.9, 1.6);
		\filldraw[fill = white] (14, 1.7) -- (14.7, .7) -- (14.7, 1.7) -- (14, 2.4) -- (14, 1.7); 
		\filldraw[fill = white] (13, 1) -- (14.7, .7) -- (14, 1.7) -- (12.9, 1.6) -- (13, 1); 
		
		\filldraw[fill = white!50!gray] (13, 1) -- (13, 0) -- (12.3, .7)  -- (12.9, 1.6) -- (13, 1);
		\filldraw[fill = white!50!gray] (13, 1) -- (13, 0) -- (14, 0) -- (14.7, .7) -- (13, 1);	
		\filldraw[fill = white!50!gray] (12.3, .7) -- (12.3, 1.7) -- (13, 2.4) -- (12.9, 1.6) -- (12.3, .7);

		\filldraw[fill=black, draw=black] (1, 0) circle [radius=.03] node [scale = .8, black, below = 0] {$b$};
		\filldraw[fill=black, draw=black] (2, 0) circle [radius=.03] node [scale = .8, black, below = 0] {$e$};
		\filldraw[fill=black, draw=black] (.3, .7) circle [radius=.03] node [scale = .8, black, below = 2] {$a$};
		\filldraw[fill=black, draw=black] (.3, 1.7) circle [radius=.03] node [scale = .8, black, above = 0] {$d$};
		\filldraw[fill=black, draw=black] (1, 2.4) circle [radius=.03] node [scale = .8, black, above = 0] {$z$};
		\filldraw[fill=black, draw=black] (2, 2.4) circle [radius=.03] node [scale = .8, black, above = 0] {$y$};
		\filldraw[fill=black, draw=black] (2.7, .7) circle [radius=.03] node [scale = .8, black, right = 0] {$v$};
		\filldraw[fill=black, draw=black] (2.7, 1.7) circle [radius=.03] node [scale = .8, black, right = 0] {$x$};

		\filldraw[fill=black, draw=black] (1, 1) circle [radius=.03] node [scale = .8, black, below = 0] {$c$};	
		\filldraw[fill=black, draw=black] (2, 1.7) circle [radius=.03] node [scale = .8, black, below = 9, left = -2] {$f$};
		\filldraw[fill=black, draw=black] (.9, 1.6) circle [radius=.03] node [scale = .8, black, left = 4, below = 0] {$g$};

		\filldraw[fill = black, draw = black] (6.08, 2.32) -- (5.92, 2.32) -- (5.92, 2.48) -- (6.08, 2.48) -- (6.08, 2.32);
		\filldraw[fill = black, draw = black] (6.08, 1.62) -- (5.92, 1.62) -- (5.92, 1.78) -- (6.08, 1.78) -- (6.08, 1.62);
		\filldraw[fill = black, draw = black] (6.78, 1.62) -- (6.62, 1.62) -- (6.62, 1.78) -- (6.78, 1.78) -- (6.78, 1.62);
		\filldraw[fill = black, draw = black] (6.78, .62) -- (6.62, .62) -- (6.62, .78) -- (6.78, .78) -- (6.78, .62);
		
		\filldraw[fill=black, draw=black] (6, 0) circle [radius=.03] node [scale = .8, black, below = 0] {$e$};
		\filldraw[fill=black, draw=black] (5, 0) circle [radius=.03] node [scale = .8, black, below = 0] {$b$};
		\filldraw[fill=black, draw=black] (4.3, .7) circle [radius=.03] node [scale = .8, black, left = 0] {$a$};
		\filldraw[fill=black, draw=black] (4.3, 1.7) circle [radius=.03] node [scale = .8, black, left = 0] {$d$};
		\filldraw[fill=black, draw=black] (5, 2.4) circle [radius=.03] node [scale = .8, black, above = 2] {$z$};
		\filldraw[fill=black, draw=black] (6, 2.4) circle [radius=.03] node [scale = .8, black, above = 2] {$y$};
		\filldraw[fill=black, draw=black] (6.7, .7) circle [radius=.03] node [scale = .8, black, right = 2] {$v$};
		\filldraw[fill=black, draw=black] (6.7, 1.7) circle [radius=.03] node [scale = .8, black, right = 2] {$x$};

		\filldraw[fill=black, draw=black] (5, 1) circle [radius=.03] node [scale = .8, black, below = 0] {$c$};
		\filldraw[fill=black, draw=black] (6, 1.7) circle [radius=.03] node [scale = .8, black, right = 0] {$f''$};
		\filldraw[fill=black, draw=black] (4.9, 1.6) circle [radius=.03] node [scale = .8, black, left = 4, below = 0] {$g$};
		
		\filldraw[fill=black, draw=black] (10, 0) circle [radius=.03] node [scale = .8, black, below = 0] {$e$};
		\filldraw[fill=black, draw=black] (9, 0) circle [radius=.03] node [scale = .8, black, below = 0] {$b$};
		\filldraw[fill=black, draw=black] (8.3, .7) circle [radius=.03] node [scale = .8, black, left = 0] {$a$};
		\filldraw[fill=black, draw=black] (8.3, 1.7) circle [radius=.03] node [scale = .8, black, left = 0] {$d$};
		\filldraw[fill=black, draw=black] (9, 2.4) circle [radius=.03] node [scale = .8, black, above = 2] {$z$};
		\filldraw[fill=black, draw=black] (10, 2.4) circle [radius=.03] node [scale = .8, black, above = 2] {$y$};
		\filldraw[fill=black, draw=black] (10.7, .7) circle [radius=.03] node [scale = .8, black, right = 2] {$v$};
		\filldraw[fill=black, draw=black] (10.7, 1.7) circle [radius=.03] node [scale = .8, black, right = 2] {$x$};

		\filldraw[fill=black, draw=black] (9, 1) circle [radius=.03] node [scale = .8, black, below = 2, below = 0] {$c$};
		\filldraw[fill=black, draw=black] (10, 1.7) circle [radius=.03] node [scale = .8, black, above = 3, right = 1] {$f''$};
		\filldraw[fill=black, draw=black] (8.9, 1.6) circle [radius=.03] node [scale = .8, black, above = 7, right = 0] {$g''$};

		\filldraw[fill = black, draw = black] (10.08, 2.32) -- (9.92, 2.32) -- (9.92, 2.48) -- (10.08, 2.48) -- (10.08, 2.32);
		\filldraw[fill = black, draw = black] (9.08, 2.32) -- (8.92, 2.32) -- (8.92, 2.48) -- (9.08, 2.48) -- (9.08, 2.32);
		\filldraw[fill = black, draw = black] (10.78, 1.62) -- (10.62, 1.62) -- (10.62, 1.78) -- (10.78, 1.78) -- (10.78, 1.62);
		\filldraw[fill = black, draw = black] (10.78, .62) -- (10.62, .62) -- (10.62, .78) -- (10.78, .78) -- (10.78, .62);
		\filldraw[fill = black, draw = black] (10.08, 1.62) -- (9.92, 1.62) -- (9.92, 1.78) -- (10.08, 1.78) -- (10.08, 1.62);
		\filldraw[fill = black, draw = black] (8.98, 1.52) -- (8.82, 1.52) -- (8.82, 1.68) -- (8.98, 1.68) -- (8.98, 1.52);
		
		\filldraw[fill=black, draw=black] (14, 0) circle [radius=.03] node [scale = .8, black, below = 0] {$e$};
		\filldraw[fill=black, draw=black] (13, 0) circle [radius=.03] node [scale = .8, black, below = 0] {$b$};
		\filldraw[fill=black, draw=black] (12.3, .7) circle [radius=.03] node [scale = .8, black, left = 0] {$a$};
		\filldraw[fill=black, draw=black] (12.3, 1.7) circle [radius=.03] node [scale = .8, black, left = 0] {$d$};
		\filldraw[fill=black, draw=black] (13, 2.4) circle [radius=.03] node [scale = .8, black, above = 2] {$z$};
		\filldraw[fill=black, draw=black] (14, 2.4) circle [radius=.03] node [scale = .8, black, above = 2] {$y$};
		\filldraw[fill=black, draw=black] (14.7, .7) circle [radius=.03] node [scale = .8, black, below = 2] {$v$};
		\filldraw[fill=black, draw=black] (14.7, 1.7) circle [radius=.03] node [scale = .8, black, above = 2] {$x$};

		\filldraw[fill=black, draw=black] (13, 1) circle [radius=.03] node [scale = .8, black, above = 7, right = 0] {$c''$};
		\filldraw[fill=black, draw=black] (14, 1.7) circle [radius=.03] node [scale = .8, black, above = 3, right = 2] {$f''$};
		\filldraw[fill=black, draw=black] (12.9, 1.6) circle [radius=.03] node [scale = .8, black, above = 6, right = 0] {$g''$};

		\filldraw[fill = black, draw = black] (14.08, 2.32) -- (13.92, 2.32) -- (13.92, 2.48) -- (14.08, 2.48) -- (14.08, 2.32);
		\filldraw[fill = black, draw = black] (13.08, 2.32) -- (12.92, 2.32) -- (12.92, 2.48) -- (13.08, 2.48) -- (13.08, 2.32);
		\filldraw[fill = black, draw = black] (14.78, 1.62) -- (14.62, 1.62) -- (14.62, 1.78) -- (14.78, 1.78) -- (14.78, 1.62);
		\filldraw[fill = black, draw = black] (14.78, .62) -- (14.62, .62) -- (14.62, .78) -- (14.78, .78) -- (14.78, .62);
		\filldraw[fill = black, draw = black] (12.98, 1.52) -- (12.82, 1.52) -- (12.82, 1.68) -- (12.98, 1.68) -- (12.98, 1.52);
		\filldraw[fill = black, draw = black] (13.08, .92) -- (12.92, .92) -- (12.92, 1.08) -- (13.08, 1.08) -- (13.08, .92);
		\filldraw[fill = black, draw = black] (14.08, 1.62) -- (13.92, 1.62) -- (13.92, 1.78) -- (14.08, 1.78) -- (14.08, 1.62);

		\end{tikzpicture}
		
	\end{center}

	\caption{\label{3equationshadedunshaded} Shown above is process of moving the shaded plaquettes through our domain, which is used to establish the Yang-Baxter equation. }
\end{figure}
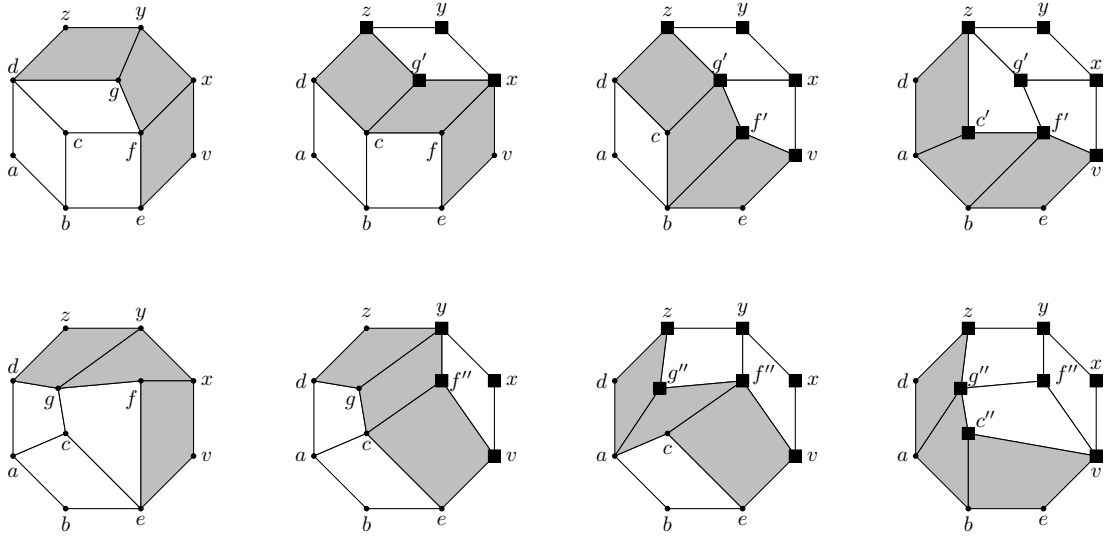

Then we can establish the Yang-Baxter equation in this setting.

\begin{thm} 
	
	\label{relationwstochastic} 
	
	Under the notation and assumptions explained above, we have that 
	\begin{flalign}
	\label{sequation}
	\begin{aligned}
	& \displaystyle\sum_{c \in \mathfrak{h}} S_P (a, b, c, d) S_Q (b, e, f, c) S_R (c, f, g, d)  = \displaystyle\sum_{c \in \mathfrak{h}}  S_{P'} (c, e, f, g) S_{Q'} (a, c, g, d) S_{R'} (b, e, c, a). 
	\end{aligned}
	\end{flalign}	
	
\end{thm} 

\begin{proof} 
	
We will establish this theorem by showing that both sides of \Cref{sequation} are the same multiples of the corresponding sides of \eqref{unshadedequation1}. 

To that end, let us begin by explaining the diagrams in \Cref{3equationshadedunshaded}. Since $(x, y, z)$ is the shaded triple associated with the plaquette $(c, f, d, g)$, there exists a unique $g' \in \mathfrak{h}$ such that $(g', x, y, z)$ is admissible. Then, $(f, x, g', c)$ and $(c, g', z, d)$ are the two plaquettes attached to the frozen plaquette $(g', x, y, z)$ along the negative direction, as in the right diagram in \Cref{assignment}; this is depicted in the second diagram on the top row of \Cref{3equationshadedunshaded}. 

Now, the consistency condition implies that $(v, x, g')$ is the shaded triple associated with the plaquette $(b, e, f, c)$, and so there again exists a unique $f' \in \mathfrak{h}$ such that $(f', v, x, g')$ is admissible. Continuing in this manner, we obtain the third and fourth diagrams on the top row of \Cref{3equationshadedunshaded}. Applying the same reasoning to the bottom-left diagram in \Cref{3equationshadedunshaded} yields the bottom row of that figure. 

This gives rise to triples  $(f', g', c')$ and $(c'', f'', g'')$ and the top and bottom row, respectively. We claim that these two triples are equal, that is, $f' = c''$, $g' = f''$, and $c' = g''$. To that end, we apply the Yang-Baxter equation \eqref{unshadedequation1} and the facts that $\Adm_1 (x, y, z) = \{ g' \}$ and $\Adm_1 (v, x, y) = \{ f'' \}$ to deduce that  
\begin{flalign}
\label{fgc}
w (g', x, y, z) w (f', v, x, g') w (c', f', g', z) = w (f'', v, x, y) w (f', v, f'', c') w (c', f'', y, z).
\end{flalign}

\noindent Now, since the left side of \eqref{fgc} is nonzero, the right side is also nonzero. Thus, $w (c', f'', y, z) \ne 0$. However, as depicted in the bottom-right diagram of \Cref{3equationshadedunshaded}, the triple $(f'', y, z)$ is frozen with $\Adm_1 (f'', y, z) = \{ g'' \}$. Hence, $c' = g''$. Similar reasoning yields $f' = c''$ and $g' = f''$. 

Now, let us evaluate the $S$ weights. By \Cref{wschiweights}, we obtain that 
\begin{flalign*}
S_P (a, b, c, d) & = \displaystyle\frac{w (a, b, c, d) \chi (c, g', z, d) \chi (b, f', g', c)}{\chi (b, f', c', a) \chi (a, c', z, d)} \displaystyle\frac{1}{ w (c', f', g', z)}; \\
S_Q (b, e, f, c) & =   \displaystyle\frac{w (b, e, f, c)\chi (e, v, x, f) \chi (f, x, g', c)}{\chi (e, v, f', b) \chi (b, f', g', c)} \displaystyle\frac{1}{ w (f', v, x, g')}; \\
S_R (c, f, g, d) & =  \displaystyle\frac{w (c, f, g, d) \chi (f, x, y, g) \chi (g, y, z, d)}{\chi (f, x, g', c) \chi (c, g', z, d)} \displaystyle\frac{1}{ w (g', x, y, z)},
\end{flalign*}

\noindent so that the product of the stochastic corrections at $P$, $Q$, and $R$ is equal to 
\begin{flalign*}	
& \displaystyle\frac{S_P (a, b, c, d) S_Q (b, e, f, c) S_R (c, f, g, d)}{w (a, b, c, d) w (b, e, f, c)  w (c, f, g, d) } \\
& \quad =  \displaystyle\frac{\chi (e, v, x, f)\chi (f, x, y, g) \chi (g, y, z, d) }{\chi (e, v, f', b) \chi (b, f', c', a) \chi (a, c', z, d)} \displaystyle\frac{1}{ w (c', f', g', z) w (f', v, x, g') w (g', x, y, z)}.
\end{flalign*}

\noindent Similar reasoning using the lower row of \Cref{3equationshadedunshaded} implies that the product of the stochastic corrections at $P'$, $Q'$, and $R'$ is equal to
\begin{flalign}
\label{swpqr1}
\begin{aligned}	
& \displaystyle\frac{S_{P'} (c, e, f, g) S_{Q'} (a, c, g, d) S_{R'} (b, e, c, a)}{w (c, e, f, g) w (a, c, g, d)  w (b, e, c, a) } \\
& \quad =  \displaystyle\frac{\chi (e, v, x, f)\chi (f, x, y, g) \chi (g, y, z, d) }{\chi (e, v, c'', b) \chi (b, c'', g'', a) \chi (a, g'', z, d)} \displaystyle\frac{1}{ w (g'', f'', y, z) w (f'', v, x, y) w (c'', v, f'', g'')}. 
\end{aligned} 
\end{flalign}

\noindent Thus, \eqref{fgc} and the fact that $(f', g', c') = (c'', f'', g'')$ implies that 
\begin{flalign} 
\label{swpqr}
\displaystyle\frac{S_P (a, b, c, d) S_Q (b, e, f, c) S_R (c, f, g, d)}{w (a, b, c, d) w (b, e, f, c)  w (c, f, g, d) }  = \displaystyle\frac{S_{P'} (c, e, f, g) S_{Q'} (a, c, g, d) S_{R'} (b, e, c, a)}{w (c, e, f, g) w (a, c, g, d)  w (b, e, c, a) }.
\end{flalign} 

\noindent Furthermore, both sides of \eqref{swpqr} are independent of $c$ since the same is true about each index appearing on the right side of \eqref{swpqr1}. Thus, \eqref{sequation} is obtained by multiplying both sides of \eqref{unshadedequation1} by either side of \eqref{swpqr}, from which we deduce the theorem. 
\end{proof}

\begin{rem}

\label{consistencyvertexequation}

As in \Cref{dynamicalsixvertexequationparameter}, what allowed for the proof of \Cref{relationwstochastic} was that the quantities appearing on either side of \eqref{swpqr} were independent of the interior vertex $c$ over which the sums in \eqref{sequation} were taken. This independence was essentially stipulated by the consistency condition, which imposed that any $c$-dependent $\chi$-weight appearing in the numerator of \eqref{swpqr1} must also appear in the denominator.
\end{rem}

\section{A Stochastic Elliptic Model} 

\label{StochasticElliptic}

In this section we explain how to use to the framework of \Cref{Domain} to stochasticize solutions to the dynamical Yang-Baxter equation coming from the fully fused eight-vertex SOS model. We begin in Section \ref{Hypergeometric} with some notation on elliptic functions, Pochhammer symbols, and hypergeometric series. Then, in Section \ref{EllipticFused} we provide and stochasticize the fused solution to the dynamical Yang-Baxter equation; in Section \ref{EllipticStochasticC} we evaluate these stochasticized weights; and in Section \ref{Elliptic01} we provide degenerations of these weights to the six-vertex and higher spin cases.

\subsection{Theta Functions and Hypergeometric Series}

\label{Hypergeometric}

Throughout this section, we fix $\eta, \tau \in \mathbb{C}$ with $\Im \tau > 0$. We will repeatedly come across the \emph{first Jacobi theta function} $\theta_1 (z; \tau) = \theta (z; \tau) = \theta (z) = f(z)$, defined (for any $z \in \mathbb{C}$) by 
\begin{flalign}
\label{theta1}
f (z) = - \displaystyle\sum_{j = - \infty}^{\infty} \exp \left( \pi \textbf{i} \tau \left( j + \displaystyle\frac{1}{2} \right)^2 + 2 \pi \textbf{i} \left( j + \displaystyle\frac{1}{2} \right) \left( z + \displaystyle\frac{1}{2} \right) \right), 
\end{flalign}

\noindent which converges since $\Im \tau > 0$ and satisfies (see, for instance, equation (2.1.19) of \cite{ESSOSM})
\begin{flalign}
\label{quarticrelationf}
\begin{aligned}
f(x + z) & f(x - z) f(y + w) f(y - w) \\
& = f(x + y) (x - y) f(z + w) f (z - w) + f(x + w) f(x - w) f(y + z) f(y - z), 
\end{aligned}
\end{flalign}

\noindent for any $w, x, y, z \in \mathbb{C}$.

For any complex number $a$ and nonnegative integer $k$, the \textit{rational Pochhammer symbol} $(a)_k$; \emph{$q$-Pochhammer symbol} $(a; q)_k$ and \emph{elliptic Pochhammer symbol} $[a]_k = [a; \eta]_k$ are defined by 
\begin{flalign}
\label{productbasicelliptic}
(a)_k = \prod_{j = 0}^{k - 1} (a + j); \qquad (a; q)_k = \prod_{j = 0}^{k - 1} (1 - q^j a); \qquad [a]_k = \prod_{j = 0}^{k - 1} f(a - 2 \eta j),
\end{flalign}

\noindent respectively. In particular, when $k = 0$, all three of these expressions above are set to $1$ (independently of the parameters $a$, $q$, and $\eta$). The definitions of these Pochhammer symbols are extended to negative integers $k$ by 
\begin{flalign*}
(a)_k = \displaystyle\prod_{j = 1}^{-k} \displaystyle\frac{1}{a - j}; \qquad (a; q)_k = \displaystyle\prod_{j = 1}^{-k} \displaystyle\frac{1}{1 - q^{-j} a}; \qquad [a]_k = \displaystyle\prod_{j = 1}^{-k} \displaystyle\frac{1}{f (a + 2 \eta j)}.
\end{flalign*}

\noindent One can quickly verify that these Pochhammer symbols satisfy the identities
\begin{flalign}
\label{elliptichypergeometricidentities1}
\begin{aligned} 
(a; q)_k (a q^k; q)_m = (a; q)_{k + m}; \qquad [a]_{m - k} = \displaystyle\frac{[a]_m}{[a - 2 \eta k]_k}; \qquad  \displaystyle\lim_{q \rightarrow 1} (1 - q)^{-k} (q^a; q)_k = (a)_k,
\end{aligned} 
\end{flalign}

\noindent for any $a \in \mathbb{C}$ and $k, m \in \mathbb{Z}$. 

From the Pochhammer symbols, we can define the \emph{rational hypergeometric series},
\begin{flalign}
\label{rationalhypergeometric}
_p F_r \left( \begin{array}{ccccc}  a_1, a_2, \ldots , a_p \\ b_1, b_2, \ldots , b_r \end{array} \right| z \bigg) = \displaystyle\sum_{k = 0}^{\infty} \displaystyle\frac{z^k}{ k!} \displaystyle\prod_{j = 1}^p (a_j)_k \displaystyle\prod_{j = 1}^r (b_j)_k^{-1},
\end{flalign}

\noindent the \emph{basic hypergeometric series},
\begin{flalign}
\label{basichypergeometric}
_p \varphi_r \left( \begin{array}{ccccc}  a_1, a_2, \ldots , a_p \\ b_1, b_2, \ldots , b_r \end{array} \right| q, z \bigg) = \displaystyle\sum_{k = 0}^{\infty} \displaystyle\frac{z^k}{ \big( q; q \big)_k} \displaystyle\prod_{j = 1}^p (a_j; q)_k \displaystyle\prod_{j = 1}^r (b_j; q)_k^{-1},
\end{flalign}

\noindent and the \emph{elliptic hypergeometric series}, 
\begin{flalign}
\label{elliptichypergeometric}
_p e_r \left( \begin{array}{ccccc}  a_1, a_2, \ldots , a_p \\ b_1, b_2, \ldots , b_r \end{array} \right| z  \bigg) = \displaystyle\sum_{k = 0}^{\infty} \displaystyle\frac{z^k}{ [-2 \eta]_k} \displaystyle\prod_{j = 1}^p [a_j]_k \displaystyle\prod_{j = 1}^r [b_j]_k^{-1},
\end{flalign}

\noindent for any positive integers $p, r$ and complex numbers $a_1, a_2, \ldots , a_p, b_1, b_2, \ldots , b_r, z$. Here, we must assume that the series \eqref{rationalhypergeometric}, \eqref{basichypergeometric}, and \eqref{elliptichypergeometric} converge. In the situations we come across, this will be guaranteed since these series will always \textit{terminate}. For \eqref{rationalhypergeometric} this means that some $a_j = - r$; for \eqref{basichypergeometric} that some $a_j = q^{-r}$; and for \eqref{elliptichypergeometric} that some $a_j = 2 \eta r$, where in all cases $r$ is a nonnegative integer. 

In what follows, we will require several identities satisfied by hypergeometric series. The first is the \emph{Vandermonde-Chu identity}, which states (see equation (1.2.9) of \cite{BHS}) that
\begin{flalign}
\label{221identityhypergeometric}
{_2 F_1} \Bigg( \begin{array}{cc}  -k, b \\ c \end{array}\Bigg| 1 \Bigg) = \displaystyle\frac{(c - b)_k}{(c)_k},
\end{flalign}

\noindent for any $k \in \mathbb{Z}_{\ge 0}$ and $b, c \in \mathbb{C}$. The second is the \emph{$q$-Heine transformation identity} for the ${_2 \varphi_1}$ hypergeometric series, which states (see equation (1.4.1) of \cite{BHS}) that
\begin{flalign}
\label{21identityhypergeometric}
{_2 \varphi_1} \Bigg( \begin{array}{cc}  q^{-k}, b \\ c \end{array}\Bigg| q, z \Bigg) = \displaystyle\frac{(b; q)_{\infty} (q^{-k} z; q)_k}{(c; q)_{\infty}} {_2 \varphi_1}  \Bigg( \begin{array}{cc}  b^{-1} c, z \\ q^{-k} z \end{array}\Bigg| q, b \Bigg),
\end{flalign}

\noindent for any $k \in \mathbb{Z}_{\ge 0}$ and $b, c, z, q \in \mathbb{C}$ with $|q|, |b|< 1$. 

To state additional identities of use to us, we must first explain certain specializations of parameters under which hypergeometric series are known to simplify. The two we will discuss here are \textit{very well-poised hypergeometric series} and \textit{balanced hypergeometric series}. 

The \emph{very well-poised elliptic hypergeometric series} is defined by 
\begin{flalign}
\label{elliptichypergeometricpoised}
_{r + 1} v_r (a_1; a_6, a_7, \ldots , a_{r + 1}; z) = \displaystyle\sum_{k = 0}^{\infty} \displaystyle\frac{z^k [a_1]_k}{ [-2 \eta]_k} \displaystyle\frac{f (a_1 - 4 \eta k)}{f (a_1)} \displaystyle\prod_{j = 6}^{r + 1} \displaystyle\frac{[a_j]_k}{[a_1 - a_j - 2 \eta ]_k},
\end{flalign}

\noindent which can be viewed as a special case of the elliptic hypergeometric series \eqref{elliptichypergeometric} where $a_2$, $a_3$, $a_4$, and $a_5$ are fixed in terms of $a_1$, $\tau$, and $\eta$; see Section 11.3 of \cite{BHS}. We also call the very well-posed elliptic hypergeometric series $_{r + 1} v_r (a_1; a_6, a_7, \ldots , a_{r + 1})$ \textit{balanced} if 
\begin{flalign*}
\displaystyle\frac{(r - 5) (a_1 - 2 \eta)}{2} = \sum_{j = 6}^{r + 1} a_j - 2 \eta, 
\end{flalign*}

\noindent which is often imposed in order to ensure total ellipticity of the sum \eqref{elliptichypergeometricpoised}. 

Now, the \textit{(terminating) elliptic Jackson identity} (which was originally found by Frenkel-Turaev \cite{ESEMHF}; see also equation (11.3.19) of \cite{BHS}, in which all parameters are divided by $-2 \eta$), states that 
\begin{flalign}
\label{hypergeometric109sumterminating} 
_{10} v_9 (a; b, c, d, e, 2 n \eta ; 1) = \displaystyle\frac{[a - 2 \eta]_n [a - b - c - 2 \eta]_n [a - b - d - 2 \eta]_n [a - c - d - 2 \eta]_n}{[a - b - 2 \eta]_n [a - c - 2 \eta]_n [a - d - 2 \eta]_n [a - b - c - d - 2 \eta]_n},
\end{flalign}

\noindent for any nonnegative integer $n$, if $_{10} v_9 (a; b, c, d, e, 2 n \eta ; 1)$ is balanced.

\subsection{Stochasticizing Fused Elliptic Weights} 

\label{EllipticFused}

In this section we stochasticize a solution to the dynamical Yang-Baxter equation that arises from fusing those coming from the elliptic solid-on-solid (SOS) or interaction round-a-face (IRF) model \cite{ESSOSM}; these solutions are essentially reparameterizations of the ones provided in \cite{DSHSVM}. 

We begin with a definition of the dynamical Yang-Baxter equation. 

\begin{definition} 
	
	\label{definitiondynamicalelliptic} 
	
	We say that a weight function $W_{J; \Lambda} \big( i_1, j_1; i_2, j_2 \b| \lambda; x, y)$ is a solution to the \emph{dynamical Yang-Baxter equation} for any $i_1, j_1, i_2, j_2 \in \mathbb{Z}_{\ge 0}$ and $\lambda, x, y, J, \Lambda \in \mathbb{C}$ if the following holds. For any fixed $x, y, z, \lambda, T \in \mathbb{C}$; $i_1, j_1, k_1, i_3, j_3, k_3 \in \mathbb{Z}_{\ge 0}$; and $J, \Lambda \in \mathbb{Z}_{> 0}$, we have that 
	\begin{flalign}
	\label{dynamicalellipticequation1}
	\begin{aligned}
	\displaystyle\sum_{i_2, j_2, k_2 \in \mathbb{Z}_{\ge 0}} & W_{J; \Lambda} \big( i_1, j_1; i_2, j_2 \b| \lambda; x, y \big) W_{J; T} \big( k_1, j_2; k_2, j_3 \b| \lambda + 2 \eta (2i_2 - \Lambda); x, z \big) \\
	& \times W_{\Lambda; T} \big( k_2, i_2; k_3, i_3 \b| \lambda; y, z \big)  \\
	\qquad \qquad = \displaystyle\sum_{i_2, j_2, k_2 \in \mathbb{Z}_{\ge 0}} & W_{J; T} \big( k_2, j_1; k_3, j_2 \b| \lambda; x, z \big) W_{\Lambda; T} \big( k_1, i_1; k_2, i_2 \b| \lambda + 2 \eta (2j_1 - J); y, z \big) \\
	& \times	 W_{J; \Lambda} \big( i_2, j_2; i_3, j_3 \b| \lambda + 2 \eta (2k_3 - T); x, y \big).
	\end{aligned} 
	\end{flalign}
	
\end{definition}

Observe that, without the parameters $\lambda$, $J$, $\Lambda$, and $T$, \eqref{dynamicalellipticequation1} is similar to \eqref{equationspin1}. The diagrammatic interpretation of \eqref{dynamicalellipticequation1} is again that $W_{J; \Lambda} \big( i_1, j_1; i_2, j_2 \b| \lambda; x, y \big)$ is the weight of a vertex $u$ in a directed path ensemble with arrow configuration $(i_1, j_1; i_2, j_2)$, where $x$ and $y$ denote the horizontal and vertical rapidity parameters associated with $u$, respectively; see the left side of \Cref{vdynamicalvertex1elliptic}. Here, $J$ and $\Lambda$ are \emph{spin parameters} that are associated with the horizontal and vertical edges that intersect to form $u$, respectively.

\begin{figure}

	\begin{center}

		\begin{tikzpicture}[
		>=stealth,
		auto,
		style={
			scale = 1.8
		}
		]

		\filldraw[fill=white, draw=black] (1, 0) circle [radius=.1] node[scale = .7, below = 37, right = 4]{$\lambda + 2 \eta (2i_2 + 2j_2 - \Lambda - J)$} node[scale = 1, above = 28, left = 18]{$\lambda$};
		
		\filldraw[fill=white, draw=black] (1, 1) circle [radius=0] node[scale = .8, below = 30, right = 4]{$\lambda + 2 \eta (2i_2 - \Lambda)$};
		
		\filldraw[fill=white, draw=black] (0, 0) circle [radius=0];

		\draw[->, black, thick] (1, -.9) -- (1, -.1) node[below = 40, scale = .8]{$i_1$} node[below = 57, scale = .8]{$y$} node[below = 67, scale = .8]{$\Lambda$}; 
		
		\draw[->, black, thick] (1, .1) -- (1, .9) node[scale = .8, above = 0]{$i_2$}; 
		
		\draw[->, black, thick] (.1, 0) -- (.9, 0)  node[left = 40, scale = .8]{$j_1$} node[left = 57, scale = .8]{$x$} node[left = 67, scale = .8]{$J$} node[scale = .8, below = 33, left = 2]{$\lambda + 2 \eta (2j_1 - J)$}; 
		
		\draw[->, black, thick] (1.1, 0) -- (1.9, 0) node[right = 0, scale = .8]{$j_2$};

		\filldraw[fill=white, draw=black] (6, 0) circle [radius=.1] node[below = 25, right = 15]{$v$} node[scale = .75, above = 33, left = 2]{$v - 2 \eta (i_1 + j_1)$};
		
		\filldraw[fill=white, draw=black] (6, 1) circle [radius=0] node[scale = .9, below = 28, right = 5]{$v -2 \eta j_2$};
		
		\filldraw[fill=white, draw=black] (5, 0) circle [radius=0] node[scale = .9, below = 28, right = 7]{$v -2 \eta i_1$};

		\draw[->, black, thick] (6	, -.9) -- (6, -.1) node[below = 40, scale = .8]{$i_1$}; 
		
		\draw[->, black, thick] (6, .1) -- (6, .9) node[scale = .8, above = 0]{$i_2$}; 
		
		\draw[->, black, thick] (5.1, 0) -- (5.9, 0) node[left = 40, scale = .8]{$j_1$}; 
		
		\draw[->, black, thick] (6.1, 0) -- (6.9, 0) node[right = 0, scale = .8]{$j_2$};

		\end{tikzpicture}
		
	\end{center}

	\caption{\label{vdynamicalvertex1elliptic} Depicted above is the way in which the dynamical parameters $\lambda$ and $v$ change between faces.  }
\end{figure}
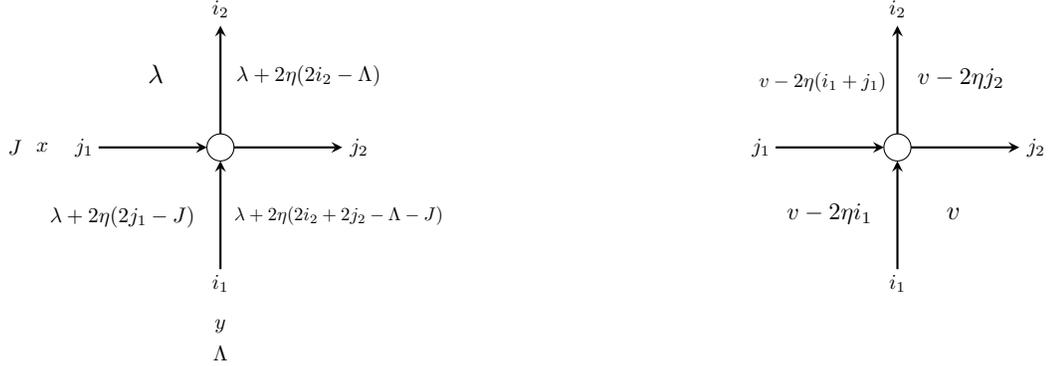

The complex number $\lambda = \lambda (u)$ is a \emph{dynamical parameter} that changes between vertices according to the identities 
\begin{flalign}
\label{lambdadynamical}
\lambda (a + 1, b) = \lambda (a, b) + 2 \eta (i_2 - \Lambda); \qquad \lambda (a, b - 1) = \lambda (a, b) + 2 \eta (2j_1 - J),
\end{flalign}

\noindent where $(i_1, j_1; i_2, j_2)$ is the arrow configuration associated with the vertex $u = (a, b)$; we refer to the left side of Figure \ref{vdynamicalvertex1elliptic} for a depiction of these identities, where for each vertex $u$, the dynamical parameter $\lambda (u)$ there is drawn in the upper-left face containing $u$.

Under this identification, the diagrammatic interpretation of \eqref{dynamicalellipticequation1} is again that of moving a line through a cross (similar to what was depicted earlier in Figure \ref{equationpaths}), where the dynamical parameter $\lambda (u)$ is now taken into account to evaluate the vertex weights. One can equivalently express the Yang-Baxter equation \eqref{dynamicalellipticequation1} in the face form \eqref{unshadedequation} on the dual graph, but here we will adhere to the vertex notation with dynamical parameters changing according to \eqref{lambdadynamical}.

Now let us provide an explicit solution to the dynamical Yang-Baxter equation in terms of very well-poised, balanced ${_{12} v_{11}}$ elliptic hypergeometric series.

\begin{definition}
	
	\label{fusedelliptic}

	Fix $x, y \in \mathbb{C}$, and let $J \in \mathbb{Z}_{> 0}$; $\lambda, \Lambda \in \mathbb{C}$; and $i_1, j_1, i_2, j_2 \in \mathbb{Z}_{\ge 0}$. We define the \emph{elliptic fused weight} $W_{J; \Lambda} \big( i_1, j_1; i_2, j_2 \b| \lambda; x, y \big) = W_{J; \Lambda} ( i_1, j_1; i_2, j_2)$ as follows. 
	
	If either $i_1 + j_1 \ne i_2 + j_2$, $j_1 \notin \{ 0, 1, \ldots , J \}$, or $j_2 \notin \{ 0, 1, \ldots , J \}$, then $W_{J; \Lambda} \big( i_1, j_1; i_2, j_2  \big) = 0$. Otherwise, set 
	\begin{flalign}
	\label{parametersa}
	\begin{aligned}
	a_1 & = \lambda + 2 \eta (2 j_1 + j_2 - J); \quad a_6 = 2 \eta j_1; \quad a_7 = 2 \eta j_2; \quad a_8 = \lambda + 2 \eta j_1; \\
	a_9 & = \lambda + 2 \eta (i_1 + 2 j_1 - J - 1 - \Lambda); \quad a_{10} = \eta (\Lambda - J) + x - y + 2 \eta (j_2 - i_1 - 1); \\
	a_{11} & = \eta (\Lambda - J) - x + y + 2 \eta (j_2 - i_1); \qquad \quad a_{12} = \lambda + 2 \eta (i_2 + j_1 + j_2 - J), 
	\end{aligned}
	\end{flalign} 
	
	\noindent and, if $i_1 + j_1 = i_2 + j_2$ and $j_1, j_2 \in \{ 0, 1, \ldots , J \}$, then define
	\begin{flalign}
	\label{wjlambda} 
	\begin{aligned}
	W_{J; \Lambda} & \big( i_1, j_1; i_2, j_2 \b| \lambda; x, y \big) \\
	& = \displaystyle\frac{[2 \eta i_2]_{i_2}}{[2 \eta i_1]_{i_1}} \displaystyle\frac{[2 \eta \Lambda]_{i_1}}{[2 \eta \Lambda]_{i_2}} \displaystyle\frac{\big[ 2 \eta (J - j_2) \big]_{j_1} }{\big[ 2 \eta j_1 \big]_{j_1} } \displaystyle\frac{\big[ 2 \eta i_1 \big]_{j_2}  \big[ 2 \eta (\Lambda - i_1 + j_2) \big]_{j_1} }{\big[ \eta (\Lambda + J) - x + y  \big]_J  } \\
	& \quad \times \big[ \eta (\Lambda + J) - x + y  - 2 \eta (i_1 + j_1) \big]_{J - j_1 - j_2} \big[ \lambda + x - y + 2 \eta (i_1 + 2j_1 - 1) - \eta (\Lambda + J) \big]_{j_1} \\
	& \quad \times \displaystyle\frac{ \big[ \lambda + 2 \eta i_2 \big]_{J - j_1 - j_2}  \big[ \lambda + 2 \eta (i_1 + 2j_1 - J) + \eta (J - \Lambda) - x + y \big]_{j_2} }{ \big[ \lambda + 2 \eta j_1 \big]_{J - j_1} \big[ \lambda + 2 \eta (2 j_1 + j_2 - J - 1) \big]_{j_1}} \\
	& \quad {\times _{12} v_{11}} (a_1; a_6, a_7, a_8, a_9, a_{10}, a_{11}, a_{12}; 1). 
	\end{aligned}
	\end{flalign}

\end{definition}

\begin{rem}
	
	\label{fusedellipticfusedrepresentation} 
	
	The elliptic fused weight $W_{J; \Lambda} \big( i_1, j_1; i_2, j_2 \b| \lambda; x, y \big)$ from Definition \ref{fusedelliptic} is directly related to the $W_J (i_1, j_1; i_2, j_2 \b| \lambda; v)$ weight given by Definition 3.6 of \cite{DSHSVM} as follows. Denoting the latter by $U_J \big( i_1, j_1; i_2, j_2 \b| \lambda, v \big)$, we have, by Theorem 3.9 of \cite{DSHSVM}, that
	\begin{flalign}
	\label{wjuj}
	W_{J; \Lambda}  \big( i_1, j_1; i_2, j_2 \b| \lambda, x, y \big) = f(2 \eta)^{i_1 - i_2} \displaystyle\frac{[2 \eta i_2]_{i_2}}{[ 2 \eta i_1 ]_{i_1}}  U_J  \big( i_1, j_1; i_2, j_2 \b| \lambda, x - y - \eta J \big).
	\end{flalign} 
	
	\noindent Thus, $W$ and $U$ are related by a gauge transformation and a change of variables replacing the $w$ and $z$ from \cite{DSHSVM} by $x + \eta (J - 1)$ and $y$ here, respectively. 
	
\end{rem}

The following proposition states that the dynamical Yang-Baxter equation holds for the elliptic fused weights. Its proof follows from Section 3 of \cite{DSHSVM} and Theorem 2.1.2 of \cite{ESSOSM} and is therefore omitted.

\begin{prop}
	
	\label{fusedellipticequation} 
	
	The elliptic fused weights $W_{J; \Lambda} \big( i_1, j_1; i_2, j_2 \b| \lambda; x, y\big)$ from \Cref{fusedelliptic} satisfy the dynamical Yang-Baxter equation in the sense of \Cref{definitiondynamicalelliptic}. 
\end{prop}

Now let us implement the stochasticization procedure on the elliptic fused weights $W_{J; \Lambda}$. This will be similar to what was done in \Cref{VertexModel} except, in the situation here, we have the dynamical parameter $\lambda$. 

In particular, the shaded vertices will again have arrow configurations of the form $(i, j; i + j, 0)$. For a fixed parameter $T \in \mathbb{C}$ (which serves as the analog of $s$ from \Cref{spin12spins}), the $\chi$ weights from \Cref{wchiconditions} will be given by $W_{J; T} (i, j; i + j, 0 \b| \lambda; x, 0)$ and $W_{\Lambda; T} (i, j; i + j, 0 \b| \lambda; y, 0)$.\footnote{The rapidity parameter associated with the stochasticization curve will be $0$ to simplify notation; setting it to be nonzero causes in a shift in the new dynamical parameter to come from the stochasticization procedure.} The first and second conditions in \Cref{wchiconditions} hold due to the fact that $W_{J; \Lambda} (i_1, j_1; i_2, j_2) = 0$ unless $i_1, j_1, i_2, j_2 \ge 0$, and the Yang-Baxter equation \eqref{dynamicalellipticequation1}, respectively.

To verify the third condition, we must assign an attachment of shaded vertices to each vertex $u$ in our domain. This is analogous to what was done in \Cref{DynamicalVertex1}. Specifically, we begin with a stochasticization curve that initially starts with some number $k$ of arrows and then collects those that our domain outputs. We then push the stochasticization curve through our domain one vertex at a time, as depicted in \Cref{arrowstochasticvertices} or \Cref{arrowstochasticelliptic} (the latter is very similar to \Cref{arrowstochastic} but also labels the dynamical parameters in the relevant faces). The two vertices of stochasticization curve adjacent to $u$ before being pushed through $u$ will be the two shaded vertices assigned to $u$.

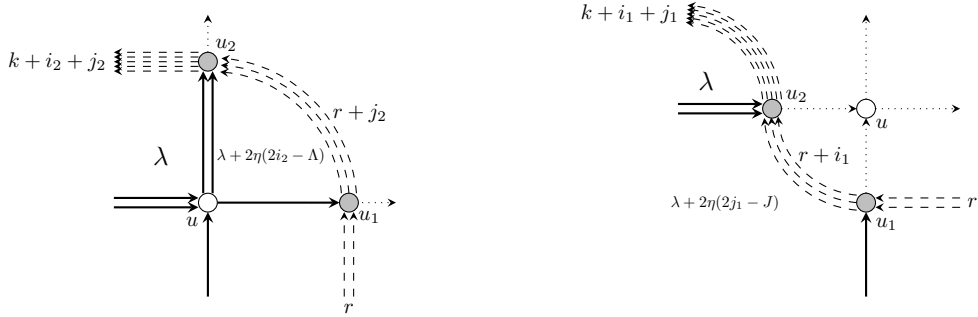
\begin{figure}

	\begin{center}

		\begin{tikzpicture}[
		>=stealth,
		auto,
		style={
			scale = 1.25
		}
		]

		\draw[->, black, thick] (2, -1) -- (2, -.1);
		\draw[->, black, thick] (2.1, 0) -- (3.4, 0);
		
		\draw[->, black, thick] (1, .05) -- (1.9, .05);
		\draw[->, black, thick] (1, -.05) -- (1.9, -.05);
		\draw[->, black, thick] (1.95, .1) -- (1.95, 1.4); 
		\draw[->, black, thick] (2.05, .1) -- (2.05, 1.4); 
		
		\draw[->, black, dotted] (2, 1.6) -- (2, 2); 
		\draw[->, black, dotted] (3.1, 0) -- (4, 0); 
		
		\draw[->, black, dashed] (3.45, -1) -- (3.45, -.1); 
		\draw[->, black, dashed] (3.55, -1) -- (3.55, -.1); 
		
		\draw[->, black, dashed] (1.9, 1.4) -- (1, 1.4); 
		\draw[->, black, dashed] (1.9, 1.5) -- (1, 1.5) node[scale = .8, left = 0]{$k + i_2 + j_2$};
		\draw[->, black, dashed] (1.9, 1.6) -- (1, 1.6);
		\draw[->, black, dashed] (1.9, 1.45) -- (1, 1.45); 
		\draw[->, black, dashed] (1.9, 1.55) -- (1, 1.55);
		
		\draw[->, black, dashed] (3.42, .1) arc(5:85:1.42);
		\draw[->, black, dashed] (3.5, .1) arc(5:85:1.5);
		\draw[->, black, dashed] (3.58, .1) arc(5:85:1.58);

		\filldraw[fill=white, draw=black] (2, 0) circle [radius = .1] node[scale = .8, below = 9, left = 1]{$u$};
		\filldraw[fill=gray!50!white, draw=black] (2, 1.5) circle [radius = .1] node[scale = .8, above = 8, right = 0]{$u_2$};
		\filldraw[fill=gray!50!white, draw=black] (3.5, 0) circle [radius = .1] node[scale = .8, below = 7, right = 1]{$u_1$};

		\filldraw[fill=white, draw=black] (1.5, .5) circle [radius = 0] node[scale = 1]{$\lambda$};
		\filldraw[fill=white, draw=black] (2.67, .5) circle [radius = 0] node[scale = .6]{$\lambda + 2 \eta (2 i_2 - \Lambda)$};

		\filldraw[fill=white, draw=black] (3.25, .95) circle [radius = 0] node[scale = .8, right = 0]{$r + j_2$};
		\filldraw[fill=white, draw=black] (3.5, -1) circle [radius = 0] node[scale = .8, below = 0]{$r$};

		\filldraw[fill=white, draw=black] (7.3, 1.3) circle [radius = 0] node[scale = 1]{$\lambda$};
		
		\filldraw[fill=white, draw=black] (7.5, 0) circle [radius = 0] node[scale = .6]{$\lambda + 2 \eta (2j_1 - J)$};

		\draw[->, black, dotted] (9.1, 1) -- (10, 1);
		\draw[->, black, dotted] (9, .1) -- (9, .9);

		\draw[->, black, dotted] (9, 1.1) -- (9, 2); 
		\draw[->, black, dotted] (8.1, 1) -- (8.9, 1); 
		
		\draw[->, black, thick] (7, .95) -- (7.9, .95);
		\draw[->, black, thick] (7, 1.05) -- (7.9, 1.05);
		
		\draw[->, black, thick] (9, -1) -- (9, -.1); 
		
		\draw[->, black, dashed] (10, -.05) -- (9.1, -.05);
		\draw[->, black, dashed] (10, .05) -- (9.1, .05);

		\draw[->, black, dashed] (7.9, 1.1) arc(5:85:.9);
		\draw[->, black, dashed] (8, 1.1) arc(5:85:1) node[scale = .8, left = 0]{$k + i_1 + j_1$};
		\draw[->, black, dashed] (8.1, 1.1) arc(5:85:1.1);
		\draw[->, black, dashed] (7.95, 1.1) arc(5:85:.95);
		\draw[->, black, dashed] (8.05, 1.1) arc(5:85:1.05);

		\draw[->, black, dashed] (8.9, .08) arc(265:185:.92);
		\draw[->, black, dashed] (8.9, 0) arc(265:185:1);
		\draw[->, black, dashed] (8.9, -.08) arc(265:185:1.08);

		\filldraw[fill=white, draw=black] (9,1) circle [radius = .1] node[scale = .8, below = 7, right = 1]{$u$};
		\filldraw[fill=gray!50!white, draw=black] (8, 1) circle [radius = .1] node[scale = .8, above = 6, right = 3]{$u_2$};
		\filldraw[fill=gray!50!white, draw=black] (9, 0) circle [radius = .1] node[scale = .8, below = 10, right = 1]{$u_1$};

		\draw (10, 0) circle [radius = 0] node[scale = .8, right = 0]{$r$}; 
		\draw (8.2, .5) circle [radius = 0] node[scale = .8, right = 0]{$r + i_1$};

		\end{tikzpicture}
		
	\end{center}
	
	\caption{\label{arrowstochasticelliptic} The process of the stochasticization curve being pushed through a vertex in the elliptic setting is depicted above.  }
\end{figure}

Since $W (0, 0; 0, 0) = 1$, the degeneration of \eqref{ws} to our situation is given by the following definition (analogous to \Cref{stochastic1}), where in what follows, $r$ denotes the number of arrows in the stochasticization curve before encountering the vertex to be stochasticized (as in \Cref{arrowstochasticelliptic}). 

\begin{definition}
	
	\label{selliptic} 
	
		Fix $x, y, \lambda, T \in \mathbb{C}$; $i_1, j_1, i_2, j_2, r \in \mathbb{Z}_{\ge 0}$; and $J, \Lambda \in \mathbb{Z}_{> 0}$. Define the \emph{stochasticized elliptic fused weights}
	\begin{flalign}
	\label{sjlambdaelliptic}	
	S_{J; \Lambda} & \big( i_1, j_1; i_2, j_2 \b| \lambda; x, y; T; r \big) = W_{J; \Lambda} \big( i_1, j_1; i_2, j_2 \b| \lambda; x, y; T; r \big) C_{J; \Lambda}  \big( i_1, j_1; i_2, j_2 \b| \lambda; x, y \big),
	\end{flalign}
	
	\noindent where we have denoted the \emph{stochastic correction} by
	\begin{flalign}
	\label{cjlambdadefinition}
	\begin{aligned}
	C_{J; \Lambda} & \big( i_1, j_1; i_2, j_2 \b| \lambda; x, y; T; r \big) \\
	& = \displaystyle\frac{W_{\Lambda; T} \big( r + j_2, i_2; r + i_2 + j_2, 0 \b| \lambda; y, 0 \big) }{W_{\Lambda; T} \big( r, i_1; r + i_1, 0 \b| \lambda + 2 \eta (2j_1 - J); y, 0 \big)} \displaystyle\frac{W_{J; T} \big( r, j_2; r + j_2, 0 \b| \lambda + 2 \eta (2i_2 - \Lambda); x, 0 \big)}{W_{J; T} \big( r + i_1, j_1; r + i_1 + j_1, 0 \b| \lambda; x, 0 \big)  }.
	\end{aligned}
	\end{flalign}

\end{definition} 

The following result explicitly evaluates the stochatsicized elliptic fused weights given by \Cref{selliptic}; we will establish it in \Cref{EllipticStochasticC}.

\begin{thm}
	
	\label{ellipticfusedstochastic} 
	
	Adopting the notation of Definition \ref{selliptic} and defining $v = \eta (J + T - 2r)$, we have that 
	\begin{flalign}
	\label{sjlambda1} 
	\begin{aligned}
	& S_{J; \Lambda} \big( i_1, j_1; i_2, j_2 \b| \lambda, v; x, y \big) \\
	& \quad  = \big[ 2 \eta (J - j_1) \big]_{j_2} \big[ 2 \eta (\Lambda - i_1 + j_2) \big]_{j_1}  \displaystyle\frac{\big[ 2 \eta i_1 \big]_{j_2}}{ \big[ 2 \eta j_2 \big]_{j_2}}  \displaystyle\frac{ \big[ \eta (\Lambda + J) - x + y  - 2 \eta (i_1 + j_1) \big]_{J - j_1 - j_2}  }{\big[ \eta (\Lambda + J) - x + y  \big]_J  }  \\
	&  \qquad \times \displaystyle\frac{ \big[ \lambda + 2 \eta (i_2  - \Lambda - 1) \big]_{J - j_1 - j_2} \big[ \lambda + 2 \eta (2j_1 - J - 1) \big]_{j_1}}{\big[ \lambda + 2 \eta (2i_2 + j_2 - \Lambda) \big]_{j_2} \big[ \lambda + 2 \eta (2 i_2 - \Lambda - 1) \big]_{J - j_2} \big[ \lambda + 2 \eta (2 j_1 + j_2 - J - 1) \big]_{j_1}} \\
	& \qquad \times \big[ \lambda + x - y + 2 \eta (i_1 + 2j_1 - 1) - \eta (\Lambda + J) \big]_{j_1}  \big[ \lambda - x + y + 2 \eta (i_1 + 2j_1 - J) + \eta (J - \Lambda) \big]_{j_2}  \\
	&  \qquad \times  \displaystyle\frac{\big[ \lambda + x - v + 2 \eta (2i_2 + 2j_2 - \Lambda - 1)  \big]_{j_2}}{\big[ \lambda + x - v + 2 \eta (i_1 + 2j_1 - 1)  \big]_{j_1} } \displaystyle\frac{\big[  v - x - 2 \eta j_2 \big]_{i_2}}{\big[ v - x - 2 \eta J \big]_{i_1}} \\
	&  \qquad \times \displaystyle\frac{\big[ \lambda + y - v + 2 \eta (2 i_2 + j_2 - 1) + \eta (J - \Lambda) \big]_{i_2}}{\big[ \lambda + y - v + 	2 \eta (2 i_1 + 2j_1 - 1) - \eta (\Lambda + J) \big]_{i_1}}\displaystyle\frac{ \big[ v - y - \eta (\Lambda + J)  \big]_{j_2}}{\big[ v - y + \eta (\Lambda - J - 2 i_1) \big]_{j_1}} \\
	& \qquad  \times   {_{12} v_{11}} (a_1; a_6, a_7, a_8, a_9, a_{10}, a_{11}, a_{12}; 1) \textbf{\emph{1}}_{i_1 + j_1 = i_2 + j_2},
	\end{aligned} 
	\end{flalign}
	
	\noindent where we have denoted $S_{J; \Lambda} \big( i_1, j_1; i_2, j_2 \b| \lambda, v; x, y \big) = S_{J; \Lambda} \big( i_1, j_1; i_2, j_2 \b| \lambda; x, y; T, r \big)$ (since \eqref{sjlambda1} indicates that it only depends on $T$ and $r$ through $v$).

\end{thm}

As in \Cref{DynamicalVertex1}, the parameter $v$ here is dynamical; it depends on the vertex $u$ that is being stochasticized, and so we sometimes write $v = v(u)$. The parameter $r$ is analogous to the parameter $k$ from \Cref{DynamicalVertex1} and is governed by the same identities (given by the first two in \eqref{ukv}). Thus, from the definition $v = \eta (J + T - 2r)$, one can quickly deduce that $v (u)$ is governed by 
\begin{flalign}	
\label{vdynamicalelliptic}
v (a - 1, b) = v (a, b) - 2 \eta i_1 (a, b); \qquad v (a, b + 1) = v(a, b) - 2 \eta j_2 (a, b),
\end{flalign}

\noindent for any $(a, b) \in \mathbb{Z}_{> 0}^2$; this is depicted on the right side of \Cref{vdynamicalvertex1elliptic}, where there the dynamical parameter $v(u)$ is drawn in the lower-right face containing $u$. 

The stochasticity and dynamical Yang-Baxter in this setting is given by the following theorem, which takes into account both dynamical parameters $v$ and $\lambda$.

\begin{thm}
	
	\label{dynamicalequationellipticstochastic}
	
	Adopt the notation of Theorem \ref{ellipticfusedstochastic}. For any fixed $\Lambda, x, y, \lambda, v  \in \mathbb{C}$; $J \in \mathbb{Z}_{> 0}$; and $i_1, j_1 \in \mathbb{Z}_{\ge 0}$, we have that $\sum_{i_2, j_2 \in \mathbb{Z}_{\ge 0}} S_{J; \Lambda} \big( i_1, j_1; i_2, j_2 \b| \lambda, v; x, y \big) = 1$. 
	
	Furthermore, for any fixed $i_1, j_1, i_3, j_3 \in \mathbb{Z}_{\ge 0}$; $J, \Lambda \in \mathbb{Z}_{> 0}$; $\lambda, T, v, x, y, z \in \mathbb{C}$, the $S$ weights satisfy
		\begin{flalign}
	\label{dynamicalellipticequation2}
	\begin{aligned}
	\displaystyle\sum_{i_2, j_2, k_2 \in \mathbb{Z}_{\ge 0}} & S_{J; \Lambda} \big( i_1, j_1; i_2, j_2 \b| \lambda; v - 2 \eta k_1; x, y \big) S_{J; T} \big( k_1, j_2; k_2, j_3 \b| \lambda + 2 \eta (2i_2 - \Lambda); v; x, z \big) \\
	&  \times  S_{\Lambda; T} \big( k_2, i_2; k_3, i_3 \b| \lambda; v - 2 \eta j_3; y, z \big) \\
	\quad = \displaystyle\sum_{i_2, j_2, k_2 \in \mathbb{Z}_{\ge 0}} & S_{\Lambda; T} \big( k_1, i_1; k_2, i_2 \b| \lambda + 2 \eta (2j_1 - J); v; y, z \big) S_{J; T} \big( k_2, j_1; k_3, j_2 \b| \lambda; v - 2 \eta i_2; x, z \big) 	\\
	& \times S_{J; \Lambda} \big( i_2, j_2; i_3, j_3 \b| \lambda + 2 \eta (2k_3 - T); v; x, y \big).
	\end{aligned} 
	\end{flalign}

\end{thm}

\begin{proof} 
	
The first and second statement of this theorem are the applications of \Cref{stochasticws} and \Cref{relationwstochastic} to our setting, respectively; the second statement \eqref{dynamicalellipticequation2} can alternatively be established directly, similar to the proof of \Cref{sdynamicalequation1}, but we will not implement this here.
\end{proof}

\begin{rem} 

\label{ellipticequationstochastic} 

We are unaware of any stochastic, fully elliptic solution to the dynamical Yang-Baxter equation that does not involve the second dynamical parameter $v$. This is different from what was described in \Cref{VertexModel}, where setting $v = 0$ in the weights \eqref{s01dynamicalweights} still gives rise to stochastic weights satisfying the non-dynamical Yang-Baxter equation. 

\end{rem}

\subsection{Proof of \Cref{ellipticfusedstochastic}} 

\label{EllipticStochasticC}

In this section we establish \Cref{ellipticfusedstochastic}; to that end, we must evaluate the stochastic corrections $C_{J; \Lambda}$ from \eqref{cjlambdadefinition}. This is done by the following proposition. 

\begin{prop} 
	
	\label{cjlambdaelliptic}
	
	Denoting $v = \eta (J + T - 2r)$, we have that 
	\begin{flalign*}
	C_{J; \Lambda} & \big( i_1, j_1; i_2, j_2 \b| \lambda; x, y; T; r \big) \\
	& = \displaystyle\frac{\big[ 2 \eta J \big]_{j_2} }{\big[ 2 \eta J \big]_{j_1} } \displaystyle\frac{\big[ 2 \eta \Lambda \big]_{i_2}}{\big[ 2 \eta \Lambda \big]_{i_1} } \displaystyle\frac{\big[ 2 \eta i_1 \big]_{i_1}}{\big[ 2 \eta i_2 \big]_{i_2}} \displaystyle\frac{\big[ 2 \eta j_1 \big]_{j_1}}{ \big[ 2 \eta j_2 \big]_{j_2}}  \displaystyle\frac{\big[ \lambda + 2 \eta j_1 \big]_{J - j_1} \big[ \lambda + 2 \eta (2j_1 - J - 1) \big]_{j_1}}{\big[ \lambda + 2 \eta (2i_2 + j_2 - \Lambda) \big]_{j_2} \big[ \lambda + 2 \eta (2 i_2 - \Lambda - 1) \big]_{J - j_2}} \\
	& \qquad \qquad \times \displaystyle\frac{\big[ \lambda + 2 \eta (i_1 + 2j_1 - J) \big]_{j_2} \big[ \lambda + 2 \eta (i_2 + j_1 - \Lambda - 1) \big]_{J - j_2}}{\big[ \lambda + 2 \eta i_2 \big]_{J - j_1} \big[ \lambda + 2 \eta (i_2 + j_1  - \Lambda - 1) \big]_{j_1}} \\
	& \qquad \qquad \times  \displaystyle\frac{\big[ \lambda + x - v + 2 \eta (2i_2 + 2j_2 - \Lambda - 1)  \big]_{j_2}}{\big[ \lambda + x - v + 2 \eta (i_1 + 2j_1 - 1)  \big]_{j_1} } \displaystyle\frac{\big[  v - x - 2 \eta j_2 \big]_{i_2}}{\big[ v - x - 2 \eta J \big]_{i_1}} \\
	& \qquad \qquad \times \displaystyle\frac{\big[ \lambda + y - v + 2 \eta (2 i_2 + j_2 - 1) + \eta (J - \Lambda) \big]_{i_2}}{\big[ \lambda + y - v + 	2 \eta (2 i_1 + 2j_1 - 1) - \eta (\Lambda + J) \big]_{i_1}}\displaystyle\frac{ \big[ v - y - \eta (\Lambda + J)  \big]_{j_2}}{\big[ v - y + \eta (\Lambda - J - 2 i_1) \big]_{j_1}}.
	\end{flalign*}

\end{prop}

Theorem \ref{ellipticfusedstochastic} follows from inserting Proposition \ref{cjlambdaelliptic} into \eqref{sjlambdaelliptic}. To establish \Cref{cjlambdaelliptic}, we require explicit expressions for fused weights of the form $W_{J; \Lambda} (i, j; i + j, 0)$. In principle, these are given by \Cref{ellipticfusedstochastic} by an elliptic hypergeometric series; however, the following result, which can be found as Proposition 4.3 of \cite{DSHSVM}, indicates that these hypergeometric weights simplify considerably when $j_2 = 0$.

\begin{prop}[{\cite[Proposition 4.3]{DSHSVM}}]
	
	\label{wjj0} 
	
	Let $i \in \mathbb{Z}_{\ge 0}$, $J \in \mathbb{Z}_{> 0}$, and $j \in \mathbb{Z}$. If $j < 0$ or $j > J$, then $W_J (i, j; i + j, 0) = 0$. Otherwise, 
	\begin{flalign}
	\label{wjj20identity}
	\begin{aligned}
	W_{J; \Lambda} \big( & i, j; i + j, 0 \b| \lambda; x, y \big) \\
	& = \displaystyle\frac{\big[ 2 \eta J \big]_j }{\big[ 2 \eta j \big]_j} \displaystyle\frac{\big[ 2 \eta (i + j) \big]_j \big[ \lambda + 2 \eta (i + j) \big]_{J - j} }{\big[ \eta (\Lambda + J) - x + y\big]_J} \\
	& \qquad \times \displaystyle\frac{\big[ x - y + \lambda + 2 \eta (i + 2j - 1) - \eta (\Lambda + J) \big]_j \big[ \eta (\Lambda + J) - 2 \eta (i + j) - x + y \big]_{J - j}}{\big[ \lambda + 2 \eta j \big]_{J - j} \big[ \lambda + 2 \eta (2j - J - 1) \big]_j}.
	\end{aligned}
	\end{flalign}
	
\end{prop}

Now we can prove \Cref{cjlambdaelliptic}. 

\begin{proof}[Proof of \Cref{cjlambdaelliptic}]
	
	By \eqref{wjj20identity}, we have that 
	\begin{flalign}
	\label{wjt}
	\begin{aligned}
	& \displaystyle\frac{W_{J; T} \big( r, j_2; r + j_2, 0 \b| \lambda + 2 \eta (2i_2 - \Lambda); x, 0 \big)}{W_{J; T} \big( r + i_1, j_1; r + i_1 + j_1, 0 \b| \lambda; x, 0 \big)  } \\
	& \qquad = \displaystyle\frac{\big[ 2 \eta (r + j_2) \big]_{j_2}}{\big[ 2 \eta (r + i_1 + j_1) 	\big]_{j_1}} \displaystyle\frac{\big[ 2 \eta J \big]_{j_2} }{\big[ 2 \eta J \big]_{j_1} } \displaystyle\frac{\big[ 2 \eta j_1 \big]_{j_1}}{ \big[ 2 \eta j_2 \big]_{j_2}} \displaystyle\frac{\big[ \lambda + 2 \eta (r + 2i_2 + j_2 - \Lambda) \big]_{J - j_2} }{\big[ \lambda + 2 \eta (r + i_1 + j_1) \big]_{J - j_1} } \\
	& \qquad \qquad \times \displaystyle\frac{\big[ \lambda + 2 \eta j_1 \big]_{J - j_1} \big[ \lambda + 2 \eta (2j_1 - J - 1) \big]_{j_1}}{\big[ \lambda + 2 \eta (2i_2 + j_2 - \Lambda) \big]_{J - j_2} \big[ \lambda + 2 \eta (2i_2 + 2j_2  - \Lambda - J - 1) \big]_{j_2}} \\
	& \qquad \qquad \times  \displaystyle\frac{\big[ \lambda + x + 2 \eta (r + 2i_2 + 2j_2 - \Lambda - 1) - \eta (T + J) \big]_{j_2}}{\big[ \lambda + x + 2 \eta (r + i_1 + 2j_1 - 1) - \eta (T + J) \big]_{j_1} } \displaystyle\frac{\big[  \eta (T + J) - 2 \eta (r + j_2) - x \big]_{i_2}}{\big[ \eta (T - J - 2r) - x \big]_{i_1}}, 
	\end{aligned} 
	\end{flalign}
	
	\noindent where we have used the fact that  
	\begin{flalign*}
	\displaystyle\frac{\big[  \eta (T + J) - 2 \eta (r + j_2) - x \big]_{J - j_2}} {\big[  \eta (T + J) - 2 \eta (r + i_1 + j_1) - x \big]_{J - j_1}} = \displaystyle\frac{\big[  \eta (T + J) - 2 \eta (r + j_2) - x \big]_{i_2}}{\big[ \eta (T - J - 2r) - x \big]_{i_1}},
	\end{flalign*}
	
	\noindent which holds from the second identity in \eqref{elliptichypergeometricidentities1} and the fact that $i_1 + j_1 = i_2 + j_2$. Again applying \eqref{wjj20identity} yields  
	\begin{flalign}
	\label{wlambdat}
	\begin{aligned}
	& \displaystyle\frac{W_{\Lambda; T} \big( r + j_2, i_2; r + i_2 + j_2, 0 \b| \lambda; y, 0 \big) }{W_{\Lambda; T} \big( r, i_1; r + i_1, 0 \b| \lambda + 2 \eta (2j_1 - J); y, 0 \big)}  \\
	& \qquad = \displaystyle\frac{\big[ 2 \eta (r + i_2 + j_2) \big]_{i_2}}{\big[ 2 \eta (r + i_1) \big]_{i_1}}  \displaystyle\frac{\big[ 2 \eta \Lambda \big]_{i_2}}{\big[ 2 \eta \Lambda \big]_{i_1} } \displaystyle\frac{\big[ 2 \eta i_1 \big]_{i_1}}{\big[ 2 \eta i_2 \big]_{i_2}}  \displaystyle\frac{\big[ \lambda + 2 \eta (r + i_1 + j_1)  \big]_{J - j_1} }{\big[ \lambda + 2 \eta (r + 2i_2 + j_2 - \Lambda) \big]_{J - j_2} } \\
	& \qquad \qquad \times \displaystyle\frac{\big[ \lambda + 2 \eta (i_1 + 2j_1 - J) \big]_{j_2}}{\big[ \lambda + 2 \eta (2 i_2 + 2j_2 - J - \Lambda) \big]_{2 j_2 - J} \big[ \lambda + 2 \eta i_2 \big]_{J - j_1 }} \displaystyle\frac{\big[ \lambda + 2 \eta (2i_2 + 2j_2  - \Lambda - J - 1) \big]_{j_2}}{\big[ \lambda + 2 \eta (2 i_2 - \Lambda - 1) \big]_{J - j_2}} \\
	& \qquad \qquad \times \displaystyle\frac{\big[ \lambda + 2 \eta (i_2 + j_1 - \Lambda - 1) \big]_{J - j_2}}{\big[ \lambda + 2 \eta (i_2 + j_1  - \Lambda - 1) \big]_{j_1}} \\
	& \qquad \qquad \times \displaystyle\frac{\big[ y + \lambda + 2 \eta (r + 2 i_2 + j_2 - 1) - \eta (\Lambda + T) \big]_{i_2}}{\big[ y + \lambda + 2 \eta (r + 2 i_1 + 2j_1 - J - 1) - \eta (\Lambda + T) \big]_{i_1}}\displaystyle\frac{ \big[ \eta (T - \Lambda - 2r) - y \big]_{j_2}}{\big[ \eta (\Lambda + T) - 2 \eta (r + i_1) - y \big]_{j_1}}, 
	\end{aligned}
	\end{flalign}
	
	\noindent where we have used the facts (again due to the second statement of \eqref{elliptichypergeometricidentities1} and the identity $i_1 + j_1 = i_2 + j_2$) that 
	\begin{flalign*}
	& \displaystyle\frac{\big[ \lambda + 2 \eta (r + i_2 + j_2) \big]_{\Lambda - i_2}}{\big[ \lambda + 2 \eta (r + i_1 + 2 j_1 - J ) \big]_{\Lambda - i_1}} = \displaystyle\frac{\big[ \lambda + 2 \eta (r + i_1 + j_1)  \big]_{J - j_1} }{\big[ \lambda + 2 \eta (r + 2i_2 + j_2 - \Lambda) \big]_{J - j_2} }; \\
	& \displaystyle\frac{ \big[ \eta (\Lambda + T) - 2 \eta (r + i_2 + j_2) - y \big]_{\Lambda - i_2}}{\big[ \eta (\Lambda + T) - 2 \eta (r + i_1) - y \big]_{\Lambda - i_1}} = \displaystyle\frac{ \big[ \eta (T - \Lambda - 2r) - y \big]_{j_2}}{\big[ \eta (\Lambda + T) - 2 \eta (r + i_1) - y \big]_{j_1}}; \\
	& \displaystyle\frac{\big[ \lambda + 2 \eta (i_1 + 2j_1 - J) \big]_{\Lambda - i_1}}{\big[ \lambda + 2 \eta i_2 \big]_{\Lambda - i_2}}  =  \displaystyle\frac{\big[ \lambda + 2 \eta (i_1 + 2j_1 - J) \big]_{j_2}}{\big[ \lambda + 2 \eta (2 i_2 + 2j_2 - J - \Lambda) \big]_{2 j_2 - J} \big[ \lambda + 2 \eta i_2 \big]_{J - j_1 }}; \\
	& \displaystyle\frac{ \big[ \lambda + 2 \eta (2 i_1 + 2 j_1 - J - \Lambda - 1) \big]_{i_1}}{\big[ \lambda + 2 \eta (2 i_2 - \Lambda - 1) \big]_{i_2}} = \displaystyle\frac{ \big[ \lambda + 2 \eta (2 i_1 + 2 j_1 - J - \Lambda - 1) \big]_{2j_2 - J}}{ \big[ \lambda + 2 \eta (i_2 + j_1 + j_2  - \Lambda - J - 1) \big]_{j_1 + j_2 - J}} \\
	& \qquad \qquad \qquad \qquad \qquad \qquad \qquad \quad  = \displaystyle\frac{\big[ \lambda + 2 \eta (2i_2 + 2j_2  - \Lambda - J - 1) \big]_{j_2}}{\big[ \lambda + 2 \eta (2 i_2 - \Lambda - 1) \big]_{J - j_2}}  \\
	& \qquad \qquad \qquad \qquad \qquad \qquad \qquad \qquad \quad \times \displaystyle\frac{\big[ \lambda + 2 \eta (i_2 + j_1 - \Lambda - 1) \big]_{J - j_2}}{\big[ \lambda + 2 \eta (i_2 + j_1  - \Lambda - 1) \big]_{j_1}}. 
	\end{flalign*}
	
	\noindent Again using the second identity in \eqref{elliptichypergeometricidentities1} and the fact that $i_1 + j_1 = i_2 + j_2$, we find that 
	\begin{flalign*}
	& \big[ \lambda + 2 \eta (2i_2 + j_2 - \Lambda) \big]_{J - j_2} \big[ \lambda + 2 \eta (2 i_2 + 2j_2 - J - \Lambda) \big]_{2 j_2 - J} = \big[ \lambda + 2 \eta (2i_2 + j_2 - \Lambda) \big]_{j_2}; \\
	& \big[ 2 \eta (r + i_1 + j_1) \big]_{j_1} \big[ 2 \eta (r + i_1) \big]_{i_1}  = \big[ 2 \eta (r + i_2 + j_2) \big]_{i_2}  \big[ 2 \eta (r + j_2) \big]_{j_2}, 
	\end{flalign*} 
	
	\noindent using which, and inserting \eqref{wjt} and \eqref{wlambdat} into \eqref{cjlambdadefinition}, we deduce the proposition (after recalling that $v = \eta (J + T - 2r)$). 
\end{proof}

\subsection{Degenerations}

\label{Elliptic01}

In this section we consider several degenerations of the stochasticized fully fused elliptic weights. 

First observe that there exists a \emph{trigonometric limit} under which the theta function given by \eqref{theta1} converges to a (scale of) a trigonometric function, in that $\lim_{\tau \rightarrow \textbf{i} \infty} e^{- \pi \textbf{i} \tau / 4} \theta (z) = 2 \sin (\pi z)$. Using this fact, one can show that as $\tau$ tends to $\textbf{i} \infty$ and $v$ tends to $\textbf{i} \infty$, the stochastic weights defined by \eqref{sjlambdaelliptic} (or explicitly by \Cref{ellipticfusedstochastic}) degenerate to those provided in Definition 1.1 of \cite{DSHSVM} that give rise to the dynamical stochastic higher spin vertex models. Already under the trigonometric limit (as $\tau$ tends to $\textbf{i} \infty$) and when $v$ remains finite, the weights given by \Cref{ellipticfusedstochastic} appear to be new, except in the case when $J = 1 = \Lambda$, when they coincide with those of \Cref{dynamicalweights1}, which were considered earlier in \cite{FSSF}. 

Here, we will mainly focus on the case when the trigonometric limit is not taken, that is, when $\tau$ remains finite. This gives rise to an elliptic, stochastic, integrable face model that does not seem to have appeared before in the literature. The fully general weights of this model given by \Cref{ellipticfusedstochastic} might not appear so pleasant, but they become significantly simpler under certain specializations of the parameters. 

We begin with the case $J = 1 = \Lambda$. 

\begin{prop}

\label{jlambda1s}

For any $\lambda, v, x, y \in \mathbb{C}$, the stochasticized elliptic fused weights at $J = 1 = \Lambda$, given by $S_{1; 1} (i_1, j_1; i_2, j_2 \b| \lambda, v; x, y)$, coincide with the ones provided in \Cref{1jsdefinition}. Here, we recall that the dynamical parameters $\lambda$ and $v$ are governed by the identities \eqref{lambdadynamical} and \eqref{vdynamicalelliptic} that are depicted in \Cref{vdynamicalvertex1elliptic} (with $J = 1 = \Lambda$). 
\end{prop}

\begin{proof}
	
	Abbreviating $W_{1; 1} \big( i_1, j_1; i_2, j_2 \b| \lambda; x, y \big) = W_{1; 1} (i_1, j_1; i_2, j_2)$, one can deduce (directly from the definition \eqref{wjlambda} or from specializing $\Lambda = 1$ in equation (3.1) of \cite{DSHSVM} and applying the transformation from \Cref{fusedellipticfusedrepresentation}) that 
	\begin{flalign}
	\label{w11}
	\begin{aligned}
	W_{1; 1} (1, 0; 1, 0) & =  \displaystyle\frac{f(y - x) f(\lambda + 2 \eta)}{f(y - x + 2 \eta) f(\lambda)}; \qquad W_{1; 1} ( 0, 1; 1, 0) =  \displaystyle\frac{f(\lambda - y + x) f(2 \eta)}{f(y - x + 2 \eta) f(\lambda)}; \\
	W_{1; 1} ( 1, 0; 0, 1) & =  \displaystyle\frac{f(\lambda + y - x) f (2\eta)}{f(y - x + 2 \eta) f(\lambda)}; \qquad W_{1; 1} ( 0, 1; 0, 1) =  \displaystyle\frac{f(y - x) f(\lambda - 2 \eta)}{f(y - x + 2 \eta) f(\lambda)}; \\
	& \qquad \qquad W_{1; 1} ( 0, 0; 0, 0) = 1 = W_{1; 1} ( 1, 1; 1, 1).
	\end{aligned}
	\end{flalign}
	
	\noindent Now the proposition follows from inserting \eqref{w11} and the explicit form of the stochastic correction given by Proposition \ref{cjlambdaelliptic} into \eqref{sjlambdaelliptic}. 
\end{proof}

Since the $S_{1; 1} (i_1, j_1; i_2, j_2)$ are zero unless $i_1, j_1, i_2, j_2 \in \{ 0, 1 \}$, we view them as giving rise to an elliptic, stochastic, dynamical deformation of the six-vertex model. The higher spin analogs of these weights (allowing for arbitrary $i_1, i_2 \in \mathbb{Z}_{\ge 0}$ but still restricting $j_1, j_2 \in \{ 0, 1 \}$) is given by the following proposition.

\begin{prop}
	
	\label{1js}
	
	For any $\Lambda, \lambda, x, y, \in \mathbb{C}$ and $k \in \mathbb{Z}_{\ge 0}$, we have that
	\begin{flalign*}
	S_{1; \Lambda} (k, 0; k, 0) & = \displaystyle\frac{f \big( y - x + \eta (\Lambda - 2k + 1) \big) f \big( \lambda + 2 \eta (k - \Lambda - 1) \big) }{f \big( y - x + \eta (\Lambda + 1) \big) f \big( \lambda + 2 \eta (2k - \Lambda - 1) \big) } \\
	& \qquad \times \displaystyle\frac{f \big( x - v \big) f \big( \lambda + y + \eta (4k - \Lambda - 1) - v \big) }{f \big( x + 2k \eta - v \big) f \big( \lambda + y + \eta (2 k - \Lambda - 1) - v \big) }; 
	\end{flalign*}
	
	\begin{flalign*}
	S_{1; \Lambda} ( k, 1; k + 1, 0) & = \displaystyle\frac{f \big( \lambda - y + x + \eta (2k - \Lambda + 1) \big) f \big( 2 \eta (\Lambda - k) \big) }{f \big( y - x + \eta (\Lambda + 1) \big) f \big( \lambda + 2 \eta (2k - \Lambda + 1) \big)} \\
	& \qquad \times \displaystyle\frac{f \big( v - x \big) f \big( \lambda + y - v + \eta (4k - \Lambda + 3) \big)}{f \big(\lambda + x - v + 2 \eta (k + 1)  \big) f \big( v - y + \eta (\Lambda - 2k - 1)  \big)};
	\end{flalign*}
	
	\begin{flalign*}
	S_{1; \Lambda} (k, 0; k - 1, 1) & =  \displaystyle\frac{f \big( \lambda + y - x + \eta (2k - \Lambda - 1) \big) f ( 2 \eta k) }{f \big( y - x + \eta (\Lambda + 1) \big) f \big( \lambda + 2 \eta (2k - \Lambda - 1) \big) } \\
	& \qquad \times \displaystyle\frac{f \big( v - y - \eta (\Lambda + 1) \big) f \big( \lambda - v + x + 2 \eta (2k - \Lambda - 1) \big) }{f \big( v - x - 2k \eta	 \big) f \big( \lambda + y + \eta (2 k - \Lambda - 1) - v \big) };
	\end{flalign*}
	
	\begin{flalign*}
	S_{1; \Lambda} ( k, 1; k, 1) & = \displaystyle\frac{f \big(	y - x + \eta (2k - \Lambda + 1) \big) f \big( \lambda + 2 \eta (k + 1) \big) }{f \big( y - x + \eta (\Lambda + 1) \big) f \big(\lambda + 2 \eta (2k - \Lambda + 1) \big)} \\
	& \qquad \times \displaystyle\frac{ f \big( x + \lambda + 2 \eta (2k - \Lambda + 1) - v \big) f \big( v - y - \eta (\Lambda + 1)  \big)}{f \big(\lambda + x - v + 2 \eta (k + 1) \big) f \big( v - y + \eta (\Lambda - 2k - 1) \big) },
	\end{flalign*}
	
	\noindent and $S_{1; \Lambda} ( i_1, j_1; i_2, j_2 ) = 0$ for all $(i_1, j_1; i_2, j_2)$ not of the above form. Above, we have abbreviated $S_{1; \Lambda} \big( i_1, j_1; i_2, j_2 \b| \lambda, v; x, y \big) = S_{1; \Lambda} (i_1, j_1; i_2, j_2)$. Here, we recall that the dynamical parameters $\lambda$ and $v$ are governed by the identities \eqref{lambdadynamical} and \eqref{vdynamicalelliptic} that are depicted in \Cref{vdynamicalvertex1elliptic} (with $J = 1$). 
\end{prop}

\begin{proof}
	
The proof of this result is entirely analogous to that of \Cref{jlambda1s} after observing (again, either as a consequence of \Cref{fusedelliptic} or by applying the transformation given by \Cref{fusedellipticfusedrepresentation} to equation (3.1) of \cite{DSHSVM}) that  
\begin{flalign*}
W_{1; \Lambda} ( k, 0; k, 0) & = \displaystyle\frac{f\big(y - x + \eta (\Lambda - 2 k +  1) \big) f \big( \lambda + 2 \eta k \big) }{f \big( y - x + \eta (\Lambda + 1)  \big) f \big( \lambda \big) }; \\
W_{1; \Lambda} ( k, 1; k + 1, 0) &= \displaystyle\frac{f \big( \lambda - y + x +  \eta (2k - \Lambda +  1) \big) f \big( 2 \eta (k + 1) \big)}{f \big(y - x + \eta (\Lambda + 1) \big) f \big( \lambda )}; \\
W_{1; \Lambda} ( k, 0; k - 1, 1) &= \displaystyle\frac{f \big( \lambda + y - x + \eta (2 k - \Lambda - 1) \big) f \big(2 \eta (k + 1) \big)}{f \big(y - x + \eta (\Lambda + 1) \big) f \big( \lambda \big)}; \\
W_{1; \Lambda} ( k, 1; k, 1) &= \displaystyle\frac{f \big( y - x + \eta (2k - \Lambda + 1)	 \big) f \big( \lambda + 2 \eta (k - \Lambda) \big)}{f \big(y - x + \eta (\Lambda + 1)  \big) f \big( \lambda  \big)}.
\end{flalign*}
\end{proof}

Proposition \ref{1js} indicates that the elliptic hypergeometric weights $S_{J; \Lambda}$ simplify for arbitrary $\Lambda$ when $J = 1$. When $J \in \mathbb{Z}_{> 1}$, these weights typically do not simplify. However, an exception in the case $x = y + \eta (J - \Lambda)$ was found in Theorem 3.10 of \cite{DSHSVM} (as an elliptic analog of Proposition 6.7 of \cite{FSRF}). It was shown there that, under the trigonometric limit, these simplified weights give rise to a dynamical variant of the $q$-Hahn boson model. 

The following proposition shows that, under this specialization $x = y + \eta (J - \Lambda)$, our stochasticized fused elliptic weights also simplify; this gives rise to an elliptic variant of the $q$-Hahn boson model of \cite{TBP,IZRCMFSS}.

\begin{prop}
	
	\label{svlambdaelliptic}

	Fix $J \in \mathbb{Z}_{> 0}$ and $i_1, i_2 \in \mathbb{Z}_{\ge 0}$, and assume that $x = y + \eta (J - \Lambda)$; further let $j_1, j_2 \in \{ 0, 1, \ldots , J \}$. If $i_1 + j_1 \ne i_2 + j_2$, then $S_{J; \Lambda} \big( i_1, j_1; i_2, j_2 \b| v, \lambda; x, y \big) = 0$. Otherwise, 
	\begin{flalign*}
	& S_{J; \Lambda} \big( i_1, j_1; i_2, j_2 \b| \lambda, v; x, y \big) \\
	 & \quad = \displaystyle\frac{\big[ 2 \eta J \big]_{j_2}}{\big[ 2 \eta j_2 \big]_{j_2}} \displaystyle\frac{ \big[ 2 \eta i_1 \big]_{j_2} \big[ 2 \eta (\Lambda - i_1) \big]_{J - j_2} }{\big[ 2 \eta \Lambda \big]_J}  \displaystyle\frac{ \big[ \lambda + 2 \eta (i_1 + 2j_1 - J) \big]_{j_2} \big[ \lambda + 2 \eta (i_2 + j_1 - \Lambda - 1) \big]_{J - j_2} }{\big[ \lambda + 2 \eta (2i_2 + j_2 - \Lambda) \big]_{j_2} \big[ \lambda + 2 \eta (2 i_2 - \Lambda - 1) \big]_{J - j_2}} \\
	& \qquad \qquad \times  \displaystyle\frac{\big[ \lambda + x - v + 2 \eta (2i_2 + 2j_2 - \Lambda - 1)  \big]_{j_2}}{\big[ \lambda + x - v + 2 \eta (i_1 + 2j_1 - 1)  \big]_{j_1} } \displaystyle\frac{\big[  v - x - 2 \eta j_2 \big]_{i_2}}{\big[ v - x - 2 \eta J \big]_{i_1}} \\
	& \qquad \qquad \times \displaystyle\frac{\big[ \lambda + y - v + 2 \eta (2 i_2 + j_2 - 1) + \eta (J - \Lambda) \big]_{i_2}}{\big[ \lambda + y - v + 	2 \eta (2 i_1 + 2j_1 - 1) - \eta (\Lambda + J) \big]_{i_1}}\displaystyle\frac{ \big[ v - y - \eta (\Lambda + J)  \big]_{j_2}}{\big[ v - y + \eta (\Lambda - J - 2 i_1) \big]_{j_1}} \textbf{\emph{1}}_{i_1 \ge j_2} .
	\end{flalign*}
	
	\noindent Here, we recall that the dynamical parameters $\lambda$ and $v$ are governed by the identities \eqref{lambdadynamical} and \eqref{vdynamicalelliptic} that are depicted in \Cref{vdynamicalvertex1elliptic}.

\end{prop}

\begin{proof} 
	Theorem 3.10 of \cite{DSHSVM} states that, if  $x = y + \eta (J - \Lambda)$, then 
	\begin{flalign}
	\label{wjlambdaxyspecialization}
	\begin{aligned}
	W_{J; \Lambda} \big( i_1, j_1; i_2, j_2 \b| \lambda; x, y\big) & =  \displaystyle\frac{\textbf{1}_{i_1 \ge j_2} \big[ 2 \eta J \big]_J}{\big[ 2 \eta j_1 \big]_{j_1} \big[ 2 \eta (J - j_1) \big]_{J - j_1}} \displaystyle\frac{\big[ 2 \eta \Lambda \big]_{i_1}}{\big[ 2 \eta \Lambda \big]_{i_2}} \displaystyle\frac{[2 \eta i_2]_{i_2}}{ [2 \eta i_1]_{i_1}} \displaystyle\frac{ \big[ 2 \eta i_1 \big]_{j_2} \big[ 2 \eta (\Lambda - i_1) \big]_{J - j_2} }{\big[ 2 \eta \Lambda \big]_J} \\
	& \qquad \times \displaystyle\frac{\big[ \lambda + 2 \eta i_2 \big]_{J - j_1} \big[ \lambda + 2 \eta (i_2 + j_1 - \Lambda - 1) \big]_{j_1}}{\big[ \lambda + 2 \eta j_1 \big]_{J - j_1} \big[ \lambda + 2 \eta (2j_1 - J - 1) \big]_{j_1} },
	\end{aligned}
	\end{flalign}
	
	\noindent if $i_1 + j_1 = i_2 + j_2$ and $j_1, j_2 \in \{ 0, 1, \ldots , J \}$ and is otherwise equal to $0$. Now the proposition follows from inserting \eqref{wjlambdaxyspecialization} and \Cref{cjlambdaelliptic} into \eqref{sjlambdaelliptic}. 
\end{proof}

We will not analyze further degenerations of our stochastic fused elliptic model and instead leave such investigation for future work. Let us conclude this section by mentioning that the stochasticity of the general $S_{J; \Lambda}$ weights follows as a special case of \Cref{stochasticws}. However, in each of the three situations explained above, this stochasticity can be checked directly. In the six-vertex and higher spin elliptic models (given by \Cref{jlambda1s} and \Cref{1js}), this is a consequence of \eqref{quarticrelationf}. In the elliptic variant of the $q$-Hahn boson model (given by \Cref{svlambdaelliptic}), this is a special case of the elliptic Jackson identity \eqref{hypergeometric109sumterminating}.

\section{Dynamical Higher Rank Vertex Model} 

\label{RankDynamical} 

In this section we apply the stochasticization procedure to solutions to the Yang-Baxter equation coming from higher rank vertex models, which were studied previously in \cite{CSVMST,CDFSVMD,SM}; this will give rise to dynamical variants of the models studied in \cite{CSVMST,CDFSVMD,SM}. We begin in \Cref{HigherSum} by stochasticizing these solutions and then we detail some special cases in \Cref{L1}. 

\subsection{Dynamical Higher Rank Model}

\label{HigherSum}

In this section we implement the stochasticization procedure on (the transposed version of) a solution to the Yang-Baxter equation due to Bosnjak-Mangazeev \cite{STRR} that arises from symmetric tensor representations of $U_q (\widehat{\mathfrak{sl}}_{n + 1})$. Let us begin by defining this solution. Throughout this section, we fix $q \in \mathbb{C}$ and $n \in \mathbb{Z}_{> 0}$.

The below definition originally appeared as equation (7.10) of \cite{STRR} under a change of parameters; as stated below, it appears as equation (C.4) of \cite{CSVMST}. In what follows, we recall that a \emph{composition} of \emph{length} $k + 1$ and \emph{size} $N$ is a sequence $\mathcal{C} = (C_0, C_1, \ldots , C_k) = (C_i)$ of nonnegative integers such that $\sum_{j = 0}^k C_j = N$. We denote $\ell (\mathcal{C}) = k + 1$ and, for any $0 \le i \le j \le k$, we set $C_{[i, j]} = \sum_{m = i}^j C_m$. Two compositions of the same length $k + 1$ can be added or subtracted as $(k + 1)$-dimensional vectors, although the result might not be a composition. For any finite set $\mathcal{S} \subset \mathbb{Z}$, we also define $|\mathcal{S}| = \sum_{s \in \mathcal{S}} s$. 

\begin{definition}[{\cite[Equation (7.10)]{STRR}}] 
	
	\label{wzabcd}
	 For any $x, y \in \mathbb{C}$ and $n$-tuples $\lambda = (\lambda_1, \ldots , \lambda_n) \in \mathbb{Z}_{\ge 0}^n$ and $\mu = (\mu_1, \ldots , \mu_n) \in \mathbb{Z}_{\ge 0}^n$, define 
	\begin{flalign*}
	\Phi (\lambda, \mu; x, y) = \displaystyle\frac{(x; q)_{|\lambda|} (x^{-1} y; q)_{|\mu| - |\lambda|} }{(y; q)_{|\mu|}} \left( \displaystyle\frac{y}{x} \right)^{|\lambda|} q^{\sum_{1 \le i < j \le n} (\mu_i - \lambda_i) \lambda_j} \displaystyle\prod_{j = 1}^n \displaystyle\frac{(q; q)_{\mu_i}}{(q; q)_{\lambda_i} (q; q)_{\mu_i - \lambda_i}}.
	\end{flalign*}
	
	\noindent Now let $z \in \mathbb{C}$ and $L, M \in \mathbb{Z}_{> 0}$. Let $\textbf{A} = (A_i)$, $\textbf{B} = (B_i)$, $\textbf{C} = (C_i)$, and $\textbf{D} = (D_i)$ be four compositions of length $n + 1$. Define $\overline{\textbf{A}} = (A_1, A_2, \ldots , A_n) \in \mathbb{Z}_{\ge 0}^n$ (that is, $\overline{\textbf{A}}$ is obtained by removing $A_0$ from $\textbf{A}$), and similarly define $\overline{\textbf{B}}$, $\overline{\textbf{C}}$, and $\overline{\textbf{D}}$. 
	
	If $\textbf{A} + \textbf{B} = \textbf{C} + \textbf{D}$, $|\textbf{A}| = M = |\textbf{C}|$, and $|\textbf{B}| = L = |\textbf{D}|$, set 
	\begin{flalign} 
	\label{wabcd} 
	\begin{aligned}
	& U_{L; M} \big( \textbf{A}, \textbf{B}; \textbf{C}, \textbf{D} \b| z \big) \\
	& \quad = z^{|\overline{\textbf{D}}| - |\overline{\textbf{B}}|} q^{L |\overline{\textbf{A}}| - M |\overline{\textbf{{D}}}|} \displaystyle\sum_{\textbf{P} \le \overline{\textbf{B}}, \overline{\textbf{C}}} \Phi \big( \overline{\textbf{C}} - \textbf{P}, \overline{\textbf{C}} + \overline{\textbf{D}} - \textbf{P}; q^{L - M} z, q^{-M} z \big) \Phi \big( \textbf{P}, \overline{\textbf{B}}; q^{-L} z^{-1}, q^{-L} \big),  
	\end{aligned}
	\end{flalign} 
	
	\noindent where the sum is over all $\textbf{P} = (P_1, \ldots , P_n) \in \mathbb{Z}_{\ge 0}^n$ such that $0 \le P_i \le \min \{ B_i, C_i \}$ for each $i \in [1, n]$. Otherwise, set  $U_{L; M} \big( \textbf{A}, \textbf{B}; \textbf{C}, \textbf{D} \b| z \big) = 0$.

\end{definition}

The following proposition states that these $U$ weights satisfy the Yang-Baxter equation; it is due to equation (3.20) of \cite{STRR} but, as stated below, it appears as Theorem C.1.1 of \cite{CSVMST}. 

\begin{prop}[{\cite[Equation (3.20)]{STRR}}] 
	
\label{equationrank} 

Fix $L, M, T \in \mathbb{Z}_{> 0}$; $x, y, z \in \mathbb{C}$; and length $n + 1$ compositions $\textbf{\emph{I}}_1, \textbf{\emph{J}}_1, \textbf{\emph{K}}_1$ and $\textbf{\emph{I}}_3, \textbf{\emph{J}}_3, \textbf{\emph{K}}_3$. Then, 
\begin{flalign} 
\label{equationvertexrank} 
\begin{aligned}
& \displaystyle\sum_{\textbf{\emph{I}}_2, \textbf{\emph{J}}_2, \textbf{\emph{K}}_2}  U_{L; M}  \left( \textbf{\emph{I}}_1, \textbf{\emph{J}}_1; \textbf{\emph{I}}_2, \textbf{\emph{J}}_2 \Big| \frac{x}{y} \right) U_{L; T}  \left( \textbf{\emph{K}}_1, \textbf{\emph{J}}_2; \textbf{\emph{K}}_2, \textbf{\emph{J}}_3 \Big| \frac{x}{z} \right) U_{M; T}  \left( \textbf{\emph{K}}_2, \textbf{\emph{I}}_2; \textbf{\emph{K}}_3, \textbf{\emph{I}}_3 \Big| \frac{y}{z} \right) \\
& \quad = \displaystyle\sum_{\textbf{\emph{I}}_2, \textbf{\emph{J}}_2, \textbf{\emph{K}}_2}   U_{M; T}  \left( \textbf{\emph{K}}_1, \textbf{\emph{I}}_1; \textbf{\emph{K}}_2, \textbf{\emph{I}}_2 \Big| \frac{y}{z} \right) U_{L; T}  \left( \textbf{\emph{K}}_2, \textbf{\emph{J}}_1; \textbf{\emph{K}}_3, \textbf{\emph{J}}_2 \Big| \frac{x}{z} \right) U_{L; M}  \left( \textbf{\emph{I}}_2, \textbf{\emph{J}}_2; \textbf{\emph{I}}_3, \textbf{\emph{J}}_3 \Big| \frac{x}{y} \right),
\end{aligned}
\end{flalign}	

\noindent where the sums on both sides of \eqref{equationvertexrank} are over all compositions $\textbf{\emph{I}}_2, \textbf{\emph{J}}_2, \textbf{\emph{K}}_2$ of length $n + 1$. 
	\end{prop}

Before proceeding, let us mention that if we were to stochasticize the $U$ weights using the direct analogs of \eqref{si1j1i2j2vertex} or \eqref{sjlambdaelliptic} (and \eqref{cjlambdadefinition}), we would be required to evaluate weights of the form $U_{L; M} \big( \textbf{A}, \textbf{B}, \textbf{C}, \textbf{D} \b| z \big)$ with $|\overline{\textbf{D}}| = 0$. Similar to what was indicated by \Cref{wjj0}, one might expect these weights to factor, but this does not seem to be transparent directly from the definition \eqref{wabcd}. However, $U_{L; M} \big( \textbf{A}, \textbf{B}, \textbf{C}, \textbf{D} \b| z \big)$ can be quickly seen to factor if $|\overline{\textbf{B}}| = 0$, since then the sum on the right side of \eqref{wabcd} defining this weight would be supported on $\textbf{P} = (0, 0, \ldots , 0)$. Therefore, it will be useful for us to define ``transposed'' modifications of the $U$ weights in which $\textbf{B}$ and $\textbf{D}$ are reversed; this is given by the following definition. 

\begin{definition}
	
	\label{weightdefinitionu} 
	
	Adopting the notation of \Cref{wzabcd}, define 
	\begin{flalign}
	\label{uabcd}
	W_{L; M} \big( \textbf{A}, \textbf{B}, \textbf{C}, \textbf{D} \b| z \big) = U_{L; M} \big( \textbf{C}, \textbf{D}, \textbf{A}, \textbf{B} \b| z \big). 
	\end{flalign}
	
\end{definition}

The fact that the $W$ weights satisfy the Yang-Baxter equation is then a consequence of \Cref{equationrank}. 

\begin{cor}
 \label{equationranku} 
 
 Adopt the notation of \Cref{equationrank}. Then \eqref{equationvertexrank} holds with the $U$ weights there replaced by the $W$ weights from \Cref{weightdefinitionu}.
\end{cor}

As in previous sections, the $W$ weights and \Cref{equationranku} have diagrammatic interpretations. Specifically, let $\mathcal{D}$ be a finite subset of the graph $\mathbb{Z}_{\ge 0}^2$. An \emph{$n$-colored directed path ensemble} on $\mathcal{D}$ is a family of \emph{colored paths} connecting adjacent vertices of $\mathcal{D}$, each edge of which is either directed up or to the right and is assigned one of $n$ colors (which are labeled by the integers $\{ 1, 2, \ldots , n \}$). We assume that each vertical edge and each horizontal edge can accommodate up to $M$ and $L$ paths, respectively.

The (colored) arrow configuration associated with some vertex $u \in \mathcal{D}$ is defined to be the quadruple of compositions of length $n + 1$ given by $(\textbf{A}, \textbf{B}, \textbf{C}, \textbf{D})$ with $|\textbf{A}| = M = |\textbf{C}|$ and $|\textbf{B}| = L = |\textbf{D}|$. Letting $X_i$ be the $i$-th component of $\textbf{X}$ for each $X \in \{ A, B, C, D \}$, the integers $A_i$, $B_i$, $C_i$, and $D_i$ denote the numbers of incoming vertical, incoming horizontal, outgoing vertical, and outgoing horizontal arrows of color $i$, respectively, for each $i \in [1, n]$. Thus, $A_0 = M - A_{[1, n]}$, $B_0 = L - B_{[1, n]}$, $C_0 = M - C_{[1, n]}$, and $D_0 = L - D_{[1, n]}$ denote the numbers of additional arrows that can be accommodated in each direction; we sometimes view these as arrows of color $0$. We refer to \Cref{configurationrank} for example.

\begin{figure}
	
	\begin{center} 
		
		\begin{tikzpicture}[
		>=stealth,
		scale = .8
		]

		\draw[->,black, dashed] (-.1, .2) -- (-.1, 2);
		\draw[->,black, thick] (0, .2) -- (0, 2);
		\draw[->,black, thick] (.1, .2) -- (.1, 2);
		
		\draw[->,black, dotted] (-.1,-2) -- (-.1, -.2);
		\draw[->,black, dashed] (0,-2) -- (0, -.2);
		\draw[->,black, thick] (.1,-2) -- (.1, -.2);
		
		\draw[->,black, thick] (-2, -.15) -- (-.2, -.15);
		\draw[->,black, dashed] (-2, -.05) -- (-.2, -.05);
		\draw[->,black, dotted] (-2, .05) -- (-.2, .05);
		\draw[->,black, dotted] (-2,.15) -- (-.2, .15);
		
		\draw[->,black, dashed] (.2, -.15) -- (2, -.15); 
		\draw[->,black, dotted] (.2, -.05) -- (2, -.075); 
		\draw[->,black, dotted] (.2, .05) -- (2, .05); 
		\draw[->,black, dotted] (.2, .15) -- (2, .15); 
		
		\draw (0, -2) circle[radius = 0] node[below = 0]{$\textbf{A}$} node[below = 20]{$y$} node[below = 30]{$M$}; 
		\draw (-2, 0) circle[radius = 0] node[left = 0]{$\textbf{B}$} node[left = 20]{$x$} node[left = 30]{$L$}; 
		\draw (0, 2) circle[radius = 0] node[above = 0]{$\textbf{C}$}; 
		\draw (2, 0) circle[radius = 0] node[right = 0]{$\textbf{D}$};

		\draw (-1.2, 1) circle[radius = 0] node[scale = .7]{$q^{L + M - A_0 - B_0} v$}; 
		\draw (1, 1) circle[radius = 0] node[scale = .7]{$q^{L - D_0} v$}; 
		\draw (-1, -1) circle[radius = 0] node[scale = .7]{$q^{M - A_0} v$}; 
		\draw (1, -1) circle[radius = 0] node[scale = .7]{$v$};

		\filldraw[fill=white, draw=black] (0, 0) circle [radius=.2];

		\end{tikzpicture}
		
	\end{center}

	\caption{\label{configurationrank} Depicted above is a colored vertex $u$, where the dotted, dashed, and solid arrows are associated with colors $0$, $1$, and $2$, respectively. The arrow configuration of $u$ is $(\textbf{A}, \textbf{B}, \textbf{C}, \textbf{D})$, where $\textbf{A} = (1, 1, 1)$;  $\textbf{B} = (2, 1, 1)$;  $\textbf{C} = (0, 1, 2)$; and $\textbf{D} = (3, 1, 0)$ (so that $L = 4$ and $M = 3$). The dynamical parameters are also labeled in each of the four faces passing through $u$. } 
\end{figure}
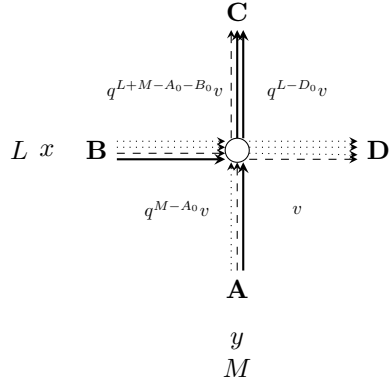

Then we view the quantity $W_{L; M} \big( \textbf{A}, \textbf{B}, \textbf{C}, \textbf{D} \b| \frac{x}{y} \big)$ as the weight associated with a vertex $u \in \mathcal{D}$ whose colored arrow configuration is $(\textbf{A}, \textbf{B}, \textbf{C}, \textbf{D})$. Here, $x$ and $y$ denote the rapidity parameters at $u$ in the horizontal and vertical directions, respectively, and $L$ and $M$ denote the spin parameters in the horizontal and vertical directions, respectively.  As in \Cref{SixVertexEquation}, the Yang-Baxter equation given by \Cref{equationranku} can be interpreted as moving a line through a cross, as depicted in \Cref{equationpaths} in the $n = 1$ case.

Now let us implement the stochasticization procedure on these $W_{L; M}$ weights; this will be similar to what was explained in \Cref{VertexModel} and \Cref{EllipticFused}. In particular, shaded vertices only output arrows of color $0$ (and therefore do not output arrows with colors between $1$ and $n$). Specifically, for fixed parameters $z \in \mathbb{C}$ and $T \in \mathbb{Z}_{> 0}$, the $\chi$ weights from \Cref{wchiconditions} will be of the form $W_{L; T} \big( \textbf{A}, \textbf{B}, \textbf{C}, L \textbf{e}_0 \b| \frac{x}{z} \big)$ and $W_{M; T} \big( \textbf{A}, \textbf{B}, \textbf{C}, M \textbf{e}_0 \b| \frac{y}{z} \big)$, where $k \textbf{e}_0$ denotes the composition $(k, 0, 0, \ldots , 0)$ of length $n + 1$ (with $k$ in coordinate zero and $0$ elsewhere) for each $k \in \mathbb{Z}_{\ge 0}$. The first and second conditions of \Cref{wchiconditions} follow from the facts that $W_{L; M} ( \textbf{A}, \textbf{B}, \textbf{C}, \textbf{D} ) = 0$ unless each part of $\textbf{A}, \textbf{B}, \textbf{C}, \textbf{D}$ is nonnegative, and \Cref{equationranku}, respectively. 

To verify the third condition, we must explain how to attach two shaded vertices to a given vertex $u$ in our domain, which will again be similar to what was explained in \Cref{DynamicalVertex1} and \Cref{EllipticFused}. In particular, we begin with a shaded stochasticization curve with rapidity parameter $z$ and spin parameter $T$ that starts to the right of our domain. Suppose that this curve initially has $H_i$ arrows of color $i$ for each $i \in [0, n]$; this gives rise to a composition $\textbf{H} = (H_0, H_1, \ldots , H_n)$ of length $n + 1$ and size $T$. 

Similarly to before, the stochasticization curve collects any arrow outputted by our domain and only outputs arrows of color $0$; see \Cref{arrowstochastic} for a depiction in the $n = 1$ case. We then move this stochasticization curve to the left, through one vertex of our domain at a time, as depicted in \Cref{arrowstochasticvertices}. The two vertices of the stochasticization curve adjacent to $u$ directly before the stochasticization curve passes through it will be the two shaded vertices we assign to $u$. 
	
Since $W_{L; M} (\textbf{A}, \textbf{B}, L \textbf{e}_0, M \textbf{e}_0) = \textbf{1}_{\textbf{A} = L \textbf{e}_0} \textbf{1}_{\textbf{B} = M \textbf{e}_0}$, the specialization of \eqref{ws} in our setting is given by the following definition. In the below, $\textbf{R} = (R_0, R_1, \ldots , R_n)$ denotes the composition such that there are $R_i$ arrows of color $i$ in the stochasticization curve before meeting $u$, for each $0 \le i \le n$.

\begin{definition} 
	
	\label{srankc} 

	Let $L, M, T \in \mathbb{Z}_{> 0}$; $x, y, z \in \mathbb{C}$; and $\textbf{A}, \textbf{B}, \textbf{C}, \textbf{D}$ be compositions of length $n + 1$. For any composition $\textbf{R}$ of size $T$ and length $n + 1$, define 
	\begin{flalign}
	\label{srankdefinition}
	S_{L; M} \big( \textbf{A}, \textbf{B}, \textbf{C}, \textbf{D} \b| x, y; T; \textbf{R}; z \big) = W_{L; M} \Big( \textbf{A}, \textbf{B}, \textbf{C}, \textbf{D} \Big| \frac{x}{y} \Big) C_{L; M} \big( \textbf{A}, \textbf{B}, \textbf{C}, \textbf{D} \b| x, y; T; \textbf{R}; z \big),
	\end{flalign}
	
	\noindent where we have denoted the stochastic corrections
	\begin{flalign}
	\label{rankcdefinition}
	\begin{aligned}
	C_{L; M} \big( \textbf{A}, \textbf{B}, \textbf{C}, \textbf{D} \b| x, y; T; \textbf{R}; z \big) & = \displaystyle\frac{W_{L; T} \big( \textbf{R}, \textbf{D}, \textbf{R} + \textbf{D} - L \textbf{e}_0, L \textbf{e}_0 \b| \frac{x}{z} \big)}{W_{L; T} \big( \textbf{R} + \textbf{A} - M \textbf{e}_0, \textbf{B}, \textbf{R} + \textbf{A} + \textbf{B} - (L + M) \textbf{e}_0, L \textbf{e}_0  \b| \frac{x}{z} \big)} \\
	& \quad \times \displaystyle\frac{W_{M; T} \big( \textbf{R} + \textbf{D} - L \textbf{e}_0, \textbf{C}, \textbf{R} + \textbf{C} + \textbf{D} - (L + M) \textbf{e}_0, M \textbf{e}_0  \b| \frac{y}{z} \big)}{W_{M; T} \big( \textbf{R}, \textbf{A}, \textbf{R} + \textbf{A} - M \textbf{e}_0, M \textbf{e}_0  \b| \frac{y}{z} \big)}.
	\end{aligned} 
	\end{flalign}
	
\end{definition} 

Now we proceed to determine these stochasticized weights. To that end, we begin with the following lemma, which explicitly evaluates the stochastic correction $C_{L; M}$.

\begin{lem}
	
\label{rankc} 

Adopt the notation of \Cref{srankc}, and let $\textbf{\emph{A}} = (A_i)$, $\textbf{\emph{B}} = (B_i)$, $\textbf{\emph{C}} = (C_i)$, $\textbf{\emph{D}} = (D_i)$, and $\textbf{\emph{R}} = (R_i)$ for $i \in [0, n]$. Denoting $v = q^{-R_0} z^{-1}$, we have that
\begin{flalign}
\label{cproductrank}
\begin{aligned}
C_{L; M} & \big( \textbf{\emph{A}}, \textbf{\emph{B}}, \textbf{\emph{C}}, \textbf{\emph{D}} \b| x, y; T; \textbf{\emph{R}}; z \big) \\
	& = q^{(M - L) (D_0 - B_0) + A_0 B_0 - C_0 D_0 + \sum_{1 \le i < j \le n} (D_j C_i - A_j B_i)} \left( \displaystyle\frac{y}{x} \right)^{D_0 - B_0}\\
	& \qquad \times \displaystyle\frac{(q^{L - D_0} xv; q)_{D_0}}{(q^{L + M - A_0 - B_0} xv; q)_{B_0}}  \displaystyle\frac{(q^{L + M -C_0 - D_0} yv; q)_{C_0}}{(q^{M - A_0} yv; q)_{A_0}}  \displaystyle\prod_{j = 0}^n \displaystyle\frac{(q; q)_{A_i} (q; q)_{B_i}}{ (q; q)_{C_i} (q; q)_{D_i}}.
\end{aligned}
\end{flalign}
\end{lem} 

\begin{rem}

\label{cproductzr} 

Observe that \eqref{cproductrank} indicates that the stochastic correction $C_{L; M}$ is only dependent on $T$, $\textbf{R}$, and $z$ through $v$. The fact that these three parameters reduce to one in this way was not transparent to us before explicitly evaluating $C_{L; M}$. 

\end{rem}

\begin{proof}[Proof of \Cref{rankc}]
	
In order to determine the stochastic correction $C_{L; M}$ using \eqref{rankcdefinition}, we must first evaluate weights of the form $W \big( \textbf{A}, \textbf{B}, \textbf{C}, \textbf{D} \big) = U \big( \textbf{C}, \textbf{D}, \textbf{A}, \textbf{B} \big)$ when $ \big| \overline{\textbf{D}} \big| = 0$. To that end observe in the sum on the right side of \eqref{wabcd} that we must have that $|\textbf{P}| = 0$ if $|\overline{\textbf{D}}| = 0$. Thus, that sum consists of one term, and so we deduce that 
\begin{flalign}
\label{wijkd0}
\begin{aligned}
W_{L; M} & \big( \textbf{I}, \textbf{J}, \textbf{K}, L \textbf{e}_0 \b| z \big) \\
& = z^{L - J_0} q^{(M - L)(J_0 - L)} \displaystyle\frac{(q^{L - M} z; q)_{M - I_0} (q^{-L}; q)_{L - J_0} }{(q^{-M} z; q)_{L + M - I_0 - J_0}} q^{\sum_{1 \le s < t \le n} I_t J_s} \displaystyle\prod_{s = 1}^n \displaystyle\frac{(q; q)_{I_s + J_s}}{(q; q)_{I_s} (q; q)_{J_s}} \\
& = z^{L - J_0} q^{(M - L)(J_0 - L)} \displaystyle\frac{(q^{- K_0} z; q)_{J_0} (q^{-L}; q)_{L - J_0} }{(q^{-M} z; q)_L} q^{\sum_{1 \le s < t \le n} I_t J_s} \displaystyle\prod_{s = 1}^n \displaystyle\frac{(q; q)_{K_s}}{(q; q)_{I_s} (q; q)_{J_s}},
\end{aligned}
\end{flalign}

\noindent for any compositions $\textbf{I} = (I_0, I_1, \ldots , I_n)$, $\textbf{J} = (J_0, J_1, \ldots , J_n)$, and $\textbf{K} = (K_0, K_1, \ldots , K_n)$ such that $\textbf{I} + \textbf{J} = \textbf{K} + L \textbf{e}_0$. Here, we have used the first identity in \eqref{elliptichypergeometricidentities1} and the facts that $\big|\overline{\textbf{I}} \big| = \big| \textbf{I} \big| - I_0$ and similarly for $\textbf{J}$ and $\textbf{K}$;  that $K_0 + L = I_0 + J_0$; and that $I_m + J_m = K_m$ for each $1 \le m \le n$.  	

Inserting \eqref{wijkd0} into \eqref{rankcdefinition}, we find that 
\begin{flalign}
\label{c1}
\begin{aligned}
C_{L; M} & \big( \textbf{A}, \textbf{B}, \textbf{C}, \textbf{D} \b| x, y; T; \textbf{R}; z \big) \\ 
&  = q^{(M - L) (D_0 - B_0) + \sum_{1 \le i < j \le n} (D_j C_i - A_j B_i)} \left( \displaystyle\frac{y}{x} \right)^{D_0 - B_0}   \displaystyle\prod_{j = 1}^n \displaystyle\frac{(q; q)_{A_i} (q; q)_{B_i}}{ (q; q)_{C_i} (q; q)_{D_i}} \\
&  \quad \times \displaystyle\frac{(q^{L - R_0 - D_0} x z^{-1}; q)_{D_0}}{(q^{L + M - R_0 - A_0 - B_0} x z^{-1}; q)_{B_0}} \displaystyle\frac{(q^{L + M - R_0 -C_0 - D_0} y z^{-1}; q)_{C_0}}{(q^{M - R_0 - A_0} y z^{-1}; q)_{A_0}}  \displaystyle\frac{ (q^{-M}; q)_{M - C_0} (q^{-L}; q)_{L - D_0} }{(q^{-M}; q)_{M - A_0} (q^{-L}; q)_{L - B_0}}, 
\end{aligned}
\end{flalign}

\noindent where we have used the fact that $A_i + B_i = C_i + D_i$ for each $i \in [0, n]$. Now the lemma follows from \eqref{c1} and the fact that 
\begin{flalign*} 
\displaystyle\frac{ (q^{-M}; q)_{M - C_0} (q^{-L}; q)_{L - D_0} }{(q^{-M}; q)_{M - A_0} (q^{-L}; q)_{L - B_0} } = q^{A_0 B_0 - C_0 D_0} \displaystyle\frac{ (q; q)_{A_0} (q; q)_{B_0} }{(q; q)_{C_0} (q; q)_{D_0} },
\end{flalign*}

\noindent which holds due the first identity in \eqref{elliptichypergeometricidentities1} and the facts that $A_0 + B_0 = C_0 + D_0$ and $v = q^{-R_0} z^{-1}$. 
\end{proof}

Now we can quickly evaluate the stochasticized weights $S_{L; M}$. 

\begin{cor} 
	
\label{sweightsrank}

Adopt the notation from \Cref{rankc}. If $\textbf{\emph{A}} + \textbf{\emph{B}} = \textbf{\emph{C}} + \textbf{\emph{D}}$, $|\textbf{\emph{A}}| = M = |\textbf{\emph{C}}|$, and $|\textbf{\emph{B}}| = L = |\textbf{\emph{D}}|$, then 
\begin{flalign}
\label{svrank}
\begin{aligned}
S_{L; M} & \big( \textbf{\emph{A}}, \textbf{\emph{B}}, \textbf{\emph{C}}, \textbf{\emph{D}} \b| x, y; v \big) \\
& = q^{ M B_0 - L C_0  + (M - L) (D_0 - B_0) + A_0 B_0 - C_0 D_0 + \sum_{1 \le i < j \le n} (D_j C_i - A_j B_i)} \\
& \qquad \times \displaystyle\frac{(q^{L - D_0} v x; q)_{D_0}}{(q^{L + M - A_0 - B_0} v x; q)_{B_0}}  \displaystyle\frac{(q^{L + M - C_0 - D_0} v y; q)_{C_0}}{(q^{M - A_0} vy; q)_{A_0}}  \displaystyle\prod_{j = 0}^n \displaystyle\frac{(q; q)_{A_i} (q; q)_{B_i}}{ (q; q)_{C_i} (q; q)_{D_i}} \\
& \qquad \times \displaystyle\sum_{\textbf{\emph{P}} \le \overline{\textbf{\emph{A}}}, \overline{\textbf{\emph{D}}}} \Phi \big( \overline{\textbf{\emph{A}}} - \textbf{\emph{P}}, \overline{\textbf{\emph{A}}} + \overline{\textbf{\emph{B}}} - \textbf{\emph{P}}; q^{L - M} x y^{-1}, q^{-M} xy^{-1} \big) \Phi \big( \textbf{\emph{P}}, \overline{\textbf{\emph{D}}}; q^{-L} x^{-1} y, q^{-L} \big),
\end{aligned}
\end{flalign} 

\noindent where the function $\Phi$ was given in \Cref{wzabcd}, and the sum on the right side of \eqref{svrank} is over all $\textbf{\emph{P}} = (P_1, \ldots , P_n) \in \mathbb{Z}_{\ge 0}^n$ such that $0 \le P_i \le \min \{ A_i, D_i \}$ for each $i \in [1, n]$. Otherwise, $S_{L; M} \big( \textbf{\emph{A}}, \textbf{\emph{B}}; \textbf{\emph{C}}, \textbf{\emph{D}} \b| x, y; v \big) = 0$. 

Here, we have set $S_{L; M} \big( \textbf{\emph{A}}, \textbf{\emph{B}}, \textbf{\emph{C}}, \textbf{\emph{D}} \b| x, y; v \big)  = S_{L; M} \big( \textbf{\emph{A}}, \textbf{\emph{B}}, \textbf{\emph{C}}, \textbf{\emph{D}} \b| x, y; T; \textbf{\emph{R}}; z \big)$, since \eqref{svrank} indicates that it depends on $T$, $\textbf{\emph{R}}$, and $z$ only through $v$.
	
\end{cor}

\begin{proof}

This follows from inserting \eqref{cproductrank} into \eqref{srankdefinition}, and then applying \eqref{uabcd} and \Cref{wzabcd}.
\end{proof}

The parameter $v = q^{-R_0} z^{-1}$ can again be viewed as a dynamical parameter, so we sometimes denote $v = v(u)$ by the value of $v$ at a vertex $u$ in $\mathbb{Z}_{> 0}^2$. One can quickly verify that the identities governing $v(u)$ are given by 
\begin{flalign} 
\label{vdynamicalrank}
v (a - 1, b) = q^{M - A_0} v; \qquad v (a, b + 1) = q^{L - D_0} v,
\end{flalign}

\noindent for any $(a, b) \in \mathbb{Z}_{> 0}^2$ whose arrow configuration is $(\textbf{A}, \textbf{B}, \textbf{C}, \textbf{D})$. Indeed, this is due to the fact that if the stochasticization curve passes an edge with spin parameter $K$ with $X_0$ incoming arrows of color $0$, then it loses $K - X_0$ arrows of color $0$ (here, we view $X \in \{A, B, C, D \}$ and $K \in \{ L, M \}$). The identities \eqref{vdynamicalrank} are depicted in \Cref{configurationrank}, where $v(u)$ is labeled in the lower-right face containing $u$.

The following theorem states that these $S$ weights are stochastic and satisfy the (dynamical) Yang-Baxter equation.

\begin{thm} 

\label{dynamicalrankstochastic} 

Adopt the notation of \Cref{sweightsrank}. Then, $\sum_{\textbf{\emph{C}}, \textbf{\emph{D}}} S_{L; M} \big( \textbf{\emph{A}}, \textbf{\emph{B}}, \textbf{\emph{C}}, \textbf{\emph{D}} \b| x, y; v \big) = 1$ for any $v \in \mathbb{C}$, where the sum is over all length $n + 1$ compositions $\textbf{\emph{C}}$ and $\textbf{\emph{D}}$. 

Furthermore, for any fixed $L, M, T \in \mathbb{Z}_{> 0}$; $x, y, z, v \in \mathbb{C}$; and length $n + 1$ compositions $\textbf{\emph{I}}_1, \textbf{\emph{J}}_1, \textbf{\emph{K}}_1$ and $\textbf{\emph{I}}_3, \textbf{\emph{J}}_3, \textbf{\emph{K}}_3$, we have that 
\begin{flalign} 
\label{rankdynamicalequation} 
\begin{aligned}
 \displaystyle\sum_{\textbf{\emph{I}}_2, \textbf{\emph{J}}_2, \textbf{\emph{K}}_2}  & S_{L; M}  \left( \textbf{\emph{I}}_1, \textbf{\emph{J}}_1; \textbf{\emph{I}}_2, \textbf{\emph{J}}_2 \Big| x, y; q^{T - (\textbf{\emph{K}}_1)_0} v \right) S_{L; T}  \left( \textbf{\emph{K}}_1, \textbf{\emph{J}}_2; \textbf{\emph{K}}_2, \textbf{\emph{J}}_3 \Big| x, z; v \right) \\
&  \times S_{M; T}  \left( \textbf{\emph{K}}_2, \textbf{\emph{I}}_2; \textbf{\emph{K}}_3, \textbf{\emph{I}}_3 \Big| y, z; q^{L - (\textbf{\emph{J}}_3)_0} v \right) \\
 \quad = \displaystyle\sum_{\textbf{\emph{I}}_2, \textbf{\emph{J}}_2, \textbf{\emph{K}}_2}  & S_{M; T}  \left( \textbf{\emph{K}}_1, \textbf{\emph{I}}_1; \textbf{\emph{K}}_2, \textbf{\emph{I}}_2 \Big| y, z; v \right) S_{L; T}  \left( \textbf{\emph{K}}_2, \textbf{\emph{J}}_1; \textbf{\emph{K}}_3, \textbf{\emph{J}}_2 \Big| x, z; q^{M - (\textbf{\emph{I}}_2)_0} \right) \\
&  \times S_{L; M}  \left( \textbf{\emph{I}}_2, \textbf{\emph{J}}_2; \textbf{\emph{I}}_3, \textbf{\emph{J}}_3 \Big| x, y; v \right),
\end{aligned}
\end{flalign}

\noindent where the sums on both sides of \eqref{rankdynamicalequation} are over all length $n + 1$ compositions $\textbf{\emph{I}}_2, \textbf{\emph{J}}_2, \textbf{\emph{K}}_2$.

\end{thm}

\begin{proof}
	
	The first and second statements of this theorem comprise the degenerations of \Cref{stochasticws} and \Cref{relationwstochastic} to our setting, respectively. 
\end{proof}

\subsection{Higher Spin Degeneration} 

\label{L1}

In this section we evaluate the ``higher spin'' and ``spin $\frac{1}{2}$'' degenerations of the $S_{L; M}$ weights from \Cref{sweightsrank}. This is given by the following proposition (the terminology we use is parallel to in the $n = 1$ case).

In what follows, if a composition $\mathcal{C}$ of length $n + 1$ is equal to $\textbf{e}_j$ for some integer $j \in [0, n]$ (that is, the composition whose $i$-th component is $\textbf{1}_{i = j}$ for each $i$), we abbreviate $\mathcal{C} = j$. 

\begin{prop}
	
\label{l1s} 

Fix integers $M > 0$ and $1 \le h < j \le n$; complex numbers $v, x, y$; and compositions $\textbf{\emph{I}} = (I_m)$ and $\textbf{\emph{K}} = (K_m)$ of length $n + 1$ and size $M$. Abbreviating $S (\textbf{\emph{I}}, \textbf{\emph{B}}, \textbf{\emph{K}}, \textbf{\emph{D}}) = S_{1; M}  \big( \textbf{\emph{I}}, \textbf{\emph{B}}, \textbf{\emph{K}}, \textbf{\emph{D}} \b| x, y; v \big)$ for any compositions $\textbf{\emph{B}}$ and $\textbf{\emph{D}}$, we have that $S (\textbf{\emph{I}}, \textbf{\emph{B}}, \textbf{\emph{K}}, \textbf{\emph{D}}) = 0$ if arrow conservation does not hold (meaning that $\textbf{\emph{I}} + \textbf{\emph{B}} \ne \textbf{\emph{K}} + \textbf{\emph{D}}$). Otherwise,
\begin{flalign}
\label{m1rankdynamicals}
\begin{aligned}
S ( \textbf{\emph{I}}, j; \textbf{\emph{K}}, h ) & =  \displaystyle\frac{q^{I_{[1, h - 1]}} (1 - q^{I_h}) x (1 - q^M v y)}{(x - q^M y) (1 - q^{M - I_0} v y)}; \quad S  ( \textbf{\emph{I}}, h; \textbf{\emph{K}}, j )  =  \displaystyle\frac{q^{I_{[1, j - 1]}} (1 - q^{I_j})  y (1 - q^M v y) }{(x - q^M y) (1 - q^{M - I_0} v y)}; \\
S  ( \textbf{\emph{I}}, h; \textbf{\emph{K}}, h ) & = \displaystyle\frac{q^{I_{[1, h - 1]}} (x - q^{I_h} y) (1 - q^M v y)}{(x - q^M y) (1 - q^{M - I_0} v y)}; \quad S  ( \textbf{\emph{I}}, j; \textbf{\emph{K}}, 0 )  = \displaystyle\frac{q^{I_{[1, n]}} (1 - q^{I_0}) y (1 - v x)}{(x - q^M y) (1 - q^{M - I_0} v y)}; \\
S  ( \textbf{\emph{I}}, 0; \textbf{\emph{K}}, h ) & =  \displaystyle\frac{q^{I_{[1, h - 1]}} (1 - q^{I_h} ) x (1 - q^M vy)}{(x - q^M y) (1 - q^{M - I_0} v x)}; \quad S ( \textbf{\emph{I}}, 0; \textbf{\emph{K}}, 0  )  =   \displaystyle\frac{q^{I_{[1, n]}} (x - q^{I_0} y) (1 - vx)}{(x - q^M y) (1 - q^{M - I_0} v x)},
\end{aligned}
\end{flalign} 

\noindent assuming that arrow conservation holds in each of the above cases, and $S (\textbf{\emph{I}}, \textbf{\emph{B}}, \textbf{\emph{K}}, \textbf{\emph{D}}) = 0$ for any $(\textbf{\emph{I}}, \textbf{\emph{B}}, \textbf{\emph{K}}, \textbf{\emph{D}})$ not of the above form. Here, we recall that the dynamical parameter $v$ changes according to the identities indicated in \eqref{vdynamicalrank} and depicted in \Cref{configurationrank}. 

In particular, if $M = 1$, then these weights are given by \Cref{l1sdefinition}.
	
\end{prop}

\begin{proof}
	
Since the last statement of the proposition follows from setting $M = 1$ in \eqref{m1rankdynamicals}, it suffices to establish \eqref{m1rankdynamicals}.

To that end, first observe that by setting $L = 1$ in \Cref{wzabcd} and that, for any compositions $\textbf{I}$ and $\textbf{K}$ of length $n + 1$ and size $M$, and integers $0 \le h < j \le n$, we have that 
\begin{flalign}
\label{rzijij}
\begin{aligned}
U_{1; M} \big(\textbf{I}, j; \textbf{I}, j \b| z \big) = & \displaystyle\frac{(1 - q^{I_j} z) q^{I_{[j + 1, n]}}}{1 - q^M z}; \qquad U_{1; M} \big( \textbf{I}, h; \textbf{K}, j \b| z \big) = \displaystyle\frac{(1 - q^{I_j}) q^{I_{[j + 1, n]}}}{1 - q^M z}; \\
& U_{1; M} \big(\textbf{I}, j; \textbf{K}, h \b| z \big) = \displaystyle\frac{(1 - q^{I_h}) q^{I_{[h + 1, n]}} z}{1 - q^M z}.
\end{aligned} 
\end{flalign}

\noindent Furthermore, \Cref{rankc} applied with $L = 1$ yields that 	
\begin{flalign}
\label{cl1rank} 
\begin{aligned}
&  C (\textbf{I}, 0, \textbf{I}, 0) = \displaystyle\frac{1 - v x}{1 - q^{M - I_0} v x}; \qquad  C (\textbf{I}, j, \textbf{K}, 0) = q^{I_{[1, j]}} \displaystyle\frac{y (1 - q^{I_0}) (1 - v x)}{x (1 - q^{K_j}) (1 - q^{M - I_0} v y)}; \\
&  \qquad \qquad \qquad  C (\textbf{I}, 0, \textbf{K}, j) = q^{-K_{[j, n]}}  \displaystyle\frac{x (1 - q^{I_j}) (1 - q^M v y)}{y (1 - q^{K_0}) (1 - q^{M - I_0} v x)}; \\ 
& \qquad \qquad C (\textbf{I}, j_1, \textbf{K}, j_2) = q^{K_{[1, j_2 - 1]} - I_{[j_1 + 1, n]}} \displaystyle\frac{(1 - q^{I_{j_2}}) (1 - q^M vy)}{(1 - q^{K_{j_1}}) (1 - q^{M - I_0} v y)}, 
\end{aligned}
\end{flalign} 

\noindent for any $j, j_1, j_2 \in [1, n]$, where we have set $C ( \textbf{I}, m_1, \textbf{K}, m_2) = C_{1; M} \big( \textbf{I}, m_1, \textbf{I}, m_2 \b| x, y; T; \textbf{R}; z \big)$ for any integers $0 \le m_1, m_2 \le n$. Now \eqref{m1rankdynamicals} follows from \eqref{rzijij}, \eqref{cl1rank}, \eqref{srankdefinition}, and \eqref{uabcd}.
\end{proof}

\begin{rem}
	
	Since both the stochasticity and dynamical Yang-Baxter equation \eqref{rankdynamicalequation} for the $S$ weights from  \eqref{m1rankdynamicals} are identities of rational functions in $q^M$, we can analytically continue in $M$. Specifically, after replacing $q^M$ with an arbitrary complex parameter and omitting the size $M$ constraint on the compositions $\textbf{I}$, $\textbf{J}$, and $\textbf{K}$, the $S$ weights still remain stochastic and satisfy the dynamical Yang-Baxter equation.

\end{rem}

\begin{rem}
	
\label{sweightsrank2} 
	
Replacing $q^M$ by $s^2$, $y$ by $s^{-1}$, $x$ by $x^{-1}$, and letting $v$ tend to $\infty$ recovers the stochastic higher spin, higher rank vertex model studied in \cite{CSVMST} (see in particular equations (1.2.2) and (2.5.1) of that paper). These weights, as well as their $J \ge 1$ fused generalizations, satisfy the non-dynamical Yang-Baxter equation. 
\end{rem}

\section{A Dynamical Stochastic Tetrahedron Model} 

\label{DynamicalDimension}

Although the stochasticization procedure from \Cref{Domain} was explained for models on a two-dimensional graph, it can also be applied in higher dimensions. In this section, we provide an example of this by stochasticizing Mangazeev-Bazhanov-Sergeev's solution \cite{ILMPW} to the tetrahedron equation, which is a three-dimensional analog of the Yang-Baxter equation. This leads to a family of stochastic weights that satisfies a dynamical variant of the tetrahedron equation. We implement this stochasticization procedure in \Cref{DimensionStochastic}, derive properties of the stochasticized weights in \Cref{EquationStochastic3}, and discuss one of its degenerations that solves the original (non-dynamical) tetrahedron equation in \Cref{EquationDimension2}.

\subsection{Stochasticizing a Solution of the Tetrahedron Equation}

\label{DimensionStochastic}

Let us begin by defining the vertex weights that we will stochasticize; these weights are due to Mangazeev-Bazhanov-Sergeev and can be found as equation (31) of \cite{ILMPW}. Throughout this section, we fix a parameter $q \in \mathbb{C}$.

\begin{definition}[{\cite[Equation (31)]{ILMPW}}]

\label{weightsequation3}

For any nonnegative integers $n_1, n_2, n_3, n_1', n_2', n_3'$, define 
\begin{flalign}
\label{rn1n2n3}
\begin{aligned}
R_{n_1 n_2 n_3}^{n_1' n_2' n_3'} & = q^{n_2 (n_2 + 1) - (n_2 - n_1') (n_2 - n_3')} \displaystyle\frac{(q^{-2n_1'}; q^2)_{n_2}}{(q^2; q^2)_{n_2}} {_2 \varphi_1}  \Bigg( \begin{array}{cc}  q^{-2n_2}, q^{2n_1 + 2} \\ q^{2(n_1' - n_2 + 1)} \end{array}\Bigg| q^2, q^{-2n_3}\Bigg) \\
& \qquad \times \textbf{1}_{n_1 + n_2 = n_1' + n_2'} \textbf{1}_{n_2 + n_3 = n_2' + n_3'}.
\end{aligned}
\end{flalign}

\end{definition}

The diagrammatic interpretation of the quantities $R_{n_1 n_2 n_3}^{n_1' n_2' n_3'}$ is that they are weights associated with a vertex $u$ in $\mathbb{Z}^3$ with arrow configuration $(n_1, n_2, n_3; n_1', n_2', n_3')$; here, this means that $u$ has $n_1$ incoming arrows parallel to one direction (say the $x$-axis); $n_2$ incoming arrows parallel to the $y$-axis; $n_3$ incoming arrows parallel to the $z$-axis; $n_1'$ outgoing arrows parallel to the $x$-axis; $n_2'$ outgoing arrows parallel to the $y$-axis; and $n_3'$ outgoing arrows parallel to the $z$-axis. Observe that, with this understanding, the $R$ weights do not satisfy arrow conservation, in that they are not supported on arrow configurations $(n_1, n_2, n_3; n_1', n_2', n_3')$ satisfying $n_1 + n_2 + n_3 = n_1' + n_2' + n_3'$.

The following proposition, which was established in Section 2 of \cite{ILMPW}, indicates that the $R_{n_1 n_2 n_3}^{n_1' n_2' n_3'}$ weights are nonnegative and satisfy the \emph{tetrahedron equation}, given by \eqref{equation31} below.

\begin{prop}[{\cite[Section 2]{ILMPW}}]
	
	\label{requation3}
	
	For any fixed nonnegative integers $n_1, n_2, n_3, n_4, n_5, n_6$ and $n_1'', n_2'', n_3'', n_4'', n_5'', n_6''$, we have that
	\begin{flalign}
	\label{equation31} 
	& \displaystyle\sum_{\textbf{\emph{n}}'} R_{n_1 n_2 n_3}^{n_1' n_2' n_3'} R_{n_1' n_4 n_5}^{n_1'' n_4' n_5'} R_{n_3' n_5' n_6'}^{n_3'' n_5'' n_6''} R_{n_2' n_4' n_6}^{n_2'' n_4'' n_6'}  = \displaystyle\sum_{\textbf{\emph{n}}'} R_{n_3 n_5 n_6}^{n_3' n_5' n_6'} R_{n_2 n_4 n_6'}^{n_2' n_4' n_6''} R_{n_1 n_4' n_5'}^{n_1' n_4'' n_5''} R_{n_1' n_2' n_3'}^{n_1'' n_2'' n_3''}, 
	\end{flalign} 

	\noindent where the sums on both sides of \eqref{equation31} are over nonnegative integer sets $\textbf{\emph{n}}' = (n_1', n_2', n_3', n_4', n_5', n_6')$. 
	
	Furthermore, if $q \in (0, 1)$, then $R_{n_1 n_2 n_3}^{n_1' n_2' n_3'} \ge 0$ for any $n_1, n_2, n_3, n_1', n_2', n_3' \in \mathbb{Z}_{\ge 0}$.
\end{prop}

The tetrahedron equation is sometimes viewed as ``pushing'' a plane through a vertex (similar to how the Yang-Baxter equation is as pushing a line through a vertex) and is depicted in \Cref{3equationdimension1}.

\begin{figure}

	\begin{center}

		\begin{tikzpicture}[
		>=stealth,
		auto,
		style={
			scale = 1.5	
		}
		]
		
		\draw[-, black] (2.5, 1.5) -- (2, 1);
		\draw[-, black] (2, 1.5) -- (2, 1);
		\draw[-, black] (1.5, 1.5) -- (2, 1);
		\draw[-, black] (.5, .5) -- (1, 0);
		\draw[-, black] (.5, 0) -- (1, 0);
		\draw[-, black] (1.5, -1.5) -- (2, -1);
		
		\draw[->, black] (1, 0) -- (.5, -.5);
		\draw[->, black] (2, -1) -- (2, -1.5);
		\draw[->, black] (3, 0) -- (3.5, -.5);
		\draw[->, black] (2, -1) -- (2.5, -1.5);
		\draw[->, black] (3, 0) -- (3.5, 0);
		\draw[->, black] (3, 0) -- (3.5, .5);

		\draw[-, black, dashed] (1, 0) -- (3, 0); 
		\draw[-, black, dashed] (1, 0) -- (2, 1); 
		\draw[-, black, dashed] (1, 0) -- (2, -1); 
		\draw[-, black, dashed] (3, 0) -- (2, 1); 
		\draw[-, black, dashed] (3, 0) -- (2, -1); 
		\draw[-, black, dashed] (2, 1) -- (2, -1);

		\filldraw[fill = black] (2.5, 1.5) circle[radius = 0] node[scale = .8, above = 0]{$n_1$};
		\filldraw[fill = black] (2, 1.5) circle[radius = 0] node[scale = .8, above = 0]{$n_2$};
		\filldraw[fill = black] (1.5, 1.5) circle[radius = 0] node[scale = .8, above = 0]{$n_3$};
		\filldraw[fill = black] (.5, .5) circle[radius = 0] node[scale = .8, left = 0]{$n_4$};
		\filldraw[fill = black] (.5, 0) circle[radius = 0] node[scale = .8, left = 0]{$n_5$};
		\filldraw[fill = black] (1.5, -1.5) circle[radius = 0] node[scale = .8, left = 0]{$n_6$};

		\filldraw[fill = black] (.5, -.5) circle[radius = 0] node[scale = .8, left = 0]{$n_1''$};
		\filldraw[fill = black] (2, -1.5) circle[radius = 0] node[scale = .8, below = 0]{$n_2''$};
		\filldraw[fill = black] (3.5, -.5) circle[radius = 0] node[scale = .8, right = 0]{$n_3''$};
		\filldraw[fill = black] (2.5, -1.5) circle[radius = 0] node[scale = .8, right = 0]{$n_4''$};
		\filldraw[fill = black] (3.5, 0) circle[radius = 0] node[scale = .8, right = 0]{$n_5''$};
		\filldraw[fill = black] (3.5, .5) circle[radius = 0] node[scale = .8, right = 0]{$n_6''$};

		\filldraw[fill = black] (1.5, .5) circle[radius = 0] node[scale = .8, left = 3]{$n_1'$};
		\filldraw[fill = black] (2, .375) circle[radius = 0] node[scale = .8, right = -1]{$n_2'$};
		\filldraw[fill = black] (2.5, .5) circle[radius = 0] node[scale = .8, right = 2]{$n_3'$};
		\filldraw[fill = black] (1.5, -.5) circle[radius = 0] node[scale = .8, left = 2]{$n_4'$};
		\filldraw[fill = black] (1.625, 0) circle[radius = 0] node[scale = .8, below = -1]{$n_5'$};
		\filldraw[fill = black] (2.5, -.5) circle[radius = 0] node[scale = .8, left = 2]{$n_6'$};

		\draw[->, black] (8, 1) -- (8.5, 1.5);
		\draw[-, black] (8, 1.5) -- (8, 1);
		\draw[-, black] (7.5, 1.5) -- (8, 1);
		\draw[-, black] (6.5, .5) -- (7, 0);
		\draw[-, black] (6.5, 0) -- (7, 0);
		\draw[->, black] (8, -1) -- (7.5, -1.5);
		
		\draw[-, black] (6.5, -.5) -- (7, 0);
		\draw[->, black] (8, -1) -- (8, -1.5);
		\draw[->, black] (9, 0) -- (9.5, -.5);
		\draw[->, black] (8, -1) -- (8.5, -1.5);
		\draw[->, black] (9, 0) -- (9.5, 0);
		\draw[-, black] (9.5, .5) -- (9, 0);

		\draw[-, black, dashed] (7, 0) -- (9, 0); 
		\draw[-, black, dashed] (7, 0) -- (8, 1); 
		\draw[-, black, dashed] (7, 0) -- (8, -1); 
		\draw[-, black, dashed] (9, 0) -- (8, 1); 
		\draw[-, black, dashed] (9, 0) -- (8, -1); 
		\draw[-, black, dashed] (8, 1) -- (8, -1);

		\filldraw[fill = black] (9.5, .5) circle[radius = 0] node[scale = .8, right = 0]{$n_1$};
		\filldraw[fill = black] (8, 1.5) circle[radius = 0] node[scale = .8, above = 0]{$n_2$};
		\filldraw[fill = black] (6.5, .5) circle[radius = 0] node[scale = .8, left = 0]{$n_3$};
		\filldraw[fill = black] (7.5, 1.5) circle[radius = 0] node[scale = .8, above = 0]{$n_4$};
		\filldraw[fill = black] (6.5, 0) circle[radius = 0] node[scale = .8, left = 0]{$n_5$};
		\filldraw[fill = black] (6.5, -.5) circle[radius = 0] node[scale = .8, left = 0]{$n_6$};

		\filldraw[fill = black] (7.5, -1.5) circle[radius = 0] node[scale = .8, left = 0]{$n_1''$};
		\filldraw[fill = black] (8, -1.5) circle[radius = 0] node[scale = .8, below = 0]{$n_2''$};
		\filldraw[fill = black] (8.5, -1.5) circle[radius = 0] node[scale = .8, right = 0]{$n_3''$};
		\filldraw[fill = black] (9.5, -.5) circle[radius = 0] node[scale = .8, right = 0]{$n_4''$};
		\filldraw[fill = black] (9.5, 0) circle[radius = 0] node[scale = .8, right = 0]{$n_5''$};
		\filldraw[fill = black] (8.5, 1.5) circle[radius = 0] node[scale = .8, above = 0]{$n_6''$};
		
		\filldraw[fill = black] (8.5, -.5) circle[radius = 0] node[scale = .8, left = 2]{$n_1'$};
		\filldraw[fill = black] (8, .375) circle[radius = 0] node[scale = .8, right = -1]{$n_2'$};
		\filldraw[fill = black] (7.5, -.5) circle[radius = 0] node[scale = .8, left = 2]{$n_3'$};
		\filldraw[fill = black] (8.5, .5) circle[radius = 0] node[scale = .8, right = 2]{$n_4'$};
		\filldraw[fill = black] (7.625, 0) circle[radius = 0] node[scale = .8, below = -1]{$n_5'$};
		\filldraw[fill = black] (7.5, .5) circle[radius = 0] node[scale = .8, left = 3]{$n_6'$};

		\end{tikzpicture}
		
	\end{center}

	\caption{\label{3equationdimension1} The tetrahedron equation, as in \Cref{requation3}, is depicted above; the numbers of arrows on the dashed lines is summed over. }

\end{figure}
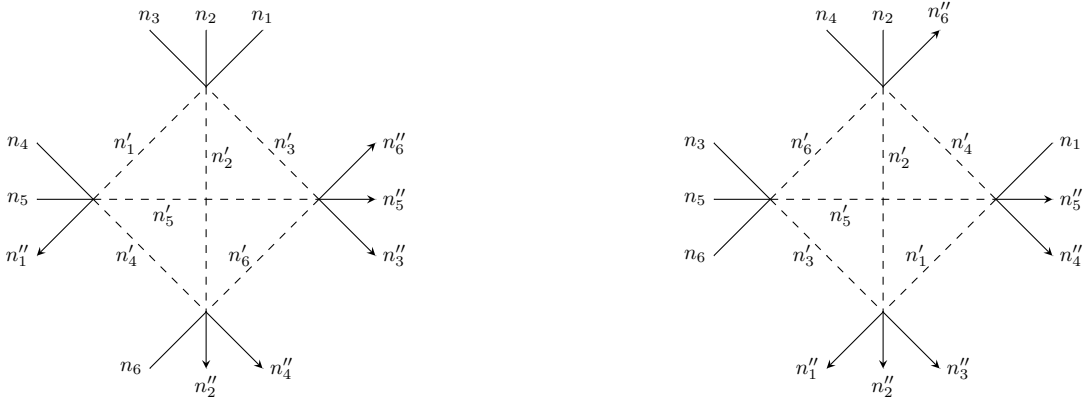

Following either \Cref{SixVertexEquation} or \Cref{DefinitionWeights}, we would then like to find a boundary condition $(n_1'', n_2'', n_3'', n_4'', n_5'', n_6'')$ such that the right side of \eqref{equation31} has at most one nonzero term. This can be imposed by setting $n_1'' = n_2'' = n_3'' = 0$ and allowing $(n_4, n_5, n_6) = (k_1, k_2, k_3) \in \mathbb{Z}_{\ge 0}^3$ to be arbitrary. Indeed, the fact that $R_{n_1 n_2 n_3}^{n_1' n_2' n_3'} = 0$ unless $n_1 + n_2 = n_1' + n_2'$ and $n_2 + n_3 = n_2' + n_3'$ implies that $R_{n_1' n_2' n_3'}^{0 0 0} = 0$ unless $(n_1, n_2', n_3') = (0, 0, 0)$. Under this setting, and once $(n_4, n_5, n_6) = (k_1, k_2, k_3)$ is fixed, then $(n_4', n_5', n_6') = (a_4', a_5', a_6')$ and $(n_4'', n_5'', n_6'') = (a_4'', a_5'', a_6'')$ on the right sides of \eqref{equation31} are also fixed to be  
\begin{flalign}
\label{n4n5n6}
\begin{aligned}
& a_4' = k_4 + n_2; \qquad \qquad	 a_5' = k_5 + n_3; \qquad \qquad  a_6' = k_6 - n_3; \\
&  a_4'' = k_4 + n_1 + n_2; \qquad a_5'' = k_5 - n_1 + n_3; \qquad a_6'' = k_6 - n_2 - n_3,
\end{aligned}
\end{flalign}

\noindent again due to the fact that $R_{n_1 n_2 n_3}^{n_1' n_2' n_3'} = 0$ unless $n_1 + n_2 = n_1' + n_2'$ and $n_2 + n_3 = n_2' + n_3'$. For the same reason, the $(n_4', n_5', n_6') = (b_4', b_5', b_6')$ from the left side of \eqref{equation31} is fixed to be 
\begin{flalign} 
\label{b4b5b6}
b_4' = k_4 + n_1'; \qquad b_5' = k_5 - n_1'; \qquad b_6' = k_6 - n_2'. 
\end{flalign} 

\noindent This is depicted in \Cref{vertex3}, which indicates the special case of the tetrahedron equation obtained by setting $(n_1'', n_2'', n_3'') = (0, 0, 0)$.

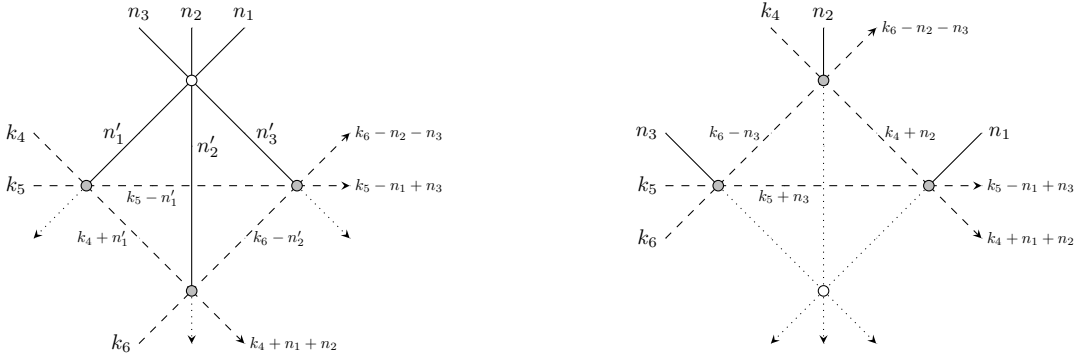
\begin{figure}

	\begin{center}

		\begin{tikzpicture}[
		>=stealth,
		auto,
		style={
			scale = 1.4
		}
		]
		
		\draw[-, black] (2.5, 1.5) -- (2, 1);
		\draw[-, black] (2, 1.5) -- (2, 1);
		\draw[-, black] (1.5, 1.5) -- (2, 1);
		\draw[-, black, dashed] (.5, .5) -- (1, 0);
		\draw[-, black, dashed] (.5, 0) -- (1, 0);
		\draw[-, black, dashed] (1.5, -1.5) -- (2, -1);
		
		\draw[->, dotted] (1, 0) -- (.5, -.5);
		\draw[->, dotted] (2, -1) -- (2, -1.5);
		\draw[->, black, dotted] (3, 0) -- (3.5, -.5);
		\draw[->, black, dashed] (2, -1) -- (2.5, -1.5);
		\draw[->, black, dashed] (3, 0) -- (3.5, 0);
		\draw[->, black, dashed] (3, 0) -- (3.5, .5);

		\draw[-, black, dashed] (1, 0) -- (3, 0); 
		\draw[-, black] (1, 0) -- (2, 1); 
		\draw[-, black, dashed] (1, 0) -- (2, -1); 
		\draw[-, black] (3, 0) -- (2, 1); 
		\draw[-, black, dashed] (3, 0) -- (2, -1); 
		\draw[-, black] (2, 1) -- (2, -1);

		\filldraw[fill = white!50!gray] (1, 0) circle[radius = .05];
		\filldraw[fill = white!50!gray] (3, 0) circle[radius = .05];
		\filldraw[fill = white] (2, 1) circle[radius = .05];
		\filldraw[fill = white!50!gray] (2, -1) circle[radius = .05];

		\filldraw[fill = black] (2.5, 1.5) circle[radius = 0] node[scale = .8, above = 0]{$n_1$};
		\filldraw[fill = black] (2, 1.5) circle[radius = 0] node[scale = .8, above = 0]{$n_2$};
		\filldraw[fill = black] (1.5, 1.5) circle[radius = 0] node[scale = .8, above = 0]{$n_3$};
		\filldraw[fill = black] (.5, .5) circle[radius = 0] node[scale = .8, left = 0]{$k_4$};
		\filldraw[fill = black] (.5, 0) circle[radius = 0] node[scale = .8, left = 0]{$k_5$};
		\filldraw[fill = black] (1.5, -1.5) circle[radius = 0] node[scale = .8, left = 0]{$k_6$};

		\filldraw[fill = black] (.5, -.5) circle[radius = 0] node[scale = .8, left = 0]{};
		\filldraw[fill = black] (2, -1.5) circle[radius = 0] node[scale = .8, below = 0]{};
		\filldraw[fill = black] (3.5, -.5) circle[radius = 0] node[scale = .8, right = 0]{};
		\filldraw[fill = black] (2.5, -1.5) circle[radius = 0] node[scale = .6, right = 0]{$k_4 + n_1 + n_2$};
		\filldraw[fill = black] (3.5, 0) circle[radius = 0] node[scale = .6, right = 0]{$k_5 - n_1 + n_3$};
		\filldraw[fill = black] (3.5, .5) circle[radius = 0] node[scale = .6, right = 0]{$k_6 - n_2 - n_3$};

		\filldraw[fill = black] (1.5, .5) circle[radius = 0] node[scale = .8, left = 3]{$n_1'$};
		\filldraw[fill = black] (2, .375) circle[radius = 0] node[scale = .8, right = -1]{$n_2'$};
		\filldraw[fill = black] (2.5, .5) circle[radius = 0] node[scale = .8, right = 2]{$n_3'$};
		\filldraw[fill = black] (1.5, -.5) circle[radius = 0] node[scale = .6, left = 3]{$k_4 + n_1'$};
		\filldraw[fill = black] (1.625, 0) circle[radius = 0] node[scale = .6, below = -1]{$k_5 - n_1'$};
		\filldraw[fill = black] (2.5, -.5) circle[radius = 0] node[scale = .6, right = 2]{$k_6 - n_2'$};

		\draw[->, black, dashed] (8, 1) -- (8.5, 1.5);
		\draw[-, black] (8, 1.5) -- (8, 1);
		\draw[-, black, dashed] (7.5, 1.5) -- (8, 1);
		\draw[-, black] (6.5, .5) -- (7, 0);
		\draw[-, black, dashed] (6.5, 0) -- (7, 0);
		\draw[->, black, dotted] (8, -1) -- (7.5, -1.5);
		
		\draw[-, black, dashed] (6.5, -.5) -- (7, 0);
		\draw[->, black, dotted] (8, -1) -- (8, -1.5);
		\draw[->, black, dashed] (9, 0) -- (9.5, -.5);
		\draw[->, black, dotted] (8, -1) -- (8.5, -1.5);
		\draw[->, black, dashed] (9, 0) -- (9.5, 0);
		\draw[-, black] (9.5, .5) -- (9, 0);

		\draw[-, black, dashed] (7, 0) -- (9, 0); 
		\draw[-, black, dashed] (7, 0) -- (8, 1); 
		\draw[-, black, dotted] (7, 0) -- (8, -1); 
		\draw[-, black, dashed] (9, 0) -- (8, 1); 
		\draw[-, black, dotted] (9, 0) -- (8, -1); 
		\draw[-, black, dotted] (8, 1) -- (8, -1);

		\filldraw[fill = white!50!gray] (7, 0) circle[radius = .05];
		\filldraw[fill = white!50!gray] (9, 0) circle[radius = .05];
		\filldraw[fill = white!50!gray] (8, 1) circle[radius = .05];
		\filldraw[fill = white] (8, -1) circle[radius = .05];

		\filldraw[fill = black] (9.5, .5) circle[radius = 0] node[scale = .8, right = 0]{$n_1$};
		\filldraw[fill = black] (8, 1.5) circle[radius = 0] node[scale = .8, above = 0]{$n_2$};
		\filldraw[fill = black] (6.5, .5) circle[radius = 0] node[scale = .8, left = 0]{$n_3$};
		\filldraw[fill = black] (7.5, 1.5) circle[radius = 0] node[scale = .8, above = 0]{$k_4$};
		\filldraw[fill = black] (6.5, 0) circle[radius = 0] node[scale = .8, left = 0]{$k_5$};
		\filldraw[fill = black] (6.5, -.5) circle[radius = 0] node[scale = .8, left = 0]{$k_6$};

		\filldraw[fill = black] (7.5, -1.5) circle[radius = 0] node[scale = .8, left = 0]{};
		\filldraw[fill = black] (8, -1.5) circle[radius = 0] node[scale = .8, below = 0]{};
		\filldraw[fill = black] (8.5, -1.5) circle[radius = 0] node[scale = .8, right = 0]{};
		\filldraw[fill = black] (9.5, -.5) circle[radius = 0] node[scale = .6, right = 0]{$k_4 + n_1 + n_2$};
		\filldraw[fill = black] (9.5, 0) circle[radius = 0] node[scale = .6, right = 0]{$k_5 - n_1 + n_3$};
		\filldraw[fill = black] (8.5, 1.5) circle[radius = 0] node[scale = .6, right = 0]{$k_6 - n_2 - n_3$};
		
		\filldraw[fill = black] (8.5, -.5) circle[radius = 0] node[scale = .8, right = 1]{};
		\filldraw[fill = black] (8, .375) circle[radius = 0] node[scale = .8, right = -1]{};
		\filldraw[fill = black] (7.5, -.5) circle[radius = 0] node[scale = .8, left = 2]{};
		\filldraw[fill = black] (8.5, .5) circle[radius = 0] node[scale = .6, right = 2]{$k_4 + n_2$};
		\filldraw[fill = black] (7.625, 0) circle[radius = 0] node[scale = .6, below = -1]{$k_5 + n_3$};
		\filldraw[fill = black] (7.5, .5) circle[radius = 0] node[scale = .6, left = 3]{$k_6 - n_3$};

		\end{tikzpicture}
		
	\end{center}

	\caption{\label{vertex3} The attachment of shaded vertices to a vertex is shown above and to the left; applying the tetrahedron equation yields a frozen vertex, as shown to the right.}

\end{figure}

Due to the fact that $R_{000}^{000} = 1$ (which follows from \eqref{rn1n2n3}), the following definition now serves as the analog of \Cref{stochastic1} and \Cref{wschiweights}.

\begin{definition}
	
	\label{sn1n2n3definitionk1k2k3} 
	
	For any nonnegative integers $n_1, n_2, n_3, n_1', n_2', n_3', k_4, k_5, k_6$, set $a_4', a_5', a_6', a_4'', a_5'', a_6''$ as in \eqref{n4n5n6} and $b_4', b_5', b_6'$ as in \eqref{b4b5b6}. Then, define  
	\begin{flalign}
	\label{skweight} 
	S_{n_1 n_2 n_3}^{n_1' n_2' n_3'} (k_4, k_5, k_6) = R_{n_1 n_2 n_3}^{n_1' n_2' n_3'} C_{n_1 n_2 n_3}^{n_1' n_2' n_3'} (k_4, k_5, k_6),
	\end{flalign}
	
	\noindent where
	\begin{flalign} 
	\label{kccorrection} 
	C_{n_1 n_2 n_3}^{n_1' n_2' n_3'} (k_4, k_5, k_6) = \displaystyle\frac{ R_{n_1' k_4 k_5}^{0 b_4' b_5'} R_{n_2' b_4' k_6}^{0 a_4'' b_6'} R_{n_3' b_5' b_6'}^{0 a_5'' a_6''}}{R_{n_3 k_5 k_6}^{0 a_5' a_6'} R_{n_2 k_4 a_6'}^{0 a_4' a_6''} R_{n_1 a_4' a_5'}^{0 a_4'' a_5''}}.
	\end{flalign}

\end{definition}  

As in \Cref{SixVertexEquation} and \Cref{DefinitionWeights}, we can interpret \Cref{sn1n2n3definitionk1k2k3} diagrammatically, by attaching three shaded vertices $u_1, u_2, u_3$ to our initial (unshaded) vertex $u$; this is depicted on the left side of \Cref{vertex3}. Applying the tetrahedron ``pushes'' these shaded vertices past $u$, after which the arrow configuration of $u$ becomes frozen; this is depicted on the right side of \Cref{vertex3}. Now the weight $S_{n_1 n_2 n_3}^{n_1' n_2' n_3'}$ is the weight of the left diagram in \Cref{vertex3} divided by the weight of the right one. 

We will see that the $S$ weights defined above are stochastic and satisfy (a dynamical variant of) the Yang-Baxter equation in \Cref{stochastic3s} and \Cref{3equationdimensiondynamical}, respectively. However, let us first evaluate them explicitly. To that end, we have the following proposition. 

\begin{prop} 
	
\label{ssk}

For any $n_1, n_2, n_3, n_1', n_2', n_3', k_4, k_5, k_6 \in \mathbb{Z}_{\ge 0}$, we have that 
\begin{flalign*}
S_{n_1 n_2 n_3}^{n_1' n_2' n_3'} (k_4, k_5, k_6) = S_{n_1 n_2 n_3}^{n_1' n_2' n_3'} (q^{2 k_5 + 2}),
\end{flalign*}

\noindent where $S_{n_1 n_2 n_3}^{n_1' n_2' n_3'} (v)$ is given by \eqref{sn1n2n3vweight}, for any $v \in \mathbb{C}$. 
\end{prop}

\begin{rem} 
	
\label{sk5} 

Analogous to what was observed by \Cref{cproductzr} in the higher rank Yang-Baxter setting, Proposition \ref{ssk} implies that $S_{n_1 n_2 n_3}^{n_1' n_2' n_3'} (k_4, k_5, k_6)$ is independent of $k_4$ and $k_6$, that is, the three parameters initially defining this quantity reduce to one that governs it. This fact will be useful in deriving the dynamical tetrahedron equation, given by \Cref{3equationdimensiondynamical} below. 
	
\end{rem}

\begin{proof}[Proof of \Cref{ssk}]
	
We first evaluate the stochastic corrections $C_{n_1 n_2 n_3}^{n_1' n_2' n_3'} (k_1, k_2, k_3)$ given by \eqref{kccorrection}. Observe that all of the $R_{n_1 n_2 n_3}^{n_1' n_2' n_3'}$ terms appearing in the right side of that expression have $n_1' = 0$; thus, we will begin by deriving a factored form for weights of the type $R_{n_1 n_2 n_3}^{0 n_2' n_3'}$. 

To that end, if $n_1, n_2, n_3, n_1', n_2', n_3'$ are nonnegative integers satisfying $n_1 + n_2 = n_1' + n_2'$ and $n_2 + n_3 = n_2' + n_3'$, then \eqref{rn1n2n3} yields 
\begin{flalign*}
R_{n_1 n_2 n_3}^{n_1' n_2' n_3'} = q^{n_2 (n_2 + 1) - (n_2 - n_1') (n_2 - n_3')} \displaystyle\sum_{k = 0}^{n_2} \displaystyle\frac{ (q^{-2 n_1'}; q^2)_{n_2 - k} (q^{2n_1 + 2}; q^2)_k}{(q^2; q^2)_{n_2 - k} (q^2; q^2)_k } q^{-2k (n_1' + n_3 + 1)}.
\end{flalign*}

\noindent In particular, if $n_1' = 0$, then $(q^{-2 n_1'}; q^2)_{n_2 - k} = \textbf{1}_{k = n_2}$, and so the sum defining $R_{n_1 n_2 n_3}^{n_1' n_2' n_3'}$ is only supported on the $k = n_2$ term. Thus, 
\begin{flalign}
\label{r0n2n3}
R_{n_1 n_2 n_3}^{0 n_2' n_3'} & = q^{n_2 (n_3' - 2 n_3 - n_1' - 1) }  \displaystyle\frac{ (q^{2n_1 + 2}; q)_{n_2}}{(q^2; q^2)_{n_2} } = q^{- n_2 (n_1 + n_3 + 1)}  \displaystyle\frac{ (q^2; q^2)_{n_1 + n_2}}{(q^2; q^2)_{n_1} (q^2; q^2)_{n_2}},
\end{flalign}

\noindent where we have used the first identity in \eqref{elliptichypergeometricidentities1} and the fact that $n_3' - n_1' = n_3 - n_1$.

Recalling the definitions of $a_4', a_5', a_6'$ and $b_4', b_5', b_6'$ from \eqref{n4n5n6} and \eqref{b4b5b6}, respectively, and then inserting \eqref{r0n2n3} into \eqref{kccorrection} yields 
\begin{flalign}
\label{cn1n2n3}
\begin{aligned}
C_{n_1 n_2 n_3}^{n_1' n_2' n_3'} (k_1, k_2, k_3) & =   \displaystyle\frac{(q^2; q^2)_{n_1} 	(q^2; q^2)_{n_2} (q^2; q^2)_{n_3}}{(q^2; q^2)_{n_1'} (q^2; q^2)_{n_2'} (q^2; q^2)_{n_3'} }   \displaystyle\frac{ (q^2; q^2)_{k_5 - n_1' + n_3'}}{(q^2; q^2)_{k_5 - n_1'}} \displaystyle\frac{ (q^2; q^2)_{k_5}}{ (q^2; q^2)_{n_3 + k_5}} \displaystyle\frac{ (q^2; q^2)_{k_4 + n_1' + n_2'}}{ (q^2; q^2)_{k_4 + n_1 + n_2}} \\
& \qquad \times q^{ k_5 (n_2 + n_3 + n_2' - n_3') + k_4 (n_1 + n_2 - n_1' - n_2') + n_1 n_2 + n_2 n_3 + n_1' n_3' - 2n_1' n_2' + n_2} \\
& =  \displaystyle\frac{(q^2; q^2)_{n_1} (q^2; q^2)_{n_2} (q^2; q^2)_{n_3}}{(q^2; q^2)_{n_1'} (q^2; q^2)_{n_2'} (q^2; q^2)_{n_3'}} \displaystyle\frac{ (q^{2n_5 - 2n_1' + 2}; q^2)_{n_3'}}{ (q^{2n_5 + 2}; q^2)_{n_3} } \\
& \qquad \times q^{n_2 + 2 n_2' n_5 + n_1 n_2 + n_2 n_3 + n_1' n_3' - 2 n_1' n_2'}, 
\end{aligned} 
\end{flalign} 

\noindent where in the second equality we used the first identity in \eqref{elliptichypergeometricidentities1} and the facts that $n_1 + n_2 = n_1' + n_2'$ and $n_2 + n_3 = n_2' + n_3'$. Now the proposition follows from inserting \eqref{cn1n2n3} into \eqref{skweight} and using \eqref{rn1n2n3}. 
\end{proof}

\subsection{Properties of the Stochasticized Weights} 

\label{EquationStochastic3}

In this section we establish two properties of the stochasticized weights $S_{n_1 n_2 n_3}^{n_1' n_2' n_3'} (v)$ from \eqref{sn1n2n3vweight}, namely that they are stochastic (\Cref{stochastic3s}) and satisfy a dynamical version of the tetrahedron equation (\Cref{3equationdimensiondynamical}). We begin with the former.

\begin{prop}
	
\label{stochastic3s}

Fix $v \in \mathbb{C}$ and $(n_1, n_2, n_3) \in \mathbb{Z}_{\ge 0}^3$. Then, \eqref{stochasticsumsv} holds. Furthermore, if $q, v \in (0, 1)$ and $v \le q^{2 n_1'}$, then $S_{n_1 n_2 n_3}^{n_1' n_2' n_3'} (v) \ge 0$.
\end{prop}

\begin{proof}
The second part of the proposition follows from the second statement of \Cref{requation3} and the fact that 
\begin{flalign}
\label{svr} 
\displaystyle\frac{S_{n_1 n_2 n_3}^{n_1' n_2' n_3'} (v)}{R_{n_1 n_2 n_3}^{n_1' n_2' n_3'}} = \displaystyle\frac{(q^2; q^2)_{n_1} (q^2; q^2)_{n_2} (q^2; q^2)_{n_3}}{(q^2; q^2)_{n_1'} (q^2; q^2)_{n_2'} (q^2; q^2)_{n_3'}} \displaystyle\frac{(q^{- 2 n_1'} v; q^2)_{n_3'}}{(v; q)_{n_3}} v^{n_2'} q^{n_2 (n_1 + n_3 + 1) + n_1' n_3' - 2 n_2' (n_1' + 1)},
\end{flalign}

\noindent which is nonnegative for $q, v \in (0, 1)$ with $v \le q^{2n_1'}$.

The proof of \eqref{stochasticsumsv} will follow from the tetrahedron equation \eqref{equation31} and an analytic continuation. More specifically, first observe that the number of nonzero terms on the left side of \eqref{stochasticsumsv} is bounded (by a quantity that depends only on $n_1$, $n_2$, and $n_3)$. Each of these terms is a rational function in $v$, whose numerator and denominator are both of bounded degrees (again, in a way that only depends on $n_1$, $n_2$, and $n_3$). Therefore, it suffices to verify \eqref{stochasticsumsv} for infinitely many values of $v$. 

Thus, let $k_4, k_5, k_6$ be arbitrary positive integers such that the quantities $a_4', a_5', a_6', a_4'', a_5'', a_6''$ and $b_4', b_5', b_6'$ from \eqref{n4n5n6} and \eqref{b4b5b6}, respectively, are all nonnegative; we will establish \eqref{stochasticsumsv} in the case $v = q^{2 k_5 + 2}$. 

To that end, we apply \eqref{equation31} where the $(n_1, n_2, n_3)$ there coincides with the $(n_1, n_2, n_3)$ here; the $(n_4, n_5, n_6)$ there coincides with the $(k_4, k_5, k_6)$ here; where $n_1'', n_2'', n_3''$ there are each equal to $0$ here; and where the $(n_4'', n_5'', n_6'')$ there coincides with the $(a_4'', a_5'', a_6'')$ here. As explained directly above \Cref{sn1n2n3definitionk1k2k3}, this choice of parameters allows for only one nonzero term on the right side of \eqref{equation31}, corresponding to when $n_1', n_2', n_3'$ there are each equal to $0$ here and the $(n_4', n_5', n_6')$ there coincides with  the $(a_4', a_5', a_6')$ here. The left side of \eqref{equation31} might have several nonzero terms but, again as explained directly above \Cref{sn1n2n3definitionk1k2k3}, if we fix the $n_1', n_2', n_3'$ indices of a nonzero summand there then the $(n_4', n_5', n_6')$ there must coincide with the $(b_4', b_5', b_6')$ here. 

Thus, \eqref{equation31} states that 
\begin{flalign}
\label{stochasticsumsv1}
 \displaystyle\sum_{\textbf{n}'} R_{n_1 n_2 n_3}^{n_1' n_2' n_3'} R_{n_1' k_4 k_5}^{0 b_4' b_5'} R_{n_3' b_5' b_6'}^{0 a_5'' a_6''} R_{n_2' b_4' k_6}^{0 a_4'' b_6'}  = R_{n_3 k_5 k_6}^{0 a_5' a_6'} R_{n_2 k_4 a_6'}^{0 a_4' a_6''}  R_{n_1 a_4' a_5'}^{0 a_4'' a_5''},
\end{flalign}

\noindent where the sum is over $\textbf{n}' = (n_1', n_2', n_3') \in \mathbb{Z}_{\ge 0}^3$ and where we have used the fact that $R_{000}^{000} = 1$. Now \eqref{stochasticsumsv} in the case $v = q^{2k_5 + 2}$ follows from \eqref{skweight}, \eqref{kccorrection}, \eqref{sn1n2n3vweight}, and \eqref{stochasticsumsv1}. 
\end{proof}

In addition to stochasticity, we would also like the $S$ weights to satisfy a (dynamical variant of) the tetrahedron equation. Recall from \Cref{dynamicalsixvertexequationparameter} and \Cref{consistencyvertexequation} that the proof of such an equation requires the consistency condition for the attachment of the shaded vertices. In our previous two-dimensional examples, we explained a diagrammatic way of guaranteeing this consistency, by moving a stochasticization curve through our domain, as depicted in \Cref{arrowstochasticvertices}. 

We do not know of an analogous diagram in the present three-dimensional situation, but we can still algebraically search for a choice of dynamical parameters that guarantees the required cancellations (as we did in \Cref{dynamicalsixvertexequationparameter}). There indeed happens to be one, which is depicted in \Cref{3equationdimension}. This gives rise to the dynamical tetrahedron equation for the $S$ weights provided \eqref{dimensiondynamical}.

\begin{thm} 
	
	\label{3equationdimensiondynamical}
	
Fix complex numbers $v$ and $w$, as well as nonnegative integers $n_1, n_2, n_3, n_4, n_5, n_6$ and $n_1'', n_2'', n_3'', n_4'', n_5'', n_6''$. Then, \eqref{dimensiondynamical} holds.

\end{thm} 

\begin{proof}[Proof (Outline)]
	
	As in the proofs of \Cref{sdynamicalequation1} and \Cref{relationwstochastic}, one first shows that  
	\begin{flalign*}
	\mathcal{X} = \displaystyle\frac{S_{n_1 n_2 n_3}^{n_1' n_2' n_3'} (t)}{R_{n_1 n_2 n_3}^{n_1' n_2' n_3'} } \displaystyle\frac{S_{n_1' n_4 n_5}^{n_1'' n_4' n_5'} (w)}{R_{n_1' n_4 n_5}^{n_1'' n_4' n_5'}} \displaystyle\frac{ S_{n_3' n_5' n_6'}^{n_3'' n_5'' n_6''} (x)}{R_{n_3' n_5' n_6'}^{n_3'' n_5'' n_6''}} \displaystyle\frac{ S_{n_2' n_4' n_6}^{n_2'' n_4'' n_6'} (v) }{  R_{n_2' n_4' n_6}^{n_2'' n_4'' n_6'}},
	\end{flalign*}
	
	\noindent and 
	\begin{flalign*}
	\mathcal{Y} = \displaystyle\frac{S_{n_3 n_5 n_6}^{n_3' n_5' n_6'} (v')}{R_{n_3 n_5 n_6}^{n_3' n_5' n_6'}} \displaystyle\frac{ S_{n_2 n_4 n_6'}^{n_2' n_4' n_6''} (y)}{R_{n_2 n_4 n_6'}^{n_2' n_4' n_6''}} \displaystyle\frac{ S_{n_1 n_4' n_5'}^{n_1' n_4'' n_5''} (z)}{R_{n_1 n_4' n_5'}^{n_1' n_4'' n_5''}} \displaystyle\frac{ S_{n_1' n_2' n_3'}^{n_1'' n_2'' n_3''} (w')}{R_{n_1' n_2' n_3'}^{n_1'' n_2'' n_3''}}, 
	\end{flalign*}
	
	\noindent are equal and independent of $\textbf{n}' = (n_1', n_2', n_3', n_4', n_5', n_6')$ if 
	\begin{flalign} 
	\label{txyzvw}
	t = q^{2n_5} w; \quad x = q^{-2 n_2''} v; \quad y = q^{-2n_3'} v'; \quad z = q^{2n_3'} w'; \quad v' = v; \quad w' = w. 
	\end{flalign}
	
	\noindent This can be verified directly, since the stochastic corrections $\frac{S}{R}$ appearing in the definitions of $\mathcal{X}$ and $\mathcal{Y}$ are explicitly given by \eqref{svr}. Then, \eqref{dimensiondynamical} would follow as a consequence of this equality and independence, and of \eqref{equation31}. 
	
	However, instead of confirming this here, let us outline an alternative method that enables one to derive the relations \eqref{txyzvw} between the dynamical parameters appearing in \eqref{dimensiondynamical}, without using the explicit form \eqref{svr} for the stochastic correction $C(u) = \frac{S (u)}{R}$. To that end, we will first use \Cref{ssk}, \eqref{skweight}, and \eqref{kccorrection} to express these stochastic corrections as products of weights of the form $R_{abc}^{0ef}$. 
	
	So, let $(j_4, j_5, j_6)$; $(k_4, k_5, k_6)$; $(l_4, l_5, l_6)$; $(m_4, m_5, m_6)$ be four triples of positive integers, to be fixed later. We view these four triples as analogs of the parameters $(k_4, k_5, k_6)$ from \Cref{sn1n2n3definitionk1k2k3}, for each of the four stochastic corrections defining $\mathcal{X}$. Specifically, the $j$-triple corresponds to the correction $C_{n_1 n_2 n_3}^{n_1' n_2' n_3'} (t)$; the $k$-triple to $C_{n_1' n_4 n_5}^{n_1'' n_4' n_5'} (w)$; the $l$-triple to $C_{n_3' n_5' n_6'}^{n_3'' n_5'' n_6''} (x)$; and the $m$-triple to $C_{n_2' n_4' n_6}^{n_2'' n_4'' n_6'} (v)$. By \Cref{ssk}, the dependence of $\mathcal{X}$ on these twelve parameters is given by 
	\begin{flalign}
	\label{twxv} 
	t = q^{2j_5 + 2}; \qquad w = q^{2k_5 + 2}; \qquad x = q^{2l_5 + 2}; \qquad v = q^{2 m_5 + 2};
	\end{flalign}
	
	 \noindent in particular, $\mathcal{X}$ is independent of the eight parameters $j_i, k_i, m_i, l_i$ for $i \in \{ 4, 6 \}$, due to \Cref{sk5}. 
	
	 Now for each $i \in \{ 4, 5, 6 \}$ let $a_i'$, $a_i''$ and $b_i$ be the parameters \eqref{n4n5n6} and \eqref{b4b5b6} associated with the correction $S_{n_1 n_2 n_3}^{n_1' n_2' n_3'} (t)$. Similarly, let $c_i'$, $c_i''$ and $d_i'$ denote the analogous parameters associated with the vertex with parameter $C_{n_1' n_4 n_5}^{n_1'' n_4' n_5'} (w)$; let $e_i'$, $e_i''$ and $f_i'$ denote those associated with $C_{n_3' n_5' n_6'}^{n_3'' n_5'' n_6''} (x)$; and let $g_i'$, $g_i''$ and $h_i$ be those associated with $C_{n_2' n_4' n_6}^{n_2'' n_4'' n_6'} (v)$. 
	 
	 Then, \eqref{kccorrection} implies that 
	\begin{flalign}
	\label{xquantity} 
	\begin{aligned}
	\mathcal{X} =  & \displaystyle\frac{R_{n_1' j_4 j_5}^{0 b_4' b_5'} R_{n_2' b_4' j_6}^{0 a_4'' b_6'} R_{n_3' b_5' b_6'}^{0 a_5'' a_6''}}{R_{n_3 j_5 j_6}^{0 a_5' a_6'} R_{n_2 j_4 a_6'}^{0 a_4' a_6''} R_{n_1 a_4' a_5'}^{0 a_4'' a_5''} } \displaystyle\frac{R_{n_1'' k_4 k_5}^{0 d_4' d_5'} R_{n_4' d_4' k_6}^{0 c_4'' d_6'} R_{n_5' d_5' d_6'}^{0 c_5'' c_6''}}{R_{n_5 k_5 k_6}^{0 c_5' c_6'} R_{n_4 k_4 c_6'}^{0 c_4' c_6''} R_{n_1' c_4' c_5'}^{0 c_4'' c_5''} } \\
	& \qquad \times \displaystyle\frac{R_{n_3'' l_4 l_5}^{0 f_4' f_5'} R_{n_5'' f_4' l_6}^{0 e_4'' f_6'} R_{n_6'' f_5' f_6'}^{0 e_5'' e_6''}}{R_{n_6' l_5 l_6}^{0 e_5' e_6'} R_{n_5' l_4 e_6'}^{0 e_4' e_6''} R_{n_3' e_4' e_5'}^{0 e_4'' e_5''} } 	\displaystyle\frac{R_{n_2'' m_4 m_5}^{0 h_4' h_5'} R_{n_4'' h_4' m_6}^{0 g_4'' h_6'} R_{n_6' h_5' h_6'}^{0 g_5'' g_6''}}{R_{n_6 m_5 m_6}^{0 g_5' g_6'} R_{n_4' m_4 g_6'}^{0 g_4' g_6''} R_{n_2' g_4' g_5'}^{0 g_4'' g_5''} }.
	\end{aligned}
	\end{flalign}
	
	In order for $\mathcal{X}$ to be independent of $\textbf{n}'$, we would like for any $R$ term in the numerator of the right side of \eqref{xquantity} containing some $n_i'$ as one of its indices to also appear in the denominator; this is similar to what was described in \Cref{dynamicalsixvertexequationparameter} and \Cref{consistencyvertexequation}. For instance, if $i = 1$, this would amount to imposing that $R_{n_1' j_4 j_5}^{0 b_4' b_5'}$ (which appears in the numerator of the right side of \eqref{xquantity}) be equal to $R_{n_1' c_4' c_5'}^{0 c_4'' c_5''}$ (which appears in the denominator). One way of guaranteeing this is by stipulating that the ordered sets $(n_1', j_4, j_5, 0, b_4', b_5')$ and $(n_1', c_4', c_5', 0, c_4'', c_5'')$ be equal, which would be the analog of the consistency condition from \Cref{wchiconditions}. This imposes the relations $(j_4, j_5) = (c_4', c_5')$ and $(b_4', b_5') = (c_4'', c_5'')$, which can be rewritten only in terms of the $n_i, n_i', n_i'', j_i, k_i, l_i, m_i$ using \eqref{n4n5n6} and \eqref{b4b5b6}. 
	
	Applying similar reasoning five more times (for the other terms on the right side of \eqref{xquantity} including some $n_i'$ as an index), we obtain a number of relations that amount to
	\begin{flalign}
	\label{jkm456}
	\begin{aligned}
	&  j_5 = k_5 + n_5; \qquad \quad m_5 - n_2'' = l_5; \\
	 k_5 - n_1'' = l_4; 	\quad j_6 = m_5 + n_6; \quad & j_4 = k_4 + n_4; \quad k_4 + n_1'' = m_4; \quad k_6 = m_6 - n_6;  \quad m_6 - n_4'' = l_6.
	\end{aligned}
	\end{flalign}
	
	\noindent Under the assumptions \eqref{jkm456}, we have that
	\begin{flalign*}
	\mathcal{X} =  & \displaystyle\frac{R_{n_1'' k_4 k_5}^{0 d_4' d_5'} R_{n_2'' m_4 m_5}^{0 h_4' h_5'} R_{n_3'' l_4 l_5}^{0 f_4' f_5'}  R_{n_4'' h_4' m_6}^{0 g_4'' h_6'} R_{n_5'' f_4' l_6}^{0 e_4'' f_6'} R_{n_6'' f_5' f_6'}^{0 e_5'' e_6''}}{R_{n_1 a_4' a_5'}^{0 a_4'' a_5''} R_{n_2 j_4 a_6'}^{0 a_4' a_6''} R_{n_3 j_5 j_6}^{0 a_5' a_6'} R_{n_4 k_4 c_6'}^{0 c_4' c_6''} R_{n_5 k_5 k_6}^{0 c_5' c_6'} R_{n_6 m_5 m_6}^{0 g_5' g_6'} },
	\end{flalign*}
	
	\noindent which is independent of $\textbf{n}'$ since all of the indices appearing there are (which can quickly be verified using \eqref{n4n5n6} and \eqref{b4b5b6}) if the $n_i$ and $n_i''$ are fixed for each $1 \le i \le 6$. 
	
	Now, let us fix $j_5, k_5, m_5, l_5$ satisfying the first two relations in \eqref{jkm456} and let the other $j_i, k_i, m_i, l_i$ be arbitrary positive integers satisfying the latter six relations there. In view of \eqref{twxv}, the first two identities in \eqref{jkm456} imply that $t = q^{n_5} w$ and $x = q^{-n_2''} v$, which comprise the first two constraints in \eqref{txyzvw}. The latter six relations in \eqref{jkm456} do not affect $\mathcal{X}$ since, as mentioned above, it is independent of the $j_i, k_i, l_i, m_i$ for $i \in \{ 4, 6 \}$. 
	
	Thus, the first two identities in \eqref{txyzvw} imply that $\mathcal{X}$ is independent of $\textbf{n}'$. Similarly, the third and fourth relations there can be shown to imply that $\mathcal{Y}$ is independent of $\textbf{n}'$, and the last two imply that $\mathcal{X} = \mathcal{Y}$. However, we will omit the latter two verifications since they are analogous to what was outlined above. 
\end{proof}

Similar to what was considered in \Cref{ellipticequationstochastic} and \Cref{sweightsrank2}, one might wonder whether it is possible to find a stochastic solution to the non-dynamical tetrahedron equation. We will address this in the next section.

\subsection{A Non-Dynamical Degeneration} 

\label{EquationDimension2} 

In this section we propose a degeneration of our stochasticized weights that satisfies the original (non-dynamical) tetrahedron equation. This will proceed by taking the limit as $q$ tends to $1$ of the $S_{n_1 n_2 n_3}^{n_1' n_2' n_3'} (v)$ stochasticized weights from \eqref{sn1n2n3vweight}.

To that end, we have the following proposition. 

\begin{prop}
	
	\label{q1sv} 
	
	Fix $v \in \mathbb{C}$ and $n_1, n_2, n_3, n_1', n_2', n_3' \in \mathbb{Z}_{\ge 0}$. Then, 
	\begin{flalign}
	\label{ts1}
	\displaystyle\lim_{q \rightarrow 1}	S_{n_1 n_2 n_3}^{n_1' n_2' n_3'} (v) = T_{n_1 n_2 n_3}^{n_1' n_2' n_3'} (v),
	\end{flalign}
	
	\noindent where, for any $v \in \mathbb{C}$ and $n_1, n_2, n_3, n_1', n_2', n_3' \in \mathbb{Z}_{\ge 0}$, 
	\begin{flalign} 
	\label{tn1n2n3vweight}
	\begin{aligned}
	T_{n_1 n_2 n_3}^{n_1' n_2' n_3'} (v) & = v^{n_2'} (1 - v)^{n_2 - n_2'} \binom{n_2}{n_2'} \textbf{\emph{1}}_{n_2' \le n_2} \textbf{\emph{1}}_{n_1 + n_2 = n_1' + n_2'} \textbf{\emph{1}}_{n_2 + n_3 = n_2' + n_3'}.
	\end{aligned}
	\end{flalign} 
	
	\noindent Moreover, for any fixed $v \in \mathbb{C}$ and $n_1, n_2, n_3 \in \mathbb{Z}_{\ge 0}$, we have that 
	\begin{flalign}
	\label{tstochastic}
	\displaystyle\sum_{\textbf{\emph{n}}'} T_{n_1 n_2 n_3}^{n_1' n_2' n_3'} (v) = 1,
	\end{flalign}
	
	\noindent where the sum is over all $\textbf{\emph{n}}' = (n_1', n_2', n_3') \in \mathbb{Z}_{\ge 0}^3$. Furthermore, for any fixed $v, w \in \mathbb{C}$ and $n_1, n_2, n_3, n_4, n_5, n_6, n_1'', n_2'', n_3'', n_4'', n_5'', n_6'' \in \mathbb{Z}_{\ge 0}$, they also satisfy the (non-dynamical) tetrahedron equation
	\begin{flalign}
	\label{dimensiondynamicalt}
	\begin{aligned} 
	& \displaystyle\sum_{\textbf{\emph{n}}'} T_{n_1 n_2 n_3}^{n_1' n_2' n_3'} (w) T_{n_1' n_4 n_5}^{n_1'' n_4' n_5'} (w) T_{n_3' n_5' n_6'}^{n_3'' n_5'' n_6''} (v) T_{n_2' n_4' n_6}^{n_2'' n_4'' n_6'} (v) \\
	& \qquad = \displaystyle\sum_{\textbf{\emph{n}}'} T_{n_3 n_5 n_6}^{n_3' n_5' n_6'} (v) T_{n_2 n_4 n_6'}^{n_2' n_4' n_6''} (v) T_{n_1 n_4' n_5'}^{n_1' n_4'' n_5''} (w) T_{n_1' n_2' n_3'}^{n_1'' n_2'' n_3''} (w), 
	\end{aligned}
	\end{flalign} 
	
	\noindent where the sum is over all nonnegative integer sets $\textbf{\emph{n}}' = (n_1', n_2', n_3', n_4', n_5', n_6')$. 
	
	Additionally, if $v \in (0, 1)$, then $T_{n_1 n_2 n_3}^{n_1' n_2' n_3'} \ge 0$ for each $n_1, n_2, n_3, n_1', n_2', n_3' \in \mathbb{Z}_{\ge 0}$. 
\end{prop}

\begin{rem}
	
\label{arrowtprobability} 

Observe that the weight $T_{n_1 n_2 n_3}^{n_1' n_2' n_3'} (v)$ from \eqref{tn1n2n3vweight} has the following probabilistic interpretation. Each arrow entering a vertex $u$ through the $n_1$ (or $n_3$) direction deterministically exits $u$ through the $n_1'$ (or $n_3'$) direction. Furthermore, if an arrow enters $u$ through the $n_2$ direction, then a Bernoulli $0-1$ random variable $\chi$ with $\mathbb{P}[\chi = 1] = v$ is sampled. If $\chi = 1$, then the arrow exits $u$ through the $n_2'$ direction; if $\chi = 0$, then two copies of the arrow are made, which exit $u$ through the $n_1'$ and $n_3'$ directions. 

\end{rem}

\begin{proof}[Proof of \Cref{q1sv}]

Observe that directly taking the limit as $q$ tends to $1$ on the right side of \eqref{sn1n2n3vweight} poses an issue, in that the prefactor $(q^2; q^2)_{n_1} (q^2; q^2)_{n_2} (q^2; q^2)_{n_3} (q^2; q^2)_{n_1'}^{-1} (q^2; q^2)_{n_2'}^{-1} (q^2; q^2)_{n_3'}^{-1}$ will behave as a constant multiple of $(1 - q)^{n_2' - n_2}$, which diverges if $n_2 > n_2'$. Therefore, we will first apply \eqref{21identityhypergeometric} with that $q$ given by $q^2$; that $k$ given by $n_2$; that $b$ given by $q^{2n_1 + 2}$; that $c$ given by $q^{2 (n_1' - n_2 + 1)}$; and that $z$ given by $q^{-2n_3}$ to rewrite the ${_2 \varphi_1}$ hypergeometric series on the right side of \eqref{sn1n2n3vweight}. This yields 
\begin{flalign*}
S_{n_1 n_2 n_3}^{n_1' n_2' n_3'} (v) & = q^{n_2 (n_1 + n_3 + n_1' + n_3' + 2) - 2 n_2' (n_1' + 1)}  \displaystyle\frac{(q^2; q^2)_{n_1} (q^{-2n_1'}; q^2)_{n_2} (q^2; q^2)_{n_3}}{(q^2; q^2)_{n_1'} (q^2; q^2)_{n_2'} (q^2; q^2)_{n_3'}} \displaystyle\frac{(q^{-2n_2 - 2n_3}; q^2)_{n_2}}{(q^{2n_1' - 2n_2 + 2}; q^2)_{n_2'}}  \\
& \qquad \times v^{n_2'} \displaystyle\frac{(q^{-2n_1'} v; q)_{n_3'}} {(v; q)_{n_3}} {_2 \varphi_1}  \Bigg( \begin{array}{cc} q^{-2n_2'}, q^{-2n_3} \\ q^{-2 n_2 - 2 n_3} \end{array}\Bigg| q^2, q^{2n_1 + 2}\Bigg) \textbf{1}_{n_1 + n_2 = n_1' + n_2'} \textbf{1}_{n_2 + n_3 = n_2' + n_3'},
\end{flalign*}

\noindent where we have used the second identity in \eqref{elliptichypergeometricidentities1} and the facts that $n_1 + n_2 = n_1' + n_2'$ and $n_2 + n_3 = n_2' + n_3'$. Hence, \eqref{tn1n2n3vweight} follows from the facts that 
\begin{flalign*}
&  \displaystyle\lim_{q \rightarrow 1}  \displaystyle\frac{(q^2; q^2)_{n_1} (q^{-2n_1'}; q^2)_{n_2} (q^2; q^2)_{n_3}}{(q^2; q^2)_{n_1'} (q^2; q^2)_{n_2'} (q^2; q^2)_{n_3'}} \displaystyle\frac{(q^{-2n_2 - 2n_3}; q^2)_{n_2}}{(q^{2n_1' - 2n_2 + 2}; q^2)_{n_2'}}   = \displaystyle\frac{n_1! (-n_1')_{n_2}  n_3!  }{n_1'! n_2'! n_3'!} \displaystyle\frac{(-n_2 - n_3)_{n_2}}{(n_1' - n_2 + 1)_{n_2'}}; \\
&  {_2 \varphi_1}  \Bigg( \begin{array}{cc} q^{-2n_2'}, q^{-2n_3} \\ q^{-2 n_2 - 2 n_3} \end{array}\Bigg| q^2, q^{2n_1 + 2}\Bigg) = {_2 F_1}  \Bigg( \begin{array}{cc} -n_2', -n_3 \\ - n_2 -  n_3 \end{array}\Bigg| 1 \Bigg) = \displaystyle\frac{(-n_2)_{n_2'}}{(-n_2 - n_3)_{n_2'}}; \\
& \displaystyle\frac{n_1! (-n_1')_{n_2}  n_3!  }{n_1'! n_2'! n_3'!} \displaystyle\frac{(-n_2 - n_3)_{n_2}}{(n_1' - n_2 + 1)_{n_2'}} \displaystyle\frac{(-n_2)_{n_2'}}{(-n_2 - n_3)_{n_2'}} = \binom{n_2}{n_2'} \textbf{1}_{n_2' \le n_2},
\end{flalign*}
 
 \noindent which are consequences of the third identity in \eqref{elliptichypergeometricidentities1}, \eqref{221identityhypergeometric}, and the facts that $n_1 + n_2 = n_1' + n_2'$ and $n_2 + n_3 = n_2' + n_3'$. 
 
The remaining results given by \eqref{tstochastic}, \eqref{dimensiondynamicalt}, and the last statement of the proposition follow from combining \eqref{ts1} with \Cref{stochastic3s} and \eqref{dimensiondynamical} (alternatively, each of these can be verified directly, independently of the content of \Cref{EquationStochastic3}). 
\end{proof}

\begin{rem}
	
	Let us briefly outline an alternative way to verify that the $T$ weights from \eqref{tn1n2n3vweight} satisfy \eqref{dimensiondynamicalt}. To that end, consider the two tetrahedra depicted in \Cref{3equationdimension}, and fix the numbers $n_1, n_2, n_3, n_4, n_5, n_6$ of arrows entering each tetrahedron. Under the probabilistic interpretation of the $T$ weights from \Cref{arrowtprobability}, the left and right sides of \eqref{dimensiondynamicalt} denote the probabilities that the numbers of arrows exiting the left and right tetrahedra in \Cref{3equationdimension} are $n_1'', n_2'', n_3'', n_4'', n_5'', n_6''$, respectively.
	
	To equate these two probabilities, we analyze the trajectory of a given arrow entering each tetrahedron. For example, consider an arrow entering through $n_2$ edge. In the left diagram in \Cref{3equationdimension}, this arrow splits into two along the $n_1''$ and $n_5'$ edges with probability $1 - w$, and it continues along the $n_2'$ edge with probability $w$. Conditional on the former event, these arrows deterministically exit through the $n_1''$ and $n_3''$ edges; conditional on the latter, the arrow deterministically exits through the $n_2''$ edge. 
	
	Therefore, the probability that an arrow entering the left tetrahedron through the $n_2$ edge causes an arrow to exit through the $n_2''$ edge is $w$, and the probability that it causes arrows to exit through the $n_1''$ and $n_3''$ edges is $1 - w$. Similar reasoning indicates that the same is true for the right tetrahedron in \Cref{3equationdimension}.
	
	Repeating this, one can show that, for any integer $i \in [1, 6]$ and integer subset $\mathcal{J} \subseteq [1, 6]$, the probability that an arrow entering the left tetrahedron through the $n_i$ vertex causes an arrow to exit through the $n_j''$ vertex for each $j \in \mathcal{J}$ is equal to the probability of the same event for the right tetrahedron. Thus, the probability that the numbers of arrows exiting the left tetrahedron are $n_1'', n_2'', n_3'', n_4'', n_5'', n_6''$ is equal to the probability of the same event for the right tetrahedron. This implies \eqref{dimensiondynamicalt}. 
\end{rem}

\end{document}